\documentclass[11pt]{article}
\usepackage{smile}

\usepackage{fullpage}
\usepackage{url}
\usepackage{lscape}
\usepackage{bigints}
\usepackage{framed}
\usepackage{mdframed}
\usepackage{enumerate}
\usepackage[inline]{enumitem}
\usepackage[T1]{fontenc}
\usepackage{moresize}
\usepackage{bm}
\usepackage{bbm}
\usepackage{dsfont}
\usepackage{amsmath}
\usepackage{amssymb}
\usepackage{amsthm}
\usepackage{amsfonts}
\usepackage{stmaryrd}
\usepackage{array}
\usepackage{mathrsfs}
\usepackage{mathtools} 
\usepackage{extarrows}
\usepackage{stackrel}
\usepackage{relsize,exscale}
\usepackage{scalerel}
\usepackage{colortbl}
\usepackage[nodisplayskipstretch]{setspace}
\usepackage{color}
\usepackage[usenames,dvipsnames]{xcolor}
\usepackage{cancel}
\usepackage{soul}
\usepackage{undertilde}
\usepackage{xfrac}
\usepackage{siunitx}
\usepackage{graphicx}
\usepackage{float}
\usepackage{rotating}
\usepackage{subcaption}
\usepackage{overpic}
\usepackage[all]{xy}
\usepackage{tikz}
\usetikzlibrary{arrows,matrix,positioning,calc,automata,patterns}
\usepackage{booktabs}
\usepackage{dcolumn}
\usepackage{multirow}
\usepackage{diagbox}
\usepackage{tabularx}
\usepackage{verbatim}
\usepackage{listings}
\usepackage[ruled,vlined]{algorithm2e}
\usepackage{fancyvrb}
\usepackage{hyperref}
\usepackage[round]{natbib}
\usepackage{sectsty}

\hypersetup{
    bookmarks=true,         
    unicode=false,          
    pdftoolbar=true,        
    pdfmenubar=true,        
    pdffitwindow=false,     
    pdfstartview={FitH},    
    pdftitle={My title},    
    pdfauthor={Author},     
    pdfsubject={Subject},   
    pdfcreator={Creator},   
    pdfproducer={Producer}, 
    pdfkeywords={key1, key2}, 
    pdfnewwindow=true,      
    colorlinks=true,        
    linkcolor=blue,         
    citecolor=blue,         
    filecolor=blue,         
    urlcolor=cyan           
}

\usepackage{stackengine}
\stackMath
\newcommand\tenq[2][1]{%
\def\useanchorwidth{T}%
\ifnum#1>1%
\stackunder[0pt]{\tenq[\numexpr#1-1\relax]{#2}}{\!\scriptscriptstyle\thicksim}%
\else%
\stackunder[1pt]{#2}{\!\scriptstyle\thicksim}%
\fi%
}

\makeatletter
\DeclareRobustCommand\widecheck[1]{{\mathpalette\@widecheck{#1}}}
\def\@widecheck#1#2{%
    \setbox\z@\hbox{\m@th$#1#2$}%
    \setbox\tw@\hbox{\m@th$#1%
       \widehat{%
          \vrule\@width\z@\@height\ht\z@
          \vrule\@height\z@\@width\wd\z@}$}%
    \dp\tw@-\ht\z@
    \@tempdima\ht\z@ \advance\@tempdima2\ht\tw@ \divide\@tempdima\thr@@
    \setbox\tw@\hbox{%
       \raise\@tempdima\hbox{\scalebox{1}[-1]{\lower\@tempdima\box
\tw@}}}%
    {\ooalign{\box\tw@ \cr \box\z@}}}
\makeatother

\def\tr{\mathop{\text{tr}}\kern.2ex}

\def\P{{\mathrm P}}
\def\Q{{\mathrm Q}}

\def\E{{\mathrm E}}

\def\d{{\mathrm d}}

\def\BL1{{\rm BL}_1}

\renewcommand{\P}{\mathrm{P}}

\newcommand{\zahl}[1]{\llbracket #1\rrbracket}
\newcommand\yestag{\addtocounter{equation}{1}\tag{\theequation}}
\newcolumntype{L}[1]{>{\raggedright\let\newline\\\arraybackslash\hspace{0pt}}m{#1}}
\newcolumntype{C}[1]{>{  \centering\let\newline\\\arraybackslash\hspace{0pt}}m{#1}}
\newcolumntype{R}[1]{>{ \raggedleft\let\newline\\\arraybackslash\hspace{0pt}}m{#1}}
\newcolumntype{d}[1]{D{.}{.}{#1}}
\newcolumntype{H}{>{\setbox0=\hbox\bgroup}c<{\egroup}@{}}
\newcolumntype{Z}{>{\setbox0=\hbox\bgroup}c<{\egroup}@{\hspace*{-\tabcolsep}}}
\newcolumntype{b}{X}
\newcolumntype{s}{>{\hsize=.5\hsize}X}
\def\checkmark{\tikz\fill[scale=0.4](0,.35) -- (.25,0) -- (1,.7) -- (.25,.15) -- cycle;}

\numberwithin{equation}{section}

\newtheorem{theorem}{Theorem}[section]
\newtheorem{lemma}{Lemma}[section]
\newtheorem{proposition}{Proposition}[section]
\newtheorem{assumption}{Assumption}[section]
\newtheorem{corollary}{Corollary}[section]
\newtheorem{innercustomgeneric}{\customgenericname}
\providecommand{\customgenericname}{}
\newcommand{\newcustomtheorem}[2]{%
  \newenvironment{#1}[1]
  {%
   \renewcommand\customgenericname{#2}%
   \renewcommand\theinnercustomgeneric{##1}%
   \innercustomgeneric
  }
  {\endinnercustomgeneric}
}
\newcustomtheorem{customdefinition}{Definition}
\newcustomtheorem{customdefinitions}{Definitions}
\newcustomtheorem{customtheorem}{Theorem}
\newcustomtheorem{customassumption}{Assumption}
\newcustomtheorem{customlemma}{Lemma}
\newcustomtheorem{customexample}{Example}
\theoremstyle{definition}
\newtheorem{definition}{Definition}[section]

\newtheorem{remark}{Remark}[section]

\usepackage{enumitem}
\makeatletter
\newcommand{\mylabel}[2]{#2\def\@currentlabel{#2}\label{#1}}
\makeatother

\graphicspath{{./fig3/}}

\allowdisplaybreaks

\begin{document}

\setlength{\abovedisplayskip}{5pt}
\setlength{\belowdisplayskip}{5pt}
\setlength{\abovedisplayshortskip}{5pt}
\setlength{\belowdisplayshortskip}{5pt}
\hypersetup{colorlinks,breaklinks,urlcolor=blue,linkcolor=blue}

\title{\LARGE Generative bootstrap processes}

\author{Ziming Lin\thanks{Department of Statistics, University of Washington, Seattle, WA 98195, USA; e-mail: {\tt zmlin@uw.edu}}
~~~and ~ Fang Han\thanks{Department of Statistics, University of Washington, Seattle, WA 98195, USA; e-mail: {\tt fanghan@uw.edu}}
}

\date{\today}

\maketitle

\vspace{-1em}

\begin{abstract}
We study generative bootstrap processes obtained by resampling from fitted
generative distributions. We establish necessary and sufficient conditions
for their conditional weak convergence to the $\P$-Brownian bridge,
formulated in terms of finite-dimensional conditional weak convergence and
conditional asymptotic equicontinuity. We further provide general sufficient
conditions and a Gin\'e--Zinn-type characterization relating this convergence
to the $\P$-Donsker property. As applications, we verify these conditions for
prominent classes of generative models, including triangular normalizing
flows, flow matching, score-based diffusion models, and Wasserstein
generative adversarial networks.
\end{abstract}


\section{Introduction}
\label{sec:introduction}

The bootstrap approximates the sampling distribution of a statistic by generating synthetic samples and recomputing the statistic on these samples \citep{efron1979bootstrap}. In Efron's original bootstrap, the synthetic samples are drawn from the empirical distribution. Other bootstrap procedures instead generate synthetic samples from a fitted model for the data distribution. For example, the parametric bootstrap resamples from a fitted parametric model, whereas the smoothed bootstrap resamples from a smoothed nonparametric estimate of the underlying distribution \citep{silverman1987bootstrap,hall1989smoothing}. Thus, a bootstrap procedure is characterized by its resampling law. Once this law is specified, the resulting resample naturally induces a \emph{bootstrap process}.

Let $\bP_n$ denote the empirical measure based on observations from an
unknown data-generating distribution $\P$. The corresponding
$\cF$-indexed \emph{empirical process} is
\[
\bG_n\cF
:=
\Big\{
\sqrt n\big(\bP_n-\P\big)f
;\ f\in\cF
\Big\}.
\]
The function class $\cF$ is said to be \emph{$\P$-Donsker} if there
exists a $\P$-Brownian bridge indexed by $\cF$, denoted by $\bG_\P$,
such that
\[
\bG_n
\rightsquigarrow
\bG_\P
\qquad\text{in }\ell^\infty(\cF),
\]
in the sense of \citet[Section~2.1]{MR1385671}.

A bootstrap process mimics the empirical process above, with $\bP_n$ and $\P$ replaced by their bootstrap analogues. It is said to be consistent if, conditionally on the observed data, it converges weakly to the same $\bG_\P$ in probability. By the functional delta method, consistency of the bootstrap process in turn yields bootstrap consistency for a broad class of smooth statistical functionals of the empirical distribution \citep{dumbgen1993nondifferentiable,MR1385671,fang2019inference}.

The theory of bootstrap processes plays a fundamental role in statistical inference and has generated a substantial literature. For Efron's original bootstrap, \citet{gine1990bootstrapping} showed that conditional weak convergence of the corresponding bootstrap process to the $\P$-Brownian bridge is characterized by the $\P$-Donsker property of $\cF$. \citet{praestgaard1990bootstrap} extended this theory to bootstrap processes based on multiplier weights. Related generalized weighted bootstrap schemes were studied by \citet{mason1992rank}, who established consistency for univariate empirical and quantile processes. \citet{praestgaard1993exchangeably} subsequently developed a general theory of conditional weak convergence for exchangeably weighted bootstrap processes indexed by arbitrary function classes. More recently, \citet{chernozhukov2016empirical} established non-asymptotic bootstrap approximations for suprema of empirical processes indexed by increasingly complex function classes, while \citet{lin2026bootstrap} obtained general bootstrap validity results for sample-splitting-based procedures, including, in particular, double/debiased machine learning estimators \citep{chernozhukov2018double}. See also \cite{lin2023failure} and \cite{lin2026consistency} for some relevant studies.

Modern generative modeling techniques substantially enlarge the class of admissible resampling laws. Rather than restricting bootstrap samples to the observed data points or to conventional parametric or smoothing-based constructions, one may fit a flexible generative model and draw synthetic observations from the resulting learned distribution. It is \citet{tran2026generative} who first formalized this idea by introducing a general generative-model-based bootstrap framework and establishing its validity for both regular and irregular estimators, including settings in which Efron's original bootstrap fails. In a related and notable development, \citet{liang2026diffusion} further advanced this perspective by showing that a diffusion-based bootstrap can consistently capture the uncertainty of high-dimensional ordinary least squares, thereby addressing another setting in which the original bootstrap is known to fail \citep{el2018can}.

The present paper develops the framework of \citet{tran2026generative} at the level of empirical and bootstrap processes. To formalize the setup, let $Z_1,\ldots,Z_n$ be independent copies drawn from $\P$, and let $\P_U$ be a known base distribution. Suppose that a fitting procedure, possibly obtained by adapting an existing generator trained on a large external data set, produces a \emph{generator} $\hat G_n$. Define
\begin{align*}
\Q_n := \hat G_n\#\P_U,
\end{align*}
where $\hat G_n\#\P_U$ is the pushforward measure of $\P_U$ by $\hat G_n$. The measure $\Q_n$ represents the \emph{learned distribution} and is itself random through the procedure used to construct $\hat G_n$. 

Let $\cO$ collect the original data together with any additional randomness used in fitting the generator. Conditionally on $\cO$, let $\tilde Z_1,\ldots,\tilde Z_n$ be independent draws from $\Q_n$. Define
\begin{align*}
\tilde\bP_n
:=
\frac1n
\sum_{i=1}^n
\delta_{\tilde Z_i}\qquad {\rm and}
\qquad
\tilde\bG_n
:=
\sqrt n
\big(
\tilde\bP_n-\Q_n
\big),
\end{align*}
where $\delta_x$ denotes the Dirac measure at $x$. We refer to $\tilde\bG_n$ as the \emph{generative bootstrap process}. Our objective is to establish, with probability tending to one, that
\begin{align}\label{eq:main-target}
\tilde\bG_n
\rightsquigarrow
\bG_\P
\quad\text{conditionally on $\cO$ in }
\ell^\infty(\cF),
\end{align}
in the sense of \citet[Chapter 3.6]{MR1385671}. 

Our main theoretical result provides sufficient conditions under which \eqref{eq:main-target} holds. In this sense, our results complement the classical theory of bootstrap processes developed by \citet{gine1990bootstrapping}, \citet{praestgaard1993exchangeably}, and others. We further verify these conditions for four representative classes of generative models: triangular normalizing flows \citep{dinh2014nice,dinh2017density,irons2022triangular}, flow matching \citep{lipman2022flow}, score-based diffusion models \citep{song2021scorebased,song2021maximum}, and Wasserstein generative adversarial networks \citep{arjovsky2017wasserstein}.

The remainder of the paper is organized as follows. Section~\ref{sec:main} introduces the formal generative bootstrap framework and develops the main theoretical results for the generative bootstrap process. Section~\ref{sec:app} verifies the general conditions for triangular normalizing flows, flow matching, score-based diffusion models, and Wasserstein generative adversarial networks. The appendices contain all proofs, detailed model-specific verifications, and auxiliary technical lemmas.

\paragraph*{Notation.} 
For any positive integers $n,p\in \bN$, let $\zahl{n}:=\{1,2,\ldots,n\}$, and let $\bR^p$ denote the $p$-dimensional Euclidean space. For any $a,b\in\bR$, write $a\vee b:=\max\{a,b\}$ and $a\wedge b:=\min\{a,b\}$. We use $\lVert\cdot\rVert$ and $\lVert\cdot\rVert_\infty$ to denote the Euclidean norm and the sup norm, respectively. For two real sequences $\{a_n\}_{n\ge1}$ and $\{b_n\}_{n\ge1}$, we write $a_n\lesssim b_n$ (or $a_n=O(b_n)$) if there exists a constant $C>0$ such that $|a_n|\le C|b_n|$ for all sufficiently large $n$, and write $a_n=o(b_n)$ if $a_n/b_n\to0$ as $n\to\infty$. For a sequence of random elements $\{Z_n\}_{n\ge1}$, we write $Z_n=o_\P(1)$ if $Z_n\to0$ in $\P$-probability, and $Z_n=O_\P(1)$ if $\{Z_n\}_{n\ge1}$ is bounded in $\P$-probability. For a set $A\subset\bR^p$, write 
\[
\operatorname{int}(A), ~~~\bar A,~~~ \text{and}~~~\partial A:=\bar A\setminus\operatorname{int}(A)
\]
for its interior, closure, and boundary, respectively. Define
\[
B_R:=\{z\in\bR^p:\lVert z\rVert\le R\}.
\]
For probability measures $\mu$ and $\nu$, write $\mu\ll\nu$ if $\mu$ is absolutely continuous with respect to $\nu$, and $\mu\perp\nu$ if $\mu$ and $\nu$ are mutually singular. Throughout the paper, we use the convention 
\[
\P f=\int f \d \P.
\]

\paragraph*{A note on measurability conventions.}
Throughout the paper, we impose the standard measurability conditions ensuring that the empirical-process suprema, conditional expectations, and conditional probabilities appearing below are measurable. Equivalent formulations in terms of outer expectation and outer probability can be given under the usual empirical-process conventions. To streamline the exposition, we suppress this distinction.

\section{Main theory}
\label{sec:main}


Let $\P_U$ be a known base distribution, and let $U,\tilde U_1,\tilde U_2,\ldots$ be independent draws from $\P_U$. Define a generic synthetic observation and the bootstrap observations by
\begin{align*}
\tilde Z
&:= \hat G_n(U),
\qquad
\tilde Z_i
:= \hat G_n(\tilde U_i),
\qquad i \in \zahl{n}.
\end{align*}
The learned bootstrap distribution is the \emph{random} probability measure
\begin{align*}
\Q_n
:=
\hat G_n\#\P_U.
\end{align*}
Recall that in this paper we use $\cO$ to represent all information (i.e. the induced sigma-field) used to construct the \emph{learned generator} $\hat G_n$, including the original sample $(Z_1,\ldots,Z_n)$ and any auxiliary randomness arising from training and optimization. Thus, conditional on $\cO$,
\[
\tilde Z,\tilde Z_1,\ldots,\tilde Z_n
\overset{\rm i.i.d.}{\sim}
\Q_n.
\]
Hereafter, we use $\P_\cO$ to denote the probability measure induced on
the $\sigma$-field $\cO$, and use $\E_{\tilde U}[\cdot\mid\cO]$ and
$\P_{\tilde U}(\cdot\mid\cO)$ to denote conditional expectation and
probability with respect to the bootstrap randomness only.

Recall that the bootstrap empirical distribution is defined by
\[
\tilde\bP_n
:=
\frac1n
\sum_{i=1}^n
\delta_{\tilde Z_i}
\]
along with
\[
\tilde\bG_n\cF
=
\Big\{\sqrt n
\big(
\tilde\bP_n-\Q_n
\big); f\in\cF\Big\}.
\]
We refer to $\tilde\bG_n\cF$ as the \emph{($\cF$-indexed) generative bootstrap process}. The choice $\Q_n=\bP_n$ recovers Efron's bootstrap process. More generally, conditional on $\cO$, $\tilde\bG_n$ is the centered empirical process associated with an independent and identically distributed (i.i.d.) sample of size $n$ drawn from $\Q_n$.

Let $\bG_\P$ denote the $\P$-Brownian bridge indexed by $\cF$, namely, the centered Gaussian process with covariance function
\begin{align*}
\E
\big[
\bG_\P[f]\bG_\P[g]
\big]
=
\P[fg]
-
\P[f]\P[g],
\qquad
f,g\in\cF.
\end{align*}
Our object of study is the conditional weak convergence of $\tilde\bG_n$ to $\bG_\P$ in $\ell^\infty(\cF)$. Throughout the paper, whenever such conditional weak convergence or the $\P$-Donsker property of $\cF$ is considered, we assume that $\bG_\P$ is a tight, separable, Borel random element of $\ell^\infty(\cF)$.

\subsection{Characterization of conditional weak convergence}
\label{sec:CWC_characterization}

To characterize conditional weak convergence, we follow the standard empirical- and bootstrap-process terminology outlined in, e.g., \citet[Chapter~3.6]{MR1385671} and \citet[Chapter~10]{kosorok2008introduction}. To this end, we first introduce the bounded-Lipschitz \emph{test class} used to formulate weak convergence and the canonical semimetric governing empirical-process increments.

For a metric space $(\bD,d)$, let 
\begin{align*}
\BL1(\bD,d)
:=
\Big\{
h:\bD\to\bR:
\sup_{x\in\bD}
\lvert h(x)\rvert
\le 1,
\quad
\sup_{x_1\ne x_2}
\frac{
\lvert h(x_1)-h(x_2)\rvert
}{
d(x_1,x_2)
}
\le 1
\Big\}.
\end{align*}
We consider the function space
\[
\ell^\infty(\cF):=\Big\{x:\cF \to  \bR: \|x\|_{\cF}:=
\sup_{f\in\cF}
\Big\lvert x[f]\Big\rvert<\infty\Big\}
\]
equipped with the supremum norm $\|\cdot\|_{\cF}$.

For any probability measure $\Q$ such that $\cF\subseteq L^2(\Q)$,
the $L^2$ space with regard to $\Q$, define the canonical semimetric
\begin{align*}
\rho_\Q(f,g)^2
:=
\Q
\big[
\big\{
(f-g)-\Q[f-g]
\big\}^2
\big],
\qquad
f,g\in\cF.
\end{align*}
With this notation, the following result specializes the classical characterization of weak convergence through finite-dimensional convergence and asymptotic equicontinuity to the conditional generative-bootstrap setting.



\begin{theorem}[Characterization of conditional weak convergence]
\label{thm:cvg_emp_proc}
Let $\cF\subseteq L^2(\P)$ be a measurable class such that the generative bootstrap process $\tilde\bG_n$ is well defined
as an $\ell^\infty(\cF)$-valued random element for every $n$. Then $\tilde\bG_n$ converges conditionally
weakly to $\bG_\P$ in $\ell^\infty(\cF)$ in the sense that
\begin{align*}
\sup_{h\in\BL1(\ell^\infty(\cF),\lVert \cdot\rVert_\cF)}
\Big\lvert
\E_{\tilde U}
\big[
h(\tilde\bG_n)
\mid
\cO
\big]
-
\E
\big[
h(\bG_\P)
\big]
\Big\rvert
=
o_{\P_\cO}(1),
\yestag
\label{eq:CWC}
\end{align*}
if and only if both of the following conditions hold:
\begin{enumerate}[itemsep=-.5ex,label=(\roman*)]
\item \emph{Finite-dimensional conditional weak convergence}: for every
$k\in\bN$ and $f_1,\ldots,f_k\in\cF$, we have
\begin{align*}
\sup_{h\in\BL1(\bR^k,\lVert \cdot\rVert_{\infty})}
\Big\lvert
\E_{\tilde U}\big[
h\big(
\tilde\bG_n[f_1],\ldots,\tilde\bG_n[f_k]
\big)
\mid\cO
\big]
-
\E\big[
h\big(
\bG_\P[f_1],\ldots,\bG_\P[f_k]
\big)
\big]
\Big\rvert
=
o_{\P_\cO}(1).
\yestag
\label{eq:FD-CWC}
\end{align*}

\item \emph{Conditional asymptotic $\rho_\P$-equicontinuity}: for every
$\epsilon,\eta>0$,
\begin{align*}
\lim_{\delta\downarrow0}
\limsup_{n\to\infty}
\P_\cO
\Big\{
\P_{\tilde U}
\Big(
\sup_{\rho_\P(f,g)<\delta}
\big\lvert
\tilde\bG_n[f]
-
\tilde\bG_n[g]
\big\rvert
>
\epsilon
\ \Big|\ 
\cO
\Big)
>
\eta
\Big\}
=
0.
\yestag
\label{eq:CAE}
\end{align*}
\end{enumerate}
\end{theorem}

Theorem~\ref{thm:cvg_emp_proc} itself does not require an envelope
condition. To streamline the presentation, however, we restrict attention in the remainder of the paper to those $\cF$ with \emph{bounded envelopes}. Under this additional condition,
Proposition~\ref{prop:FD_cov_characterization} in the appendix shows that
\eqref{eq:FD-CWC} is equivalent to the following
\emph{finite-dimensional covariance consistency} condition: for every $f,g\in\cF$,
\begin{align*}
\Q_n[fg]
-
\Q_n[f]\Q_n[g]
=
\P[fg]
-
\P[f]\P[g]
+
o_{\P_\cO}(1).
\yestag
\label{eq:FD-cov}
\end{align*}
Consequently, in the bounded-envelope setting adopted below,
\eqref{eq:CWC} holds if and only if \eqref{eq:FD-cov} and
\eqref{eq:CAE} hold. The remainder of this section develops verifiable sufficient
conditions for these two requirements in turn.


\subsection{Sufficient conditions for finite-dimensional covariance consistency}
\label{sec:suff_cond_FD_cov}

Condition~\eqref{eq:FD-cov} reduces to the convergence of
$\Q_n[h]$ to $\P[h]$ for those $h=f$, $g$, and $fg$. 

\begin{definition}[Weak consistency] $\Q_n$ is said  to converge weakly to
$\P$ in $\P_\cO$-probability and written as 
\[
\Q_n \rightsquigarrow \P \text{ in } \P_\cO\text{-probability}
\]
if, for every \emph{bounded continuous function}
$h:\bR^p\to\bR$,
\begin{align*}
(\Q_n-\P)h = o_{\P_\cO}(1).
\end{align*}
We also refer to this property as \emph{weak consistency} of $\Q_n$ to $\P$.
\end{definition}

Weak consistency gives convergence of expectations for bounded continuous functions, but the indexing class $\cF$ need not consist of continuous
functions. The following approximation lemma provides a general device
for extending weak convergence to bounded measurable test functions.

\begin{lemma}[Continuous approximation]
\label{lem:continuous_approximation_principle}
Let $\cG$ be a class of bounded Borel measurable functions on $\bR^p$.
Assume $\Q_n\rightsquigarrow\P$ in $\P_\cO$-probability. Suppose further that,
for every $h\in\cG$, there exists a sequence of bounded continuous functions
$\{h_m\}_{m\ge1}$ such that
\begin{align*}
\P[\lvert h-h_m\rvert]
\longrightarrow
0,
\yestag
\label{eq:CAP_P_approximation}
\end{align*}
and, for every $\epsilon>0$,
\begin{align*}
\lim_{m\to\infty}
\limsup_{n\to\infty}
\P_\cO\Big(
\Q_n[\lvert h-h_m\rvert]>\epsilon
\Big)
=
0.
\yestag
\label{eq:CAP_Qn_approximation}
\end{align*}
Then $(\Q_n-\P)h=o_{\P_\cO}(1)$ for every $h\in\cG$.
\end{lemma}

Lemma~\ref{lem:continuous_approximation_principle} yields two complementary
routes for verifying \eqref{eq:FD-cov}. The first places regularity on the
indexing functions and requires only weak consistency 
of $\Q_n$. The second
instead places regularity on the measures, thereby allowing arbitrary
bounded measurable indexing functions.

\begin{theorem}[Sufficient conditions for Condition \eqref{eq:FD-cov}]
\label{thm:FD-cov}
Suppose $\cF$ has a bounded envelope function. 
Then \eqref{eq:FD-cov} holds if either of the following conditions is
satisfied.
\begin{enumerate}[itemsep=-.5ex,label=(\roman*)]

\item
$\Q_n\rightsquigarrow\P$ in $\P_\cO$-probability and every $f\in\cF$ is
$\P$-a.s.\ continuous.

\item
$\Q_n\rightsquigarrow\P$ in $\P_\cO$-probability, and there exists a
deterministic finite Borel measure $\nu$ on $\bR^p$ such that
$\P\ll\nu$ and $\Q_n$ is asymptotically uniformly absolutely continuous
with respect to $\nu$, that is, for every $\epsilon>0$,
\begin{align*}
\lim_{\delta\downarrow0}
\limsup_{n\to\infty}
\P_\cO\Big(
\sup_{A:\nu(A)\le\delta}
\Q_n(A)
>
\epsilon
\Big)
=
0,
\yestag
\label{eq:AUAC_Qn_nu}
\end{align*}
where the supremum is taken over all Borel measurable sets $A$.
\end{enumerate}
\end{theorem}

Theorem~\ref{thm:FD-cov}(i) is particularly convenient when consistency of
$\Q_n$ is available in Wasserstein distance. By
Lemma~\ref{lemma:wq-implies-weak} in the appendix, $\Q_n\rightsquigarrow\P$ in
$\P_\cO$-probability whenever, for some $q\ge1$,
$W_q(\Q_n,\P)=o_{\P_\cO}(1)$. This condition also accommodates familiar
indicator classes underlying empirical distribution functions. For
$\cF=\{\ind_A:A\in\cA\}$, the $\P$-a.s.\ continuity requirement reduces to
\[
\P(\partial A)=0
\qquad
\text{for every }A\in\cA.
\]

Theorem~\ref{thm:FD-cov}(ii) instead shifts regularity from the indexing
functions to the measures: sets with vanishingly small $\nu$-measure cannot
carry nonnegligible $\Q_n$-mass. This removes the continuity requirement on
$\cF$ entirely. A stronger but convenient sufficient condition is convergence
in total variation. Standard divergence inequalities then yield the following
immediate consequences.

\begin{corollary}[Total-variation and divergence criteria]
\label{cor:divergence_FD_cov}
Suppose $\cF$ has a bounded envelope. Then \eqref{eq:FD-cov} holds if either of the following
conditions is satisfied:
\begin{enumerate}[itemsep=-.5ex,label=(\roman*)]
\item
the total variation distance converges:
\[
d_{\rm TV}(\Q_n,\P)=o_{\P_\cO}(1);
\]

\item
the Kullback–Leibler (KL) divergence converges in either direction:
\[
D_{\rm KL}(\Q_n\|\P)=o_{\P_\cO}(1)
\qquad\text{or}\qquad
D_{\rm KL}(\P\|\Q_n)=o_{\P_\cO}(1);
\]

\item
the Hellinger distance converges:
\[
H(\Q_n,\P)=o_{\P_\cO}(1);
\]

\item
the chi-square divergence converges:
$\Q_n\ll\P$ with probability tending to one and
\[
\chi^2(\Q_n\|\P)=o_{\P_\cO}(1).
\]
\end{enumerate}
\end{corollary}

Corollary \ref{cor:divergence_FD_cov} verifies the measure-side condition through
convergence in statistical divergences. For density-based generative models,
the same condition can often be checked more directly through the learned
densities. The next lemma shows that \eqref{eq:AUAC_Qn_nu} is equivalent to
\emph{uniform integrability} of these densities with respect to the reference
measure.

\begin{lemma}[Uniform integrability and asymptotic uniform absolute continuity]
\label{lem:nu_UI_AUAC}
Let $\nu$ be a finite Borel measure on $\bR^p$.
Suppose
$\P_\cO(\Q_n\ll\nu)\to1$. On the event $\{\Q_n\ll\nu\}$, write
$q_n:=\d\Q_n/\d\nu$, and set $q_n:=0$ otherwise. Then
\eqref{eq:AUAC_Qn_nu} holds if and only if the densities $q_n$ are uniformly
integrable with respect to $\nu$ in $\P_\cO$-probability: for every
$\epsilon>0$,
\begin{align*}
\lim_{L\to\infty}
\limsup_{n\to\infty}
\P_\cO\Big(
\int
q_n\ind\{q_n>L\}\d\nu
>
\epsilon
\Big)
=
0.
\yestag
\label{eq:nu_UI_density}
\end{align*}
\end{lemma}

Uniform integrability can in turn be verified through a superlinear
integrability bound.

\begin{corollary}[Density criteria for asymptotic uniform absolute continuity]
\label{cor:superlinear_AUAC}
Let $\nu$ be a finite Borel measure on $\bR^p$. Suppose
$\P_\cO(\Q_n\ll\nu)\to1$. On the event $\{\Q_n\ll\nu\}$, write
$q_n:=\d\Q_n/\d\nu$, and set $q_n:=0$ otherwise. If there exists a measurable
function $\Phi:[0,\infty)\to[0,\infty)$ satisfying
$\Phi(t)/t\to\infty$ as $t\to\infty$ such that
\begin{align*}
\int
\Phi(q_n)\d\nu
=
O_{\P_\cO}(1),
\yestag
\label{eq:reference_orlicz_bound}
\end{align*}
then \eqref{eq:nu_UI_density} holds. In particular, the same conclusion holds
if, for some $r\in(1,\infty]$,
\[
\lVert q_n\rVert_{L^r(\nu)}
=
O_{\P_\cO}(1).
\]
\end{corollary}

\begin{remark}
Taking $\nu=\P$ gives a useful special case. If
$\P_\cO(\Q_n\ll\P)\to1$ and, for some $r\in(1,\infty]$,
\[
\Big\lVert
\frac{\d\Q_n}{\d\P}
\Big\rVert_{L^r(\P)}
=
O_{\P_\cO}(1),
\]
then \eqref{eq:nu_UI_density} holds. Taking
$\Phi(t)=t\log t-t+1$ further shows that the bounded forward-KL condition
\[
D_{\rm KL}(\Q_n\|\P)
=
O_{\P_\cO}(1)
\]
is another sufficient condition for \eqref{eq:nu_UI_density} to hold.
\end{remark}

Theorem~\ref{thm:FD-cov}(ii), Lemma~\ref{lem:nu_UI_AUAC}, and
Corollary~\ref{cor:superlinear_AUAC} all require the reference measure
$\nu$ to be finite.
They therefore cannot
be applied directly with Lebesgue measure on $\bR^p$. We end this section with a 
localization argument, which replaces this global requirement by uniform
integrability of the learned densities on bounded sets. This criterion
will be particularly useful in Section~\ref{sec:app}, where we apply
the general theory to learned laws $\Q_n$ arising from different
generative modeling approaches.

\begin{corollary}[Localized Lebesgue-density criterion]
\label{cor:localized_Lebesgue_FD_cov}
Suppose $\cF$ has a bounded envelope. Assume
$\Q_n\rightsquigarrow\P$ in $\P_\cO$-probability and $\P\ll\lambda$, where
$\lambda$ denotes Lebesgue measure on $\bR^p$. Assume further that
$\P_\cO(\Q_n\ll\lambda)\to1$. On the event $\{\Q_n\ll\lambda\}$, write
$q_n:=\d\Q_n/\d\lambda$, and set $q_n:=0$ otherwise. Suppose that, for every
$R>0$ and every $\epsilon>0$,
\begin{align*}
\lim_{L\to\infty}
\limsup_{n\to\infty}
\P_\cO\Big(
\int_{B_R}
q_n\ind\{q_n>L\}\d\lambda
>
\epsilon
\Big)
=
0.
\yestag
\label{eq:local_Lebesgue_UI}
\end{align*}
Then Condition \eqref{eq:FD-cov} holds.
\end{corollary}

Localization reduces the measure-side condition to uniform integrability of
the learned densities on each fixed ball. In particular, by the same argument
as in Corollary~\ref{cor:superlinear_AUAC}, applied to each fixed ball $B_R$,
condition~\eqref{eq:local_Lebesgue_UI} holds if, for every $R>0$, either
\begin{align*}
\lVert q_n\rVert_{L^\infty(B_R,\lambda)}
:=
\operatorname*{esssup}_{z\in B_R}
\lvert q_n(z)\rvert
=
O_{\P_\cO}(1),
\end{align*}
or, for some $r_R>1$,
\[
\int_{B_R}
q_n^{r_R}\d\lambda
=
O_{\P_\cO}(1).
\]

\subsection{Sufficient conditions for conditional asymptotic $\rho_\P$-equicontinuity}
\label{sec:suff_cond_CAE}

We next turn to conditional asymptotic $\rho_\P$-equicontinuity \eqref{eq:CAE}. We
develop two complementary routes. The first transfers empirical-process
bounds from a fixed reference measure to the learned law $\Q_n$ through
\emph{measure domination}.
The second uses a distribution-free entropy bound
together with a local comparison between the canonical semimetrics
$\rho_\P$ and $\rho_{\Q_n}$.

We begin with the measure domination route.

\begin{definition}[Measure domination]
For probability measures $\mu$ and $\nu$, we say that $\mu$ is
\emph{dominated by $\nu$ up to a factor $D$}, written
$\mu\le D\nu$, if
\[
\mu(A)\le D\nu(A)
\]
for every Borel set $A$. Equivalently, $\mu\ll\nu$ and
\[
\frac{\d\mu}{\d\nu}\le D
\qquad
\nu\text{-almost surely}.
\]
\end{definition}

For any
$r\in(0,\infty]$, we define
\[
\Big\lVert
\frac{\d\mu}{\d\nu}
\Big\rVert_{L^r(\nu)}
:=
+\infty
\qquad
\text{when }\mu\not\ll\nu.
\]
Under this convention, $\mu\le D\nu$ is also equivalent to
\[
\Big\lVert
\frac{\d\mu}{\d\nu}
\Big\rVert_{L^\infty(\nu)}
\le D.
\]

To compare empirical-process fluctuations under different sampling
laws, let $X_1,\ldots,X_m$ be i.i.d.\ from a probability measure $\mu$
and define
\begin{align*}
\bG_m^\mu[a]
&:=
\frac{1}{\sqrt m}
\sum_{i=1}^m
\big\{
a(X_i)-\mu[a]
\big\},
\qquad
\lVert\bG_m^\mu\rVert_{\cA}
:=
\sup_{a\in\cA}
\big\lvert
\bG_m^\mu[a]
\big\rvert.
\end{align*}
Expectations indexed by $\mu$ include the randomness of this i.i.d.\
sample. When $\mu=\Q_n$, such expectations are understood conditionally
on the fitting information $\cO$.

For $\delta>0$, define the local difference class
\begin{align*}
\cF_\delta
:=
\big\{
f-g:
f,g\in\cF,\ 
\rho_\P(f,g)<\delta
\big\}.
\end{align*}
The following theorem then translates Condition \eqref{eq:CAE} to a measure domination condition, Condition \eqref{eq:stochastic_reference_domination}, coupled with a Donsker condition on that \emph{reference measure}, Condition \eqref{eq:local_mod_mean_ref_meas}.

\begin{theorem}[Reference domination implies conditional asymptotic equicontinuity]
\label{thm:reference_measure}
Suppose $\cF$ has a bounded envelope function. Assume further that there
exists a deterministic probability measure $\nu$ such that
\begin{align*}
\lim_{C\to\infty}
\limsup_{n\to\infty}
\P_\cO\Big(
\Q_n\nleq C\nu
\Big)
=
0,
\yestag
\label{eq:stochastic_reference_domination}
\end{align*}
and
\begin{align*}
\lim_{\delta\downarrow0}
\limsup_{m\to\infty}
\E_\nu
\Big[
\lVert\bG_m^\nu\rVert_{\cF_\delta}
\Big]
=
0.
\yestag
\label{eq:local_mod_mean_ref_meas}
\end{align*}
Then the conditional asymptotic $\rho_\P$-equicontinuity condition
\eqref{eq:CAE} holds.
\end{theorem}

The measure domination condition admits a simple density-ratio
interpretation. Under the convention above,
\eqref{eq:stochastic_reference_domination} is equivalent to
\[
\Big\lVert
\frac{\d\Q_n}{\d\nu}
\Big\rVert_{L^\infty(\nu)}
=
O_{\P_\cO}(1).
\]
The target distribution itself is therefore a natural choice of
reference measure. When $\nu=\P$ and $\cF$ is $\P$-Donsker, Condition \eqref{eq:local_mod_mean_ref_meas} holds
automatically, leaving only measure domination of $\Q_n$ by $\P$.
This yields the following corollary.

\begin{corollary}[$\P$-Donsker plus $\P$-domination implies Condition \eqref{eq:CAE}]
\label{cor:domination}
Suppose $\cF$ is $\P$-Donsker and has a bounded envelope function. If further
\begin{align}
\lim_{C\to\infty}\label{eq:han-lrb0}
\limsup_{n\to\infty}
\P_\cO\Big(
\Q_n\nleq C\P
\Big)
=
0,
\end{align}
or equivalently,
\begin{align}\label{eq:han-lrb}
\Big\lVert
\frac{\d\Q_n}{\d\P}
\Big\rVert_{L^\infty(\P)}
=
O_{\P_\cO}(1),
\end{align}
then Condition
\eqref{eq:CAE} holds.
\end{corollary}

For density-based models, Condition \eqref{eq:han-lrb} can often be
verified directly from Lebesgue densities. In particular, if $\Q_n$ is
supported on a set over which the target density is bounded away from
zero, then an $L^\infty$ bound on the learned density yields the required
$\P$-domination. This gives the following criterion, which is the last result we provide along the measure domination route.

\begin{corollary}[Lebesgue-density sufficient condition]
\label{cor:Lebegue_density}
Suppose $\cF$ is $\P$-Donsker with a bounded envelope. Suppose $\P$
admits a Lebesgue density $p$ and that there exists a Borel set
$S\subseteq\bR^p$ such that
\[
c
:=
\operatorname*{essinf}_{z\in S}
p(z)
>
0.
\]
Assume that, with probability tending to one, $\Q_n\ll\lambda$ and
$\Q_n(S)=1$. On the event $\{\Q_n\ll\lambda\}$, write
$q_n:=\d\Q_n/\d\lambda$. If
\begin{align*}
\lVert q_n\rVert_{L^\infty(\lambda)}
=
O_{\P_\cO}(1),
\end{align*}
then the conditional asymptotic $\rho_\P$-equicontinuity condition
\eqref{eq:CAE} holds.
\end{corollary}

The measure domination route is convenient when the learned laws can be
controlled by a fixed reference measure, but such a requirement can be
restrictive for particular generative models.  The second route offers alternatives
that avoid this domination requirement. For $0<u\le1$, define
\begin{align*}
J(u,\cF,F)
:=
\int_0^u
\sup_{\Lambda}
\sqrt{
1+\log \cN\Big(
\epsilon\lVert F\rVert_{\Lambda,2},
\cF,
L^2(\Lambda)
\Big)
}
\,\d\epsilon,
\end{align*}
where $F$ is an envelope function of $\cF$, the supremum is taken over all probability measures $\Lambda$ on
the sample space, and $\cN(\epsilon,\cG,d)$ denotes the
$\epsilon$-covering number of a function class $\cG$ with respect to the
semimetric $d$.

\begin{theorem}[Uniform entropy implies conditional asymptotic equicontinuity]
\label{thm:unif_entropy}
Suppose $\cF$ has an envelope function bounded by $B>0$. Choose $F\equiv B$
and assume that the uniform entropy integral satisfies
\begin{align*}
J(1,\cF,F)
<
\infty.
\yestag
\label{eq:unif_entropy_F}
\end{align*}
Suppose further that, for every $\epsilon>0$,
\begin{align*}
\lim_{\delta\downarrow0}
\limsup_{n\to\infty}
\P_\cO\Big(
\sup_{\rho_\P(f,g)<\delta}
\rho_{\Q_n}(f,g)
>
\epsilon
\Big)
=
0.
\yestag
\label{eq:local_metric_transfer_Qn}
\end{align*}
Then the conditional asymptotic $\rho_\P$-equicontinuity condition
\eqref{eq:CAE} holds.
\end{theorem}

The entropy condition \eqref{eq:unif_entropy_F} is stronger than the
$\P$-Donsker property, as it controls covering numbers uniformly over
all underlying probability measures. It is nevertheless satisfied by
many standard classes, including finite classes, finite-dimensional
parametric classes with regular parametrizations, VC-subgraph classes,
and many VC-type indicator classes.

The metric-transfer condition \eqref{eq:local_metric_transfer_Qn}
requires increments that are small under $\rho_\P$ to remain
asymptotically small under $\rho_{\Q_n}$. The following lemma provides a
sufficient condition for \eqref{eq:local_metric_transfer_Qn} based on
local Lebesgue-density control and asymptotic tightness of the learned
laws.

\begin{lemma}[Local Lebesgue control implies local metric transfer]
\label{lem:local_density_metric_transfer}
Suppose $\cF$ has a bounded envelope function. Assume that
$\Q_n$ is asymptotically tight in $\P_\cO$-probability in the sense that,
for every $\epsilon>0$,
\begin{align*}
\lim_{R\to\infty}
\limsup_{n\to\infty}
\P_\cO\Big(
\Q_n(B_R^c)>\epsilon
\Big)
=
0.
\yestag
\label{eq:Qn_asymptotic_tightness}
\end{align*}
Suppose $\P$ admits a Lebesgue density $p$ satisfying
\[
\operatorname*{essinf}_{z\in B_R}
p(z)
>
0
\qquad
\text{for every }R<\infty.
\]
Assume further that $\Q_n\ll\lambda$ with probability tending to one.
On this event, write $q_n:=\d\Q_n/\d\lambda$, and set $q_n:=0$ otherwise.
Suppose that, for every $R<\infty$, there exists $r_R>1$ such that
\begin{align*}
\int_{B_R}
q_n(z)^{r_R}\d z
=
O_{\P_\cO}(1).
\end{align*}
Then \eqref{eq:local_metric_transfer_Qn} holds.
\end{lemma}

\begin{remark}
The asymptotic tightness requirement
\eqref{eq:Qn_asymptotic_tightness} is mild. In particular, it follows
automatically from $\Q_n\rightsquigarrow\P$ in $\P_\cO$-probability. To
see this, for any $\epsilon>0$, choose $R<\infty$ sufficiently large
such that
\[
\P(B_R^c)<\epsilon/2
\qquad\text{and}\qquad
\P(\partial B_R)=0.
\]
The continuity-set form of the Portmanteau theorem then gives
\[
\Q_n(B_R^c)-\P(B_R^c)
=
o_{\P_\cO}(1),
\]
and consequently
\[
\P_\cO
\big\{
\Q_n(B_R^c)>\epsilon
\big\}
\longrightarrow
0.
\]
\end{remark}

\subsection{A Donsker characterization under measure domination}
\label{subsec:relation_Donsker}

It is well known that conditional weak convergence of Efron's original bootstrap process is equivalent to the $\P$-Donsker property of $\cF$
\citep[Theorem~3.1]{gine1990bootstrapping}. For the generative
bootstrap, Corollary~\ref{cor:domination} for instance gives one direction: under
measure (upper) domination by $\P$, the $\P$-Donsker property implies
conditional asymptotic $\rho_\P$-equicontinuity. To obtain an analogous
equivalence as Efron's bootstrap, it remains to recover the $\P$-Donsker property from
conditional asymptotic equicontinuity. The following result shows that
a lower-coverage condition holding with nonvanishing probability is
sufficient.

\begin{theorem}[Reverse transfer under lower coverage]
\label{thm:CAE_to_Donsker}
Let $\cF$ be a measurable class with a bounded envelope function. Suppose further that, for some $D<\infty$,
\begin{align*}
\liminf_{n\to\infty}
\P_\cO\Big(
\P\le D\Q_n
\Big)
>0.
\yestag
\label{eq:lower_domination_nonvanishing}
\end{align*}
Then, if the conditional asymptotic $\rho_\P$-equicontinuity condition
\eqref{eq:CAE} holds, $\cF$ is $\P$-Donsker.
\end{theorem}


Combining Theorem \ref{thm:CAE_to_Donsker} with Corollary~\ref{cor:domination}  yields
the following analogue of the Gin\'e--Zinn characterization for 
generative bootstrap processes.

\begin{corollary}[Donsker characterization under measure domination]
\label{cor:Donsker_equivalence_weak_consistency}
Let $\cF$ be a measurable class with a bounded envelope function. Suppose that \eqref{eq:han-lrb0} and 
\eqref{eq:lower_domination_nonvanishing} hold and that
$\Q_n\rightsquigarrow\P$ in $\P_\cO$-probability. Then the conditional
weak convergence \eqref{eq:CWC} holds if and only if $\cF$ is
$\P$-Donsker.
\end{corollary}


\begin{remark}
A convenient sufficient condition for the comparisons \eqref{eq:han-lrb0} and 
\eqref{eq:lower_domination_nonvanishing} is
\begin{align*}
\Big\lVert
\frac{\d\Q_n}{\d\P}
\Big\rVert_{L^\infty(\P)}
&=
O_{\P_\cO}(1),
\qquad
\Big\lVert
\frac{\d\P}{\d\Q_n}
\Big\rVert_{L^\infty(\Q_n)}
=
O_{\P_\cO}(1).
\end{align*}
The first bound is equivalent to
\eqref{eq:han-lrb0}. The second implies
\eqref{eq:lower_domination_nonvanishing}: by stochastic boundedness, one
can choose a fixed $D<\infty$ such that
\begin{align*}
\liminf_{n\to\infty}
\P_\cO\Big(
\Big\lVert
\frac{\d\P}{\d\Q_n}
\Big\rVert_{L^\infty(\Q_n)}
\le D
\Big)
>0,
\end{align*}
and the event in the preceding display is precisely
$\{\P\le D\Q_n\}$. In particular, both directional comparisons are
verified by the stronger condition that there exist fixed constants
$0<c\le C<\infty$ such that
\begin{align*}
\P_\cO\Big(
\Q_n\ll\P
\text{ and }
c
\le
\frac{\d\Q_n}{\d\P}
\le
C,
\quad
\P\text{-almost surely}
\Big)
\longrightarrow
1.
\end{align*}
\end{remark}

\section{Applications to generative models}
\label{sec:app}

This section applies the general theory developed in
Section~\ref{sec:main} to four representative classes of generative
models. 
These four examples illustrate complementary routes through the general
theory. \emph{Triangular normalizing flows} combine distributional consistency
with global likelihood-ratio domination; \emph{flow matching} combines
Wasserstein consistency with local density control; \emph{score-based
diffusion models} yield total-variation consistency; and \emph{Wasserstein
generative adversarial networks} rely on Wasserstein consistency without
requiring the learned law to admit a density or be dominated by the
target distribution. 

Each subsection concludes with a brief review of the relevant learning
theory and a comparison with our distributional consistency results. Our
aim is \emph{not} to compete with the state-of-the-art learning theory for the
corresponding generative models. Rather, we focus on identifying weak
sufficient conditions under which the associated generative bootstrap
process is consistent in the sense of Section~\ref{sec:main}.


\subsection{Triangular normalizing flows}
\label{sec:triangular_normalizing_flows}

A normalizing flow represents a target distribution as the pushforward
of a simple base distribution through an invertible transformation,
thereby allowing its likelihood to be evaluated by the change-of-variables
formula. Triangular and block-triangular Jacobians underlie many
autoregressive and coupling-based flow architectures
\citep{dinh2014nice,dinh2017density,papamakarios2017masked,
huang2018neural,kingma2018glow}; see \citet{kobyzev2020normalizing}
for a review. Here, we consider a support-matched monotone triangular
construction based on the triangular transport framework of
\citet{irons2022triangular}; the same framework has also been exploited in \cite{tran2026generative}.


\subsubsection{Consistency of normalizing flow bootstrap}

We begin with conditions on the target distribution. Throughout this
subsection, the support of $\P$ refers to the smallest closed set
$\cZ\subseteq\bR^p$ such that $\P(\cZ)=1$.

\begin{assumption}[Data assumption 1: compact support and density regularity]
\label{ass:triangular_data}
The data distribution $\P$ satisfies the following conditions.
\begin{enumerate}[itemsep=-.5ex,label=(\roman*)]
\item The support $\cZ$ of $\P$ is known, compact, and convex with
nonempty interior.

\item The measure $\P$ admits a version of its Lebesgue density $p$
supported on $\cZ$ such that, for some constant $\underline p>0$,
\[
p(z)\ge\underline p,
\qquad
z\in\operatorname{int}(\cZ).
\]

\item The log-density is integrable:
\[
\P[\lvert\log p\rvert]<\infty.
\]
\end{enumerate}
\end{assumption}

We then parameterize the transformation from the data space to the noise
space. 
In the following, for a differentiable map
$S=(S_1,\ldots,S_p)^\top$, write
\[
D_kS_j(z)
:=
\frac{\partial S_j(z)}{\partial z_k},
\qquad
J_S(z)
:=
\big(D_kS_j(z)\big)_{1\le j,k\le p}.
\]

\begin{definition}[Monotone triangular $C^1$-diffeomorphism]
\label{def:triangular_normalizing_flow}
Let $\cZ,\cU\subset\bR^p$ be compact convex sets with nonempty interiors.
A map
\begin{align*}
S=(S_1,\ldots,S_p)^\top:\cZ\to\cU
\end{align*}
is said to be a \emph{monotone triangular $C^1$-diffeomorphism} if the following
conditions hold.
\begin{enumerate}[itemsep=-.5ex,label=(\roman*)]
\item $S$ is a homeomorphism from $\cZ$ onto $\cU$, and both
\[
S|_{\operatorname{int}(\cZ)}
:
\operatorname{int}(\cZ)
\to
\operatorname{int}(\cU)
\]
and
\[
S^{-1}|_{\operatorname{int}(\cU)}
:
\operatorname{int}(\cU)
\to
\operatorname{int}(\cZ)
\]
are continuously differentiable.

\item For each $j\in\zahl{p}$ and every $z\in\cZ$,
\[
S_j(z)=S_j(z_j,\ldots,z_p).
\]

\item For each $j\in\zahl{p}$,
$D_jS_j(z)>0$ for every $z\in\operatorname{int}(\cZ)$, and all
first-order partial derivatives of $S$ admit continuous extensions to
$\cZ$.
\end{enumerate}
\end{definition}

Each such map induces a generator $S^{-1}$ from the noise space to the
data space. Let $\P_U$ be a known noise distribution with support
$\cU$ and Lebesgue density $p_U$. For a monotone triangular
$C^1$-diffeomorphism $S:\cZ\to\cU$, write
\[
\Q_S
:=
S^{-1}\#\P_U.
\]
Since $S^{-1}(\cU)=\cZ$ and $\P_U(\cU)=1$, we have
$\Q_S(\cZ)=1$.

Since compact convex sets with nonempty interiors have Lebesgue-null
boundaries, the change-of-variables formula applied on the interiors
shows that $\Q_S$ admits a version of its Lebesgue density satisfying
\begin{align*}
q_S(z)
&=
p_U\big(S(z)\big)
\lvert\det J_S(z)\rvert
=
p_U\big(S(z)\big)
\prod_{j=1}^pD_jS_j(z),
\qquad
z\in\operatorname{int}(\cZ).
\end{align*}
We set $q_S(z)=0$ for $z\notin\cZ$ and define it arbitrarily on
$\partial\cZ$. The corresponding negative log-likelihood loss is
\begin{align*}
\ell_S(z)
&:=
-\log q_S(z)
=
-\log p_U\big(S(z)\big)
-
\sum_{j=1}^p\log D_jS_j(z),
\qquad
z\in\operatorname{int}(\cZ).
\end{align*}
We extend $\ell_S$ arbitrarily to $\partial\cZ$. Since $\P$ admits a
Lebesgue density and $\lambda(\partial\cZ)=0$, this extension does not
affect any $\P$-integral.

We now pass from a generic triangular transport to the fitted
normalizing-flow model used for bootstrap resampling.

\begin{definition}[Support-matched triangular normalizing flow]
\label{def:triangular_normalizing_flow_model}
Let $\cT_n$ be a sequence of classes of monotone triangular
$C^1$-diffeomorphisms from $\cZ$ onto $\cU$, whose complexity may
increase with $n$. The fitted triangular map $\hat S_n\in\cT_n$ is an
approximate maximum-likelihood estimator satisfying, for some
nonnegative empirical optimization error $\varepsilon_n^{\rm NF}$,
\begin{align*}
\bP_n\ell_{\hat S_n}
\le
\inf_{S\in\cT_n}
\bP_n\ell_S
+
\varepsilon_n^{\rm NF}.
\end{align*}
The learned generator, its induced bootstrap law, and the corresponding
density are given by
\begin{align*}
\hat G_n
&:=
\hat S_n^{-1},
\qquad
\Q_n
:=
\hat G_n\#\P_U
=
\Q_{\hat S_n},
\qquad
q_n
:=
q_{\hat S_n}.
\end{align*}
\end{definition}

We next impose regularity conditions on the triangular model, together
with approximation, generalization, and optimization conditions on the
likelihood fit.

\begin{assumption}[Model assumption 1: triangular-flow model and fitting conditions]
\label{ass:triangular_model}
The triangular normalizing-flow model satisfies the following conditions.
\begin{enumerate}[itemsep=-.5ex,label=(\roman*)]
\item The support of $\P_U$ is a compact convex set
$\cU\subset\bR^p$ with nonempty interior, and $\P_U$ admits a version
of its Lebesgue density $p_U$ satisfying, for some constants
$0<\underline p_U\le\overline p_U<\infty$,
\begin{align*}
\underline p_U
\le
p_U(u)
\le
\overline p_U,
\qquad
u\in\operatorname{int}(\cU).
\end{align*}

\item For every $n$, $\cT_n$ is a class of monotone triangular
$C^1$-diffeomorphisms from $\cZ$ onto $\cU$. There exists a
deterministic constant $M_S<\infty$ such that, for all sufficiently
large $n$,
\begin{align*}
\sup_{S\in\cT_n}
\sup_{z\in\cZ}
\max_{1\le j\le p}
D_jS_j(z)
\le
M_S.
\end{align*}

\item Approximation:
\[
\alpha_n^{\rm NF}
:=
\inf_{S\in\cT_n}
D_{\rm KL}(\P\|\Q_S)
\longrightarrow
0.
\]

\item Generalization: $\ell_S$ is $\P$-integrable for every
$S\in\cT_n$, and
\[
\Delta_n^{\rm NF}
:=
\sup_{S\in\cT_n}
\big\lvert
(\bP_n-\P)\ell_S
\big\rvert
=
o_{\P_\cO}(1).
\]

\item Optimization:
\[
\varepsilon_n^{\rm NF}
=
o_{\P_\cO}(1).
\]
\end{enumerate}
\end{assumption}

These conditions translate directly into KL consistency through
the following oracle inequality.

\begin{lemma}[Likelihood oracle inequality]
\label{lem:triangular_KL_consistency}
Suppose Assumption~\ref{ass:triangular_data} and
Assumption~\ref{ass:triangular_model}(i) hold, and suppose
$\ell_S\in L^1(\P)$ for every $S\in\cT_n$. Then
\begin{align*}
D_{\rm KL}(\P\|\Q_n)
\le
\alpha_n^{\rm NF}
+
2\Delta_n^{\rm NF}
+
\varepsilon_n^{\rm NF}.
\end{align*}
Consequently, under Assumption~\ref{ass:triangular_model}(iii)--(v),
\[
D_{\rm KL}(\P\|\Q_n)
=
o_{\P_\cO}(1).
\]
\end{lemma}

Lemma~\ref{lem:triangular_KL_consistency} and
Corollary~\ref{cor:divergence_FD_cov} yield Condition \eqref{eq:FD-cov}. Moreover,
Assumptions~\ref{ass:triangular_data}(ii) and
\ref{ass:triangular_model}(i)--(ii) imply, for all sufficiently large
$n$,
\begin{align*}
q_n(z)
\le
\overline p_U M_S^p
\le
\frac{\overline p_U M_S^p}{\underline p}\,p(z)
\end{align*}
for Lebesgue-almost every $z\in\cZ$. Hence
Corollary~\ref{cor:domination} gives conditional asymptotic
$\rho_\P$-equicontinuity for every $\P$-Donsker class with a bounded
envelope. Theorem~\ref{thm:cvg_emp_proc} therefore yields the main
result of this subsection.

\begin{theorem}[Triangular-flow bootstrap process consistency]
\label{thm:triangular_flow_bootstrap_CWC}
Suppose Assumptions~\ref{ass:triangular_data}
and~\ref{ass:triangular_model} hold. Let $\cF$ be a $\P$-Donsker class
with a bounded envelope function. Then the triangular-flow generative
bootstrap process satisfies
\begin{align*}
\sup_{h\in\BL1(\ell^\infty(\cF), 
\lVert \cdot\rVert_\cF
)}
\Big\lvert
\E_{\tilde U}
\big[
h(\tilde\bG_n)
\mid
\cO
\big]
-
\E
\big[
h(\bG_\P)
\big]
\Big\rvert
=
o_{\P_\cO}(1).
\end{align*}
\end{theorem}

The strength of the triangular-flow argument stems from exact support
matching and the resulting global likelihood-ratio control. These
conditions accommodate arbitrary $\P$-Donsker classes with bounded
envelopes, but are most natural for compactly supported targets whose
densities are bounded away from zero. For normalizing flows on the full
Euclidean space, such global likelihood-ratio control is generally not
available automatically.

\begin{remark}
Section~\ref{sec:NF-apdx} will provide a concrete triangular-flow
class for which all the conditions of
Theorem~\ref{thm:triangular_flow_bootstrap_CWC} can be verified.
Moreover, we will show that this class satisfies the stronger conditions of
Corollary~\ref{cor:Donsker_equivalence_weak_consistency}. Consequently,
for this class of normalizing flows, conditional weak convergence of the
triangular-flow generative bootstrap process holds for every class
$\cF$ with a bounded envelope if and only if $\cF$ is $\P$-Donsker.
\end{remark}




\subsubsection{Learning-theoretic review and discussion}

\label{sec:triangular-literature}

The learning theory of triangular flows studies how likelihood-based
fitting translates into guarantees for the learned distribution. Under compact-support,
positivity, and smoothness assumptions, \citet{irons2022triangular} used
empirical-process bounds to establish KL consistency and sample-size rates
for approximate likelihood minimizers over a fixed triangular class. 
\citet{wang2022minimax} refined the likelihood analysis on the unit cube:
for sufficiently smooth positive H\"older densities, they obtain Hellinger
concentration at the minimax density-estimation rate, with corresponding
TV and KL bounds. Their penalized and wavelet-sieve estimators extend this
analysis beyond optimization over a fixed smoothness ball.


\begin{table}[t]
\centering
\caption{Selected statistical guarantees for triangular-flow distribution
estimators. Conditions refer to the indicated results;
KL denotes $D_{\rm KL}(\P\|\Q_n)$.}
\label{tab:triangular-learning-comparison}
\vspace{4pt}
\begingroup
\small
\setlength{\tabcolsep}{4pt}
\setlength{\arrayrulewidth}{0.4pt}
\setlength{\extrarowheight}{2pt}
\renewcommand{\arraystretch}{1.12}
\begin{tabularx}{\textwidth}{@{}
>{\raggedright\arraybackslash}p{0.235\textwidth}|
>{\raggedright\arraybackslash}X|
>{\raggedright\arraybackslash}p{0.31\textwidth}@{}}
\hline
& Key assumptions and model class
& Guarantee for the learned distribution \\[3pt]
\hline
\citet{irons2022triangular}\newline
Theorem~3.6
& Positive smooth target and base densities on compact convex supports;
  fixed triangular class containing the true transport;
  approximate likelihood minimization
& KL consistency in probability; expected KL rates with additional
  optimization-error control \\[3pt]
\hline
\citet{wang2022minimax}\newline
Theorem~2.3,\newline
Corollary~2.4
& Positive $\alpha$-H\"older target density on the unit cube,
  $\alpha>p/2$; smooth positive product base density;
  regular triangular likelihood class
& Hellinger concentration at the minimax rate; corresponding TV
  and KL rates in probability \\[3pt]
\hline
\citet{tran2026generative}\newline
Section~7.4, Step~1c
& Compact supports with the stated support containment;
  derivative-controlled autoregressive layers;
  vanishing empirical log-likelihood gap
& KL and $W_1$ errors tend to zero in probability \\[3pt]
\hline
This subsection
& Support-matched $C^1$ maps on compact sets;
  positive densities and integrable log losses;
  vanishing approximation, generalization, and optimization errors
& Realization-wise KL bound; KL error tends to zero in probability \\[3pt]
\hline
\end{tabularx}
\endgroup
\end{table}

These statistical bounds depend on the approximation error of the map
class. Under analytic-density assumptions on a cube, \citet{zech2022sparse} constructed monotone triangular approximations,
including neural-network constructions, and derive Hellinger, TV,
and Wasserstein approximation bounds as the approximation class grows. These bounds follow from controlling
the maps and, for the density-based distances, their derivatives.
\citet{baptista2025approximation} gave a broader transport-stability
framework that clarifies this dependence on the distance: $L^r$ errors in
sampling maps control $W_r$ error, whereas their KL bounds require
additional derivative and invertibility conditions. Such results quantify
the approximation error of a chosen map class; they do not by themselves
give statistical rates in the number of observations.

Our subsection uses the likelihood approach at the level of sufficient
conditions. Rather than choosing a map class to achieve a particular
statistical rate, we require its best KL approximation error, uniform
log-loss deviation, and numerical optimization error to vanish.
The same likelihood argument as in \citet{irons2022triangular} then gives
$D_{\rm KL}(\P\|\Q_n)\to0$ in $\P_\cO$-probability, allowing
$n$-dependent classes. A related qualitative analysis by
\citet{tran2026generative} combines an assumed empirical log-likelihood
gap with a uniform law for derivative-controlled autoregressive flows. Table~\ref{tab:triangular-learning-comparison} compares selected statistical
guarantees for the learned distributions.

\subsection{Flow matching}
\label{sec:FM}

We next consider flow matching, a continuous-time generative model
defined through a time-dependent vector field and an ordinary
differential equation. Unlike the support-matched triangular flow in
the preceding section, flow matching is naturally compatible
with target distributions having full support. Our argument combines
Wasserstein consistency of the learned distribution with local
Lebesgue-density control, and focuses on the affine conditional interpolation
standard in conditional flow matching \citep{lipman2022flow,albergo2022building,albergo2025stochastic}.

\subsubsection{Consistency of flow matching bootstrap}

We begin by specifying the affine conditional path and the corresponding
marginal flow connecting the base and target distributions. Fix
continuously differentiable functions
$\mu,\sigma:[0,1]\to[0,1]$ satisfying
\begin{align*}
\mu_0=0,
\qquad
\sigma_0=1,
\qquad
\mu_1=1,
\qquad
\sigma_1=0,
\qquad
\sigma_t>0
\quad
\text{for }0\le t<1.
\end{align*}
Let $U\sim\P_U$ and $Z\sim\P$ be independent $\bR^p$-valued random
variables. For $0\le t\le1$, let $\dot\mu_t$ and $\dot\sigma_t$ denote
the derivatives of $\mu_t$ and $\sigma_t$ with respect to $t$,
respectively, and define
\begin{align*}
X_t
:=
\mu_t Z+\sigma_t U,
\qquad
V_t
:=
\dot\mu_t Z+\dot\sigma_t U.
\end{align*}
Let $\P_t$ denote the law of $X_t$. Then
$\P_0=\P_U$ and $\P_1=\P$. Conditional on $Z=z$, the distribution of
$X_t$ is the pushforward of $\P_U$ under the map
$u\mapsto\mu_t z+\sigma_t u$; when $\P_U$ is standard Gaussian, this
recovers the usual Gaussian conditional path. The corresponding
tractable conditional flow-matching target is $V_t$.

The affine construction yields the usual weak continuity equation for
$(\P_t,v_t)$. To avoid imposing separate regularity conditions on the
associated characteristic flow, we assume directly that
$\{\P_t\}_{t\in[0,1]}$ admits a deterministic flow representation.

\begin{assumption}[Data and path assumption 2: second moment and marginal flow]
\label{ass:FM_data}
Assume the target distribution $\P$ has finite second moment:
\begin{align*}
\int_{\bR^p}
\lVert z\rVert^2
\d\P(z)
<
\infty.
\end{align*}
For $0\le t<1$, let
\[
v_t(x)
:=
\E[V_t\mid X_t=x]
\]
be a version of the marginal velocity. Assume that $v_{\cdot}(\cdot)$ admits a jointly
measurable extension to $[0,1]\times\bR^p$, still denoted by $v$, and
that there exists a jointly measurable family of maps
$\{\Phi_t:t\in[0,1]\}$ satisfying
\begin{align*}
\Phi_t(u)
=
u+
\int_0^t
v_s\big(
\Phi_s(u)
\big)
\d s,
\qquad
0\le t\le1,
\end{align*}
and
\[
\P_t
=
\Phi_t\#\P_U,
\qquad
0\le t\le1.
\]
\end{assumption}

In particular,
$\P=\Phi_1\#\P_U$, so the population flow transports the base
distribution to the target distribution. We next specify a regular,
controlled class of candidate vector fields for approximating this
population velocity.


\begin{assumption}[Controlled flow-matching vector fields]
\label{def:controlled_neural_FM_class}
For each $n$, let $\cV_n$ be a nonempty deterministic class of jointly
measurable vector fields
\[
\tilde v:[0,1]\times\bR^p\to\bR^p. 
\]
Assume that, for every
$\tilde v\in\cV_n$, the map $x\mapsto\tilde v_t(x)$ is continuously
differentiable for every $t$, the map
$(t,x)\mapsto\nabla_x\tilde v_t(x)$ is jointly measurable, and there
exist constants $L,M<\infty$, independent of $n$, such that
\begin{align*}
\sup_{t\in[0,1]}
\lVert
\tilde v_t(0)
\rVert
\le M,
\qquad
\sup_{t\in[0,1],\,x\in\bR^p}
\lVert
\nabla_x\tilde v_t(x)
\rVert_{\rm op}
\le L,
\end{align*}
where $\nabla$ refers to the Jacobian, and $\|\cdot\|_{\rm op}$ stands for the matrix operator norm.
\end{assumption}

The bounds in Assumption~\ref{def:controlled_neural_FM_class} imply
\begin{align*}
\lVert
\tilde v_t(x)-\tilde v_t(y)
\rVert
&\le
L\lVert x-y\rVert,
\\
\lVert
\tilde v_t(x)
\rVert
&\le
L\lVert x\rVert+M,
\\
\big\lvert
\nabla\cdot\tilde v_t(x)
\big\rvert
&\le
pL.
\end{align*}
Consequently, every $\tilde v\in\cV_n$ generates a unique global flow
on $[0,1]$ in the sense of Carath\'eodory
\citep[Theorem~54 and Proposition~C.3.8]{sontag1998mathematical}.
The spatial $C^1$ regularity, together with the uniqueness of
the forward and backward equations, further implies that each time-$t$
flow map is a $C^1$-diffeomorphism of $\bR^p$.

We now describe the training procedure and the induced flow-matching
generator. Conditionally on $Z_1,\ldots,Z_n$, we draw
$T_1,\ldots,T_n$ independently from ${\rm Unif}(0,1)$ and
$U_1,\ldots,U_n$ independently from $\P_U$, with the two samples
conditionally independent. Define
\begin{align*}
X_i
:=
\mu_{T_i}Z_i+\sigma_{T_i}U_i,
\qquad
V_i
:=
\dot\mu_{T_i}Z_i+\dot\sigma_{T_i}U_i,
\qquad
1\le i\le n.
\end{align*}
Unconditionally, the triples
$\{(T_i,X_i,V_i)\}_{i=1}^n$ are i.i.d.\ with the same distribution as
$(T,X_T,V_T)$, where
$Z\sim\P$, $U\sim\P_U$, and $T\sim{\rm Unif}(0,1)$ are mutually
independent. The augmentation variables are included in the fitting
information $\cO$.


\begin{assumption}[Augmented-sample flow-matching fit and generator]
\label{def:FM_training_objective}
\label{def:flow_matching_generator}
For each $\tilde v\in\cV_n$, define the population and empirical
conditional flow-matching risks by
\begin{align*}
\cR^{\rm FM}(\tilde v)
&:=
\E\Big[
\big\lVert
\tilde v_T(X_T)-V_T
\big\rVert^2
\Big],
\\
\hat \cR^{\rm FM}_n(\tilde v)
&:=
\frac{1}{n}
\sum_{i=1}^n
\big\lVert
\tilde v_{T_i}(X_i)-V_i
\big\rVert^2.
\end{align*}
Assume that the fitting procedure approximately minimizes
$\hat \cR^{\rm FM}_n$ over $\cV_n$ and returns an
$\cO$-measurable fitted vector field $\hat v_n\in\cV_n$, which
determines the learned flow through
\begin{align*}
\hat\Phi_{n,t}(u)
=
u+
\int_0^t
\hat v_{n,s}\big(
\hat\Phi_{n,s}(u)
\big)
\d s,
\qquad
0\le t\le1.
\end{align*}
The learned generator and its induced bootstrap law are then
\begin{align*}
\hat G_n(u)
:=
\hat\Phi_{n,1}(u),
\qquad
\Q_n
:=
\hat G_n\#\P_U.
\end{align*}
\end{assumption}

The population flow-matching risk admits a useful projection
interpretation. Since $T$ is independent of $(Z,U)$, conditionally on
$T=t$, the pair $(X_T,V_T)$ has the same distribution as $(X_t,V_t)$.
Hence, by the definition of the marginal velocity,
\[
v_T(X_T)
=
\E[V_T\mid T,X_T]
\]
almost surely. The squared-loss projection identity therefore gives,
for every $\tilde v\in\cV_n$,
\begin{align*}
\cR^{\rm FM}(\tilde v)
&=
\cR^{\rm FM}(v)
+
\E\Big[
\lVert
\tilde v_T(X_T)-v_T(X_T)
\rVert^2
\Big]
\\
&=
\cR^{\rm FM}(v)
+
\int_0^1
\lVert
\tilde v_t-v_t
\rVert_{L^2(\P_t)}^2
\d t.
\yestag
\label{eq:CFM_projection_identity}
\end{align*}
Thus, the population conditional flow-matching objective targets the
marginal velocity generating the desired probability path.

Motivated by this projection identity, we next impose statistical and
computational conditions under which the augmented-sample fit
consistently learns the marginal velocity.

\begin{assumption}[Model assumption 2: flow-matching approximation and fitting conditions]
\label{ass:controlled_neural_FM_training}
Assume the base distribution $\P_U$ satisfies
\begin{align*}
\int_{\bR^p}
\lVert u\rVert^2
\d\P_U(u)
<
\infty.
\end{align*}
Moreover, suppose the augmented-sample fit in
Assumption~\ref{def:FM_training_objective} satisfies the following
conditions.
\begin{enumerate}[label=(\roman*),itemsep=-.5ex]
\item Approximation:
\[
\alpha_n^{\rm FM}
:=
\inf_{\tilde v\in\cV_n}
\big\{
\cR^{\rm FM}(\tilde v)-\cR^{\rm FM}(v)
\big\}
=
o(1).
\]

\item Generalization:
\[
\Delta_n^{\rm FM}
:=
\sup_{\tilde v\in\cV_n}
\big\lvert
\hat \cR^{\rm FM}_n(\tilde v)
-
\cR^{\rm FM}(\tilde v)
\big\rvert
=
o_{\P_\cO}(1).
\]

\item Optimization:
\[
\varepsilon_n^{\rm FM}
:=
\hat \cR^{\rm FM}_n(\hat v_n)
-
\inf_{\tilde v\in\cV_n}
\hat \cR^{\rm FM}_n(\tilde v)
=
o_{\P_\cO}(1).
\]
\end{enumerate}
\end{assumption}


Combining the projection identity
\eqref{eq:CFM_projection_identity} with the preceding conditions yields
an integrated $L^2(\P_t)$ error bound for the fitted velocity. A
Gronwall stability argument \citep[Chapter 1]{pachpatte1998inequalities} comparing the learned and population ODE
flows then converts this bound into Wasserstein consistency of the
generated law.

\begin{lemma}[Flow-matching training consistency implies Wasserstein consistency]
\label{lem:FM_training_implies_weak_convergence}
Suppose Assumptions~\ref{ass:FM_data}-\ref{ass:controlled_neural_FM_training} hold. Then
\begin{align*}
\int_0^1
\lVert
\hat v_{n,t}-v_t
\rVert_{L^2(\P_t)}^2
\d t
\le
\alpha_n^{\rm FM}
+
2\Delta_n^{\rm FM}
+
\varepsilon_n^{\rm FM},
\end{align*}
and, realization-wise,
\begin{align*}
W_2(\Q_n,\P)
\le
\exp(L)
\Big(
\alpha_n^{\rm FM}
+
2\Delta_n^{\rm FM}
+
\varepsilon_n^{\rm FM}
\Big)^{1/2}.
\end{align*}
Consequently,
\[
W_2(\Q_n,\P)
=
o_{\P_\cO}(1),
\]
and hence
$\Q_n\rightsquigarrow\P$ in $\P_\cO$-probability.
\end{lemma}

The regularity conditions on the learned vector field also control the
Jacobian determinant of the induced flow, leading to the following
density bound.

\begin{lemma}[Density bound for controlled flow-matching generators]
\label{lem:FM_density_bound}
Suppose Assumptions~\ref{def:controlled_neural_FM_class} and
\ref{def:FM_training_objective} hold. Assume further that $\P_U$ admits
a Lebesgue density $p_0$ satisfying
$\lVert p_0\rVert_{L^\infty(\lambda)}<\infty$. Then $\Q_n$ admits a
Lebesgue density $q_n$ and, realization-wise,
\begin{align*}
\lVert q_n\rVert_{L^\infty(\lambda)}
\le
\lVert p_0\rVert_{L^\infty(\lambda)}
\exp(pL),
\end{align*}
where $p$ is the dimension of the data space.
\end{lemma}

Lemma~\ref{lem:FM_training_implies_weak_convergence} gives
$\Q_n\rightsquigarrow\P$ in $\P_\cO$-probability, while
Lemma~\ref{lem:FM_density_bound} gives a uniform $L^\infty$ bound on
the learned densities. Hence
Corollary~\ref{cor:localized_Lebesgue_FD_cov} yields
finite-dimensional covariance consistency. The same density bound,
together with weak consistency and the local positivity condition on
$p$, verifies the local metric-transfer condition through
Lemma~\ref{lem:local_density_metric_transfer}. Under
\eqref{eq:unif_entropy_F},
Theorem~\ref{thm:unif_entropy} therefore yields conditional asymptotic
$\rho_\P$-equicontinuity. Applying
Theorem~\ref{thm:cvg_emp_proc} gives the following result.

\begin{theorem}[Flow-matching bootstrap consistency]
\label{thm:FM_bootstrap_CWC}
Suppose Assumptions~\ref{ass:FM_data}-\ref{ass:controlled_neural_FM_training} hold. Assume further
that $\P_U$ admits a Lebesgue density $p_0$ satisfying
\[
\lVert p_0\rVert_{L^\infty(\lambda)}
<
\infty,
\]
and that $\P$ admits a Lebesgue density $p$ satisfying
\[
\operatorname*{essinf}_{z\in B_R}
p(z)
>
0
\qquad
\text{for every }R<\infty.
\]
Let $\cF$ be a measurable class with an envelope bounded by
$B<\infty$ and satisfying the uniform entropy condition
\eqref{eq:unif_entropy_F}; that is,
$J(1,\cF,F)<\infty$ with $F\equiv B$. Then the flow-matching
generative bootstrap process satisfies
\begin{align*}
\sup_{h\in\BL1(\ell^\infty(\cF),\lVert \cdot \rVert_\cF)}
\Big\lvert
\E_{\tilde U}\Big[
h(\tilde\bG_n)
\mid
\cO
\Big]
-
\E\Big[
h(\bG_\P)
\Big]
\Big\rvert
=
o_{\P_\cO}(1).
\end{align*}
\end{theorem}




\subsubsection{Learning-theoretic review and discussion}
\label{sec:FM-literature}

We review how the flow-matching learning theory controls the
Wasserstein distance between the target and learned distributions, in
either $W_1$ or $W_2$. Two issues are particularly relevant to our
setting: the regularity required for ODE stability and the estimation of
the velocity field from finite training samples.

\begin{table}[t]
\centering
\caption{Selected flow-matching guarantees under the assumptions of the
indicated results. Only key conditions are listed.}
\label{tab:FM-learning-comparison}
\vspace{4pt}
\begingroup
\small
\setlength{\tabcolsep}{4pt}
\setlength{\arrayrulewidth}{0.4pt}
\setlength{\extrarowheight}{2pt}
\renewcommand{\arraystretch}{1.12}
\begin{tabularx}{\textwidth}{@{}
>{\raggedright\arraybackslash}p{0.235\textwidth}|
>{\raggedright\arraybackslash}X|
>{\raggedright\arraybackslash}p{0.30\textwidth}@{}}
\hline
& Key assumptions and model class
& Wasserstein error guarantee \\[3pt]
\hline
\citet{benton2024flow}\newline Theorem~1
& Given integrated velocity-error bound; well-posed smooth flows;
  time-integrable spatial Lipschitz bound on the fitted field
& Non-asymptotic $W_2$ bound in terms of velocity estimation error \\[3pt]
\hline
\citet{fukumizu2025flow}\newline
Section~4.2, Theorems~3--4
& Compact target support; integrated-loss minimization; bounded loss and
  covering-number control; Lipschitz fields on a truncated interval
& Bound for expected $W_2$ error relative to the smoothed target,
  through approximation and loss-class complexity \\[3pt]
\hline
\citet{gao2024continuous}\newline
Theorem~4.4, Gaussian-mixture case
& Centered Gaussian convolution of a bounded law; constrained ReLU
  classes; sampled loss; early stopping and Euler sampling
& Rate for expected $W_2$ error, including numerical sampling error \\[3pt]
\hline
This subsection
& Finite second moments; marginal-flow representation; fixed Lipschitz
  and linear-growth bounds on $[0,1]$; vanishing approximation, generalization and
  optimization errors
& Realization-wise $W_2$ bound; $W_2$ error tends to zero in probability \\[3pt]
\hline
\end{tabularx}

\endgroup
\end{table}


The regularity of the fitted vector field determines how velocity
estimation error propagates to Wasserstein error.
\citet[Proposition~3]{albergo2022building} bounded the squared $W_2$
error by the integrated squared velocity error under a uniform spatial
Lipschitz bound on the fitted field. \citet[Theorem~1]{benton2024flow}
extended this comparison to time-dependent Lipschitz bounds, with the
constant in their $W_2$ bound depending on the time integral of these
bounds.

Learning the velocity field from data additionally requires
approximation and finite-sample control. \citet{gao2024continuous}
analyzed an objective based on sampled times and noise, combining ReLU
approximation and finite-sample bounds with early-stopping and Euler
discretization errors to obtain a rate for the expected $W_2$ error
under their distributional regularity conditions.
\citet{zhou2025flowerror} likewise studied velocity fields learned by deep
least-squares regression and establish $W_2$ convergence for
bounded-support targets. \citet{fukumizu2025flow}, in contrast, analyzed
an integrated loss for positive, compactly supported Besov densities,
combining neural-network approximation with statistical error bounds
under additional boundary and path regularity conditions.


Related work also addresses endpoint regularity in $W_1$ error bounds
\citep{kunkel2025minimax,kunkel2026lipschitz} and flow-matching
estimation for manifold-supported targets \citep{kumar2026manifold}.

In our framework, a fixed spatial Lipschitz bound over $[0,1]$ provides
ODE stability, while separate approximation, generalization, and
optimization conditions yield consistency of the fitted velocity field.
This parallels the empirical-risk and $W_2$-stability analysis of
\citet[Section~4.2]{fukumizu2025flow}, although our generalization
condition is formulated for the sampled, rather than integrated, loss.
Section~\ref{sec:FM-apdx} verifies a uniform law for this sampled loss
over norm-controlled affine-residual networks under sub-Gaussian tail
conditions and a parameter-growth restriction. Approximation still
requires a target-specific closure condition within the constrained
function class. 

Table~\ref{tab:FM-learning-comparison} summarizes
selected guarantees from the literature and compares them with the
conditions used here.

\subsection{Score-based diffusion models}
\label{sec:diffusion}


We next consider score-based diffusion models, which learn the scores of a
continuum of Gaussian-smoothed distributions and generate observations
through a learned reverse-time stochastic differential equation
\citep{song2021scorebased}. Under the score and regularity conditions below,
we establish total-variation consistency of the learned law and transfer
it to the generative bootstrap process.

\subsubsection{Consistency of diffusion bootstrap process}

We use the standard Ornstein--Uhlenbeck (OU) noising path, which avoids
additional notation for the noise schedule. More general
variance-preserving schedules can be handled through a deterministic time
change.

\begin{assumption}[Data assumption 3: density and second moment]
\label{ass:diffusion_data_path}
The target distribution $\P$ admits a Lebesgue density and satisfies
\begin{align*}
\int_{\bR^p}
\lVert z\rVert^2
\d\P(z)
<
\infty.
\end{align*}
\end{assumption}

Let $Z\sim\P$, and let
$W^{\rm fwd}=\{W_t^{\rm fwd}:t\ge0\}$ be a $p$-dimensional Brownian
motion independent of $Z$. Define the forward Ornstein--Uhlenbeck
process by
\begin{align*}
\d X_t
=
-X_t\d t
+
\sqrt{2}\d W_t^{\rm fwd},
\qquad
X_0=Z.
\end{align*}
For each $t>0$, the conditional distribution of $X_t$ given $Z=z$ is
\[
{\rm N}\big(
e^{-t}z,
(1-e^{-2t})I_p
\big).
\]
Equivalently,
\[
X_t
\overset{\rm d}{=}
e^{-t}Z
+
\sqrt{1-e^{-2t}}\xi,
\]
where $\xi\sim\gamma_p:={\rm N}(0,I_p)$ is independent of $Z$.
Let $\P_t$ denote the distribution of $X_t$, let $p_t$ denote its
Lebesgue density, and define the score function by
\begin{align*}
s_t(x)
:=
\nabla_x\log p_t(x).
\end{align*}

Let $\tau_n$ and $T_n$ be deterministic sequences satisfying
$0<\tau_n<T_n$, $\tau_n\downarrow0$, and $T_n\to\infty$. By the
standard time-reversal analysis detailed in
Proposition~\ref{prop:diffusion_OU_reverse_path}, for each $n$, the
time reversal of the forward OU process over $[\tau_n,T_n]$ admits the
representation
\begin{align*}
\d \bar X_{n,r}
=
\big\{
\bar X_{n,r}
+
2s_{T_n-r}(\bar X_{n,r})
\big\}
\d r
+
\sqrt{2}\d \bar W_r^{\rm rev},
\qquad
\bar X_{n,0}
\sim
\P_{T_n},
\end{align*}
for $0\le r\le T_n-\tau_n$, where
$\bar W^{\rm rev}$ is a $p$-dimensional Brownian motion with respect to
the reversed filtration. Moreover,
\begin{align*}
\big\{
\bar X_{n,r}:
0\le r\le T_n-\tau_n
\big\}
\overset{\rm d}{=}
\big\{
X_{T_n-r}:
0\le r\le T_n-\tau_n
\big\}.
\end{align*}

The learned sampler below replaces the initial law $\P_{T_n}$ by
$\gamma_p$ and the score $s_t$ by an estimator $\hat s_{n,t}$.
To estimate the unavailable marginal score $s_t$, we target the
tractable conditional score
\citep{hyvarinen2005estimation,vincent2011connection}. Since the
conditional density of $X_t$ given $Z=z$ is Gaussian, define
\begin{align*}
\psi_t(x,z)
:=
\nabla_x\log p_t(x\mid z)
=
-
\frac{x-e^{-t}z}{1-e^{-2t}}.
\end{align*}
In particular,
\[
\psi_t(X_t,Z)
=
-
\frac{\xi}{\sqrt{1-e^{-2t}}}.
\]
The lower cutoff $\tau_n$ serves as an early-stopping time and keeps
this target away from the singularity at $t=0$.

We now describe the denoising-score fit and the induced reverse-diffusion
generator. Conditionally on $Z_1,\ldots,Z_n$, draw
$T_1,\ldots,T_n$ independently from ${\rm Unif}(\tau_n,T_n)$ and
$\xi_1,\ldots,\xi_n$ independently from $\gamma_p$, with the two
samples conditionally independent. Define
\begin{align*}
X_i
:=
e^{-T_i}Z_i
+
\sqrt{1-e^{-2T_i}}\xi_i,
\qquad
\Psi_i
:=
-
\frac{\xi_i}{\sqrt{1-e^{-2T_i}}},
\qquad
1\le i\le n.
\end{align*}
The augmentation variables are included in the fitting information
$\cO$.


\begin{assumption}[Augmented-sample denoising-score fit and reverse-diffusion generator]
\label{def:diffusion_fit_generator}
For each $n$, let $\cS_n$ be a deterministic class of jointly measurable
score fields
\begin{align*}
\tilde s:
[\tau_n,T_n]\times\bR^p
\to
\bR^p
\end{align*}
such that every $\tilde s\in\cS_n$ satisfies
\begin{align*}
\int_{\tau_n}^{T_n}
\lVert
\tilde s_t
\rVert_{L^2(\P_t)}^2
\d t
<
\infty.
\end{align*}
For $\tilde s\in\cS_n$, define the population and empirical denoising
score-matching (DSM) risks by
\begin{align*}
\cR_n^{\rm DSM}(\tilde s)
&:=
\int_{\tau_n}^{T_n}
\E\Big[
\big\lVert
\tilde s_t(X_t)-\psi_t(X_t,Z)
\big\rVert^2
\Big]
\d t,
\\
\hat \cR_n^{\rm DSM}(\tilde s)
&:=
\frac{T_n-\tau_n}{n}
\sum_{i=1}^n
\big\lVert
\tilde s_{T_i}(X_i)-\Psi_i
\big\rVert^2.
\yestag
\label{eq:diffusion_DSM_objectives}
\end{align*}
Assume the fitting procedure minimizes
$\hat \cR_n^{\rm DSM}$ and returns an $\cO$-measurable score field
$\hat s_n\in\cS_n$.
\end{assumption}

Let $\P_U$ be the law of
$U=(U_0,W^{\rm rev})$, where $U_0\sim\gamma_p$ and
$W^{\rm rev}=\{W_t^{\rm rev}:t\ge0\}$ is a $p$-dimensional Brownian
motion independent of $U_0$. Given $\hat s_n$, define the learned
reverse process by
\begin{align*}
\hat X_{n,r}
=
U_0
+
\int_0^r
\big\{
\hat X_{n,u}
+
2\hat s_{n,T_n-u}(\hat X_{n,u})
\big\}
\d u
+
\sqrt{2}W_r^{\rm rev},
\qquad
0\le r\le T_n-\tau_n.
\yestag
\label{eq:learned_reverse_diffusion}
\end{align*}
The regularity condition imposed below ensures that this equation admits
a unique solution driven by $(U_0,W^{\rm rev})$. Define
\begin{align*}
\hat G_n(U)
:=
\hat X_{n,T_n-\tau_n},
\qquad
\Q_n
:=
\hat G_n\#\P_U.
\end{align*}

The conditional-expectation identity for Gaussian perturbations gives
\[
s_t(X_t)
=
\E\big[
\psi_t(X_t,Z)
\mid
X_t
\big]
\]
for every $t>0$. Consequently, the squared-loss projection identity
gives, for every $\tilde s\in\cS_n$,
\begin{align*}
\cR_n^{\rm DSM}(\tilde s)
-
\cR_n^{\rm DSM}(s)
&=
\int_{\tau_n}^{T_n}
\lVert
\tilde s_t-s_t
\rVert_{L^2(\P_t)}^2
\d t.
\yestag
\label{eq:diffusion_DSM_projection_identity}
\end{align*}
Thus, the population excess denoising-score risk is exactly the
integrated $L^2(\P_t)$ error of the fitted score. We impose this
quantity directly as an output-level condition on the fitted score.

\begin{assumption}[Model assumption 3: integrated score consistency]
\label{ass:diffusion_score_consistency}
The fitted score and reverse diffusion satisfy the following conditions.
\begin{enumerate}[label=(\roman*),itemsep=-.5ex]
\item Conditionally on $\cO$, the fitted score $\hat s_n$ is
diffusion-admissible on $[\tau_n,T_n]$ in the sense of
Definition~\ref{def:diffusion_admissible_score},
$\P_\cO$-almost surely.

\item The integrated score error satisfies
\begin{align*}
\mathcal E_n^{\rm DSM}
:=
\int_{\tau_n}^{T_n}
\lVert
\hat s_{n,t}-s_t
\rVert_{L^2(\P_t)}^2
\d t
=
o_{\P_\cO}(1).
\yestag
\label{eq:diffusion_score_consistency}
\end{align*}
\end{enumerate}
\end{assumption}


Proposition~\ref{prop:diffusion_admissible_regularity} in the appendix
shows that Assumption~\ref{ass:diffusion_score_consistency}(i)
ensures that the learned reverse equation
\eqref{eq:learned_reverse_diffusion} is well posed.
Related score-estimation and sampling guarantees are reviewed in
Section~\ref{sec:diffusion-literature}.

We next use the integrated score error to control the learned terminal
law in KL divergence and thereby establish total-variation consistency.

\begin{lemma}[Score consistency implies total-variation consistency]
\label{lem:diffusion_TV_consistency}
Suppose Assumptions~\ref{ass:diffusion_data_path}
-\ref{ass:diffusion_score_consistency} hold. Then,
realization-wise,
\begin{align*}
D_{\rm KL}(\P_{\tau_n}\|\Q_n)
\le
D_{\rm KL}(\P_{T_n}\|\gamma_p)
+
\mathcal E_n^{\rm DSM}.
\yestag
\label{eq:diffusion_KL_bound}
\end{align*}
Consequently,
\begin{align*}
d_{\rm TV}(\Q_n,\P)
&\le
d_{\rm TV}(\P_{\tau_n},\P)
+
\Bigg\{
\frac12
\Big[
D_{\rm KL}(\P_{T_n}\|\gamma_p)
+
\mathcal E_n^{\rm DSM}
\Big]
\Bigg\}^{1/2},
\yestag
\label{eq:diffusion_TV_decomposition}
\end{align*}
and hence
\[
d_{\rm TV}(\Q_n,\P)
=
o_{\P_\cO}(1).
\]
\end{lemma}

We next transfer the total-variation consistency obtained above to
process convergence under the uniform entropy condition. The following
result is model-independent.

\begin{lemma}[Total-variation consistency implies process convergence under uniform entropy]
\label{lem:TV_uniform_entropy_CWC}
Let $\cF$ be a measurable class with an envelope bounded by $B<\infty$.
Assume that $\cF$ satisfies the uniform entropy condition
\eqref{eq:unif_entropy_F}; that is,
$J(1,\cF,F)<\infty$ with $F\equiv B$. If
\begin{align*}
d_{\rm TV}(\Q_n,\P)
=
o_{\P_\cO}(1),
\end{align*}
then the conditional weak convergence \eqref{eq:CWC} holds.
\end{lemma}

Combining the preceding two lemmas yields the diffusion-bootstrap process
result.

\begin{theorem}[Diffusion bootstrap empirical process convergence]
\label{thm:diffusion_bootstrap_CWC}
Suppose Assumptions~\ref{ass:diffusion_data_path}-\ref{ass:diffusion_score_consistency} hold. Let $\cF$ be a
measurable class with an envelope bounded by $B<\infty$ and satisfying
the uniform entropy condition \eqref{eq:unif_entropy_F}; that is,
$J(1,\cF,F)<\infty$ with $F\equiv B$. Then the diffusion generative
bootstrap process satisfies
\begin{align*}
\sup_{h\in\BL1(\ell^\infty(\cF),\lVert \cdot \rVert_\cF)}
\Big\lvert
\E_{\tilde U}\Big[
h(\tilde\bG_n)
\mid
\cO
\Big]
-
\E\Big[
h(\bG_\P)
\Big]
\Big\rvert
=
o_{\P_\cO}(1).
\end{align*}
\end{theorem}

\begin{remark}[Relation to statistic-specific diffusion-bootstrap theory]
\citet{liang2026diffusion} studied the diffusion bootstrap for OLS in
both fixed-dimensional and proportional high-dimensional regimes.
Their analysis imposes stronger score and regularity conditions
\citep[Assumptions~2.2, 2.5, and 2.6]{liang2026diffusion} to control
moments and the inverse Gram matrix arising in the OLS analysis. For
the process-level result established here, the integrated score error
in \eqref{eq:diffusion_score_consistency} is sufficient.
\end{remark}

\begin{remark}[Numerical discretization]
Theorem~\ref{thm:diffusion_bootstrap_CWC} concerns the continuous-time
reverse diffusion in \eqref{eq:learned_reverse_diffusion}. If a
discretized sampler induces a law $\Q_n^{\rm disc}$ satisfying
\begin{align*}
d_{\rm TV}(\Q_n^{\rm disc},\Q_n)
=
o_{\P_\cO}(1),
\end{align*}
then the same conclusion holds for the corresponding discretized
generative bootstrap process by the triangle inequality and
Lemma~\ref{lem:TV_uniform_entropy_CWC}.
\end{remark}



\subsubsection{Learning-theoretic review and discussion}
\label{sec:diffusion-literature}

The learning theory of score-based diffusion models addresses two closely
related questions: how accurately the score functions can be estimated
from data, and how score-estimation error translates into guarantees for
the generated distribution. We briefly review the former before turning
to the latter, which is more directly related to our setting.

\begin{table}[t]
\centering
\caption{Selected reverse-SDE results under the score-error assumptions of
 each cited theorem. Additional assumptions are abbreviated.}
\label{tab:diffusion-learning-comparison}
\vspace{4pt}
\begingroup
\small
\setlength{\tabcolsep}{4pt}
\setlength{\arrayrulewidth}{0.4pt}
\setlength{\extrarowheight}{2pt}
\renewcommand{\arraystretch}{1.15}
\begin{tabularx}{\textwidth}{@{}
>{\raggedright\arraybackslash}p{0.20\textwidth}|
>{\raggedright\arraybackslash}p{0.235\textwidth}|
>{\centering\arraybackslash}p{0.105\textwidth}|
>{\raggedright\arraybackslash}p{0.215\textwidth}|
>{\raggedright\arraybackslash}X@{}}
\hline
& Key additional assumptions &
Novikov\newline assumed? & Error controlled & Type of guarantee \\[3pt]
\hline
\citet{song2021maximum}\newline Theorem~1\newline Continuous SDE
& Smooth densities; Lipschitz true and fitted scores
& $\checkmark$
& $D_{\rm KL}(\P\|\Q)$
& Non-asymptotic upper bound \\[3pt]
\hline
\citet{chen2023sampling}\newline Theorem~2\newline Discretized SDE
& Finite second moment; $D_{\rm KL}(\P\|\gamma_p)<\infty$; Lipschitz true score
& $\times$
& $d_{\rm TV}(\P,\Q)$
& Non-asymptotic upper bound \\[3pt]
\hline
\citet{chen2023improved}\newline Theorem~2.2\newline Discretized SDE
& Finite second moment; early stopping
& $\times$
& $D_{\rm KL}(\P_\tau\|\Q)$
& Non-asymptotic upper bound \\[3pt]
\hline
\citet{benton2024nearly}\newline Theorem~2, App.~C\newline Discretized SDE
& Finite second moment; early stopping
& $\times$
& $D_{\rm KL}(\P_\tau\|\Q)$
& Non-asymptotic upper bound \\[3pt]
\hline
This subsection\newline Continuous SDE
& Density; finite second moment; admissible fitted score
& $\times$
& $D_{\rm KL}(\P_{\tau_n}\|\Q_n)$;\newline $d_{\rm TV}(\P,\Q_n)$
& Non-asymptotic KL bound; TV consistency \\[3pt]
\hline
\end{tabularx}
\par\smallskip
\begin{minipage}{\textwidth}
\small
Here $\P$ is the target, $\P_\tau$ its forward-noised law at $\tau>0$, and
$\Q$ the sampler output law. All upper bounds quantify initialization and
score errors; discretized samplers also have a time-discretization term.
These bounds imply convergence when their error terms vanish.
For random fitted scores, $\Q$ is conditional on fitting, and convergence
in probability refers to fitting randomness.
\end{minipage}
\endgroup
\end{table}

For score estimation, \citet{oko2023diffusion} analyzed neural denoising
score matching and obtain nearly minimax total-variation rates for
compactly supported Besov densities under positivity and boundary
regularity conditions. \citet{zhang2024minimax} obtained nearly minimax
total-variation rates using a regularized kernel estimator for
sub-Gaussian Sobolev densities with $0<\beta\le2$, without requiring a
density lower bound. \citet{stephanovitch2025generalization} established
nearly minimax $W_1$ rates for neural score-based generators under
H\"older and shape constraints, covering both stochastic and
deterministic samplers. Related score-estimation bounds for specific
distribution classes, neural-network architectures, and training
procedures were developed by
\citet{chen2023score}, \cite{han2024neural}, \cite{fu2026approximation}, among many others.

Given an approximate score, a complementary line of work quantifies the
accuracy of reverse-diffusion sampling. Below is a very brief summary.

\begin{enumerate}[label=(\roman*),itemsep=-.5ex]

\item \citet[Theorem~1]{song2021maximum} bounded the KL divergence from
the target distribution to the generated law by a likelihood-weighted
integrated score error and an initialization error. Their Appendix~A
assumes globally Lipschitz true and fitted scores together with a
Novikov condition.

\item \citet{chen2023sampling} derived total-variation bounds under
$L^2$ score accuracy and globally Lipschitz true scores, without a
log-Sobolev assumption. Their localization argument avoids imposing a
global Novikov condition. \citet{lee2023convergence} established
polynomial-complexity Wasserstein guarantees under support or tail
conditions, as well as total-variation guarantees under additional
smoothness assumptions.

\item \citet[Theorem~2.2]{chen2023improved} obtained KL bounds for a
smoothed target under a finite-second-moment condition and time-averaged
$L^2$ score accuracy, without assuming global Lipschitz regularity of
the true score. Their differential-inequality argument also bypasses a
Novikov condition. \citet{benton2024nearly} further sharpened the
dimension dependence in the sampling complexity using stochastic
localization.

\end{enumerate}

Our result assumes integrated score consistency directly. For the
early-stopped OU sampler, we apply the KL comparison of
\citet{lacker2023hierarchies} under
Assumptions~\ref{ass:diffusion_data_path} and
\ref{ass:diffusion_score_consistency}; see
Section~\ref{sec:diffusion-apdx}. This route requires neither global
Lipschitz regularity of the true score nor a Novikov condition. The
resulting KL bound is relative to the smoothed target
$\P_{\tau_n}$, while total-variation consistency for $\P$ follows in
$\P_\cO$-probability as the smoothing error vanishes.
Table~\ref{tab:diffusion-learning-comparison} compares the main
assumptions and guarantees across these results.

\subsection{Wasserstein generative adversarial networks}
\label{sec:WGAN}

Lastly, we consider Wasserstein generative adversarial networks (WGANs) proposed in 
\cite{arjovsky2017wasserstein}. Unlike the preceding examples, the
argument here requires neither density nor domination assumptions on the
target or learned distribution. Instead, process convergence is obtained
directly from Wasserstein consistency together with additional regularity
of the indexing class $\cF$.

This distinction is particularly relevant when a low-dimensional latent
variable is mapped into a higher-dimensional data space. The generated
law may then be concentrated on a lower-dimensional set while remaining
close to the target distribution in Wasserstein distance.
Wasserstein-based training is naturally suited to such support-mismatched
settings \citep{arjovsky2017towards}.

\subsubsection{Consistency of WGAN bootstrap process}

At the distributional level, we impose only a finite first-moment
condition on the target law.

\begin{assumption}[Data assumption 4: finite first moment]
\label{ass:WGAN_data}
The data distribution $\P$ satisfies
\begin{align*}
\int_{\bR^p}
\lVert z\rVert
\d\P(z)
<
\infty.
\end{align*}
\end{assumption}

To formulate our result, we use the following Lipschitz notation. For a
map $\varphi:\bR^a\to\bR^b$, write
\begin{align*}
{\rm Lip}(\varphi)
:=
\sup_{x\ne y}
\frac{
\lVert \varphi(x)-\varphi(y)\rVert
}{
\lVert x-y\rVert
}.
\end{align*}
Define the anchored unit-Lipschitz class
\begin{align*}
{\rm Lip}_1^0(\bR^p)
:=
\big\{
\varphi:\bR^p\to\bR:
\varphi(0)=0,\
{\rm Lip}(\varphi)\le1
\big\}.
\end{align*}
We now formulate the empirical adversarial criterion and define the
resulting fitted generator--critic pair.


\begin{assumption}[Empirical Wasserstein-GAN criterion and fitted pair]
\label{def:WGAN_generator}
Assume $\P_U$ to be a known noise distribution on $\bR^q$. Let
$U_1,\ldots,U_n$ denote the latent variables used to fit the
Wasserstein GAN, sampled i.i.d.\ from $\P_U$ independently of
$Z_1,\ldots,Z_n$ and included in the fitting information $\cO$. Define
their empirical distribution by
\[
\bP_{U,n}
:=
\frac{1}{n}
\sum_{i=1}^n
\delta_{U_i}.
\]
Let $\cG_n$ be a deterministic class of measurable generators from
$\bR^q$ to $\bR^p$. For any $G\in\cG_n$ and real-valued measurable
function $D:\bR^p\to\bR$, define the empirical adversarial criterion
\begin{align*}
\widehat{\mathcal L}_n(G,D)
&:=
\big(
\bP_n-G\#\bP_{U,n}
\big)[D]
=
\frac{1}{n}
\sum_{i=1}^n
\big\{
D(Z_i)
-
D\big(
G(U_i)
\big)
\big\}.
\yestag
\label{eq:empirical_WGAN_objective}
\end{align*}
Assume $\mathcal D_n$ to be a nonempty deterministic critic class satisfying
\begin{align*}
\mathcal D_n
\subseteq
{\rm Lip}_1^0(\bR^p).
\yestag
\label{eq:WGAN_critic_class}
\end{align*}
During WGAN training, the critic is restricted to
$D\in\mathcal D_n$ and seeks to maximize
$\widehat{\mathcal L}_n(G,D)$, whereas the generator seeks to minimize
the resulting critic value. Lastly, assume the fitting procedure returns an
$\cO$-measurable pair
\begin{align*}
\big(
\hat G_n,\hat D_n
\big)
\in
\cG_n\times\mathcal D_n.
\end{align*}
The induced bootstrap law is then
\[
\Q_n
:=
\hat G_n\#\P_U.
\]
Any additional randomness arising from the numerical fitting algorithm
is included in $\cO$.
\end{assumption}

The criterion in \eqref{eq:empirical_WGAN_objective} is the
finite-latent-sample analogue of the Kantorovich--Rubinstein dual
objective. Appendix~\ref{sec:WGAN-apdx} derives this criterion from the
ideal $W_1$ minimum-distance objective and provides a simple
norm-controlled neural-network design for the generator and critic
classes. When the generator and critic are finitely parameterized neural
networks, \eqref{eq:empirical_WGAN_objective} defines a
finite-dimensional saddle-point problem that is generally solved only
approximately.

In practice, the Lipschitz constant of the critic is commonly controlled
or regularized through weight clipping \citep{arjovsky2017wasserstein},
gradient penalties \citep{gulrajani2017improved}, or architectural norm
constraints such as spectral normalization \citep{miyato2018spectral}.
We abstract away from these implementation details and impose the exact
critic-class condition \eqref{eq:WGAN_critic_class}.

We now impose output-level conditions on the fitted generator--critic
pair.

\begin{assumption}[Model assumption 4: well-trained WGAN]
\label{ass:WGAN_model}
The fitted Wasserstein GAN satisfies the following conditions.
\begin{enumerate}[itemsep=-.5ex,label=(\roman*)]

\item The noise distribution has finite first moment:
\[
\P_U[\lVert U\rVert]
<
\infty.
\]
Moreover, $\hat G_n$ is globally Lipschitz
$\P_\cO$-almost surely.

\item The fitted critic is asymptotically adequate at the fitted
generator:
\begin{align*}
\sup_{\varphi\in{\rm Lip}_1^0(\bR^p)}
\widehat{\mathcal L}_n
\big(
\hat G_n,\varphi
\big)
-
\widehat{\mathcal L}_n
\big(
\hat G_n,\hat D_n
\big)
=
o_{\P_\cO}(1).
\yestag
\label{eq:WGAN_critic_adequacy}
\end{align*}

\item The fitted empirical adversarial value satisfies
\begin{align*}
\Big\lvert
\widehat{\mathcal L}_n
\big(
\hat G_n,\hat D_n
\big)
\Big\rvert
=
o_{\P_\cO}(1).
\yestag
\label{eq:WGAN_empirical_value}
\end{align*}

\item The fitted generator satisfies the latent-sample stability
condition
\begin{align*}
{\rm Lip}(\hat G_n)
W_1
\big(
\bP_{U,n},
\P_U
\big)
=
o_{\P_\cO}(1).
\yestag
\label{eq:WGAN_latent_stability}
\end{align*}

\end{enumerate}
\end{assumption}

Conditions~\eqref{eq:WGAN_critic_adequacy}
and~\eqref{eq:WGAN_empirical_value} are output-level well-trainedness conditions on the fitted generator--critic pair. 
The latent-sample stability condition
\eqref{eq:WGAN_latent_stability} controls the discrepancy between the
distribution generated from the finite latent training sample and that
generated from a fresh latent draw. It allows
${\rm Lip}(\hat G_n)$ to increase with $n$, provided that this increase
is sufficiently slow relative to the approximation of $\P_U$ by
$\bP_{U,n}$.

The following lemma establishes $W_1$-consistency of the learned bootstrap law under these conditions.

\begin{lemma}[Wasserstein consistency of a well-trained WGAN]
\label{lem:WGAN_training_implies_W1}
Suppose Assumptions~\ref{ass:WGAN_data}-\ref{ass:WGAN_model} hold. Then, realization-wise,
\begin{align*}
W_1(\Q_n,\P)
\le
W_1(\bP_n,\P)
+
W_1
\big(
\bP_n,
\hat G_n\#\bP_{U,n}
\big)
+
{\rm Lip}(\hat G_n)
W_1
\big(
\bP_{U,n},
\P_U
\big).
\yestag
\label{eq:WGAN_consistency_decomposition}
\end{align*}
Consequently,
\[
W_1(\Q_n,\P)
=
o_{\P_\cO}(1),
\]
and hence
$\Q_n\rightsquigarrow\P$ in $\P_\cO$-probability.
\end{lemma}



We next transfer Wasserstein consistency to the bootstrap
process under a common concave modulus of continuity for the indexing
class. The following result is model-independent.

\begin{lemma}[Wasserstein consistency for classes with a common modulus]
\label{lem:W1_modulus_CWC}
Assume the following hold true. 
\begin{enumerate}[itemsep=-.5ex,label=(\roman*)]
\item $\cF$ is a measurable class with an envelope function bounded by
some constant $B>0$. 
\item There exists a nondecreasing, concave,
continuous function $\omega:[0,\infty)\to[0,\infty)$ satisfying
$\omega(0)=0$ such that
\begin{align*}
\lvert f(x)-f(y)\rvert
\le
\omega\big(
\lVert x-y\rVert
\big)
\end{align*}
for every $f\in\cF$ and every $x,y\in\bR^p$. 
\item The $\P$-Brownian bridge $\bG_\P$ is a tight, separable, Borel random
element of $\ell^\infty(\cF)$ and that $\cF$ satisfies
\eqref{eq:unif_entropy_F}, with $F\equiv B$ and
$J(1,\cF,F)<\infty$.
\item $W_1(\Q_n,\P)$ is well defined and
$W_1(\Q_n,\P)=o_{\P_\cO}(1)$.
\end{enumerate}
We then have that the conditional weak convergence \eqref{eq:CWC} holds true.
\end{lemma}

The common-modulus assumption is stronger than pointwise continuity but
strictly weaker than a uniform Lipschitz bound. For example, it allows a
finite collection of bounded functions satisfying a common H\"older
condition with exponent in $(0,1)$.

Combining Lemmas \ref{lem:WGAN_training_implies_W1} and \ref{lem:W1_modulus_CWC} yields the following WGAN-bootstrap
process result.

\begin{theorem}[Wasserstein-GAN bootstrap empirical process convergence]
\label{thm:WGAN_bootstrap_CWC}
Suppose Assumptions~\ref{ass:WGAN_data}-\ref{ass:WGAN_model} hold. Let $\cF$ satisfy all the function-class
conditions of Lemma~\ref{lem:W1_modulus_CWC}. Then the
Wasserstein-GAN generative bootstrap process satisfies
\begin{align*}
\sup_{h\in\BL1(\ell^\infty(\cF),\lVert \cdot \rVert_\cF)}
\Big\lvert
\E_{\tilde U}
\Big[
h(\tilde\bG_n)
\mid
\cO
\Big]
-
\E
\Big[
h(\bG_\P)
\Big]
\Big\rvert
=
o_{\P_\cO}(1).
\end{align*}
In particular, the conclusion holds if there exists $L<\infty$ such
that
\[
\sup_{f\in\cF}
{\rm Lip}(f)
\le
L.
\]
\end{theorem}

The preceding result requires neither the target distribution $\P$ nor
the learned distribution $\Q_n$ to admit a Lebesgue density, nor does
it require $\Q_n$ to be dominated by $\P$. This flexibility is
particularly relevant for low-dimensional generators, as illustrated
by the following remark.

\begin{remark}[Low-dimensional generators may remain singular]
\label{rem:WGAN_singular_generator}
Suppose $\P\ll\lambda_p$, the latent dimension satisfies $q<p$, and,
for $\P_\cO$-almost every realization of the fitting information
$\cO$, the fitted generator
$\hat G_n:\bR^q\to\bR^p$ is locally Lipschitz. By
Lemma~\ref{lem:lower_dimensional_Lipschitz_image}, the range
$\hat G_n(\bR^q)$ is Borel and satisfies
\[
\lambda_p\big(
\hat G_n(\bR^q)
\big)
=
0.
\]
Hence,
\[
\P\big(
\hat G_n(\bR^q)
\big)
=
0,
\qquad
\Q_n\big(
\hat G_n(\bR^q)
\big)
=
1.
\]
It follows that $\Q_n\perp\P$ realization-wise and therefore
\begin{align*}
d_{\rm TV}(\Q_n,\P)
=
1,
\qquad
D_{\rm KL}(\Q_n\|\P)
=
D_{\rm KL}(\P\|\Q_n)
=
\infty.
\end{align*}
These conclusions are compatible with
$W_1(\Q_n,\P)=o_{\P_\cO}(1)$, since Wasserstein convergence does not
preserve absolute continuity or support dimension. Indeed,
low-dimensional Lipschitz pushforward distributions can approximate
full-dimensional target distributions arbitrarily closely in $W_1$,
provided that the generator classes are allowed to become increasingly
complex; see
\citet[Proposition~13]{stephanovitch2024optimal}.
\end{remark}




\subsubsection{Learning-theoretic review and discussion}
\label{sec:WGAN-literature}

The learning theory of WGANs studies how the empirical adversarial
objective controls the discrepancy between the target and learned
distributions. The main questions concern critic approximation, other
sources of estimation error, and the effect of structural properties of
the target distribution on convergence rates.

\begin{table}[t]
\centering
\caption{Selected $W_1$ error guarantees for adversarial estimators.
The listed assumptions are abbreviated and refer to the indicated results.}
\label{tab:WGAN-learning-comparison}
\vspace{4pt}
\begingroup
\small
\setlength{\tabcolsep}{4pt}
\setlength{\arrayrulewidth}{0.4pt}
\setlength{\extrarowheight}{2pt}
\renewcommand{\arraystretch}{1.12}
\begin{tabularx}{\textwidth}{@{}
>{\raggedright\arraybackslash}p{0.255\textwidth}|
>{\raggedright\arraybackslash}X|
>{\raggedright\arraybackslash}p{0.285\textwidth}@{}}
\hline
& Key assumptions and model class
& Guarantee for $W_1$ error \\[3pt]
\hline
\citet{huang2022error}\newline
$W_1$ consequence of\newline Theorem~5 and Remark~6
& Target and generated laws on a common cube; absolutely continuous scalar
noise; suitable ReLU capacity and latent sample size
& Rate in expectation, plus optimization error \\[3pt]
\hline
\citet{gao2023approximating}\newline
Theorem~3.8
& Finite third moments; tail and generator-approximation conditions;
GroupSort critics
& Non-asymptotic bound in expectation \\[3pt]
\hline
\citet{stephanovitch2024wasserstein}\newline
Theorem~5.8
& Smooth, nondegenerate manifold generator; tailored neural classes
& Minimax rate in expectation, up to logarithmic factors \\[3pt]
\hline
\citet{tran2026generative}\newline
Section~7.3, Step~2
& Compact data and latent supports; fixed generator norm bounds;
well-trained pair
& Convergence to zero in probability \\[3pt]
\hline
This subsection
& Finite first moments; Lipschitz fitted generator; critic adequacy,
vanishing fitted value, and latent-sample stability
& Realization-wise bound; convergence to zero in probability \\[3pt]
\hline
\end{tabularx}
\par\smallskip
\begin{minipage}{\textwidth}
\small
$W_1$ error means $W_1(\P,\Q_n)$ for the target and learned laws,
with expectations taken over the training samples. 
\end{minipage}
\endgroup
\end{table}

\begin{enumerate}[label=(\roman*),itemsep=-.5ex]

\item Critic approximation determines how well the restricted critic
distance controls the $W_1$ error.
\citet{arora2017generalization} showed that a small neural-network
distance need not imply a small $W_1$ distance.
\citet{biau2021some} studied the gap between a restricted critic
distance and $W_1$, balancing critic-approximation error against
statistical error. One approach to controlling this gap is to
approximate the unit-Lipschitz functions defining $W_1$.
\citet{tanielian2021approximating} derived quantitative GroupSort
approximations on compact domains while preserving the Lipschitz
constraint, thereby providing a means of controlling the
critic-approximation error.

\item The $W_1$ error also depends on generator approximation, finite
data and latent samples, and numerical optimization.
\citet[Lemma~12]{liang2021how} separated generator approximation from
data- and latent-sampling errors in the restricted critic distance.
For ReLU networks, \citet{huang2022error} combined these sources of
error with critic approximation and numerical optimization to obtain
rates for the expected $W_1$ error on compact domains. For GroupSort
critics, \citet{gao2023approximating} used the approximation results
above together with empirical Wasserstein bounds to control the
expected $W_1$ error, with separate contributions from generator
approximation, critic approximation, data sampling, and latent sampling
under their moment, tail, and generator-approximation conditions.

\item Structural properties of the target distribution affect the rate
at which the expected $W_1$ error decreases with the training sample
size. \citet[Theorem~1]{schreuder2021statistical} bounded the expected
$W_1$ error for minimum-$W_1$ estimation in noiseless Lipschitz
pushforward models by the sum of a generator-approximation error and a
sampling term whose rate depends on the latent dimension rather than the
ambient data dimension. Their target distribution need not admit a
Lebesgue density on $\bR^p$. For smooth, nondegenerate manifold models,
\citet{stephanovitch2024wasserstein} exploited smoothness and intrinsic
dimension to construct neural generators and critics attaining minimax
rates for the expected $W_1$ error, up to logarithmic factors.

\end{enumerate}

In our framework, the approximation and optimization requirements above
are summarized through conditions on the fitted generator--critic pair,
while the data- and latent-sampling errors are controlled separately.
As in \citet[Section~5.1]{tran2026generative}, critic adequacy together
with a vanishing empirical adversarial value ensures
\[
W_1(\bP_n,\hat G_n\#\bP_{U,n})
=
o_{\P_\cO}(1).
\]
Section~\ref{sec:WGAN-apdx} provides sufficient conditions in terms of
generator approximation, critic approximation, and numerical
optimization. Together with these fitted-pair conditions, finite first
moments and \eqref{eq:WGAN_latent_stability} yield qualitative $W_1$
consistency. The latent-sample argument uses Wasserstein contraction,
as also used in Equation~(3.4) of \citet{gao2023approximating}, in place
of the compact-support uniform-law argument of
\citet[Section~7.3]{tran2026generative}. This allows unbounded data and
latent supports without requiring a smooth pushforward or manifold
model. Table~\ref{tab:WGAN-learning-comparison} compares selected
theoretical guarantees.

\section*{Acknowledgments}

The authors acknowledge the use of ChatGPT (OpenAI) in the preparation of this manuscript. In addition to assistance with writing and exposition, ChatGPT was used as a tool for exploring proof strategies, scrutinizing assumptions and technical arguments, and identifying connections with related theoretical results. These interactions informed the development and refinement of parts of the manuscript. All mathematical results, arguments, and references were independently verified by the authors, who take full responsibility for the content of the paper.


\appendix

\section{Proofs of the main results}

\subsection{Proofs of the results in Section~\ref{sec:CWC_characterization}}

\subsubsection{Proof of Theorem~\ref{thm:cvg_emp_proc}}

\begin{proof}
We first prove necessity. Suppose \eqref{eq:CWC} holds. For any finite set of fixed functions $f_1, \ldots, f_k \in \cF$, define the coordinate projection $\pi: \ell^\infty(\cF) \to \bR^k$ by
\begin{align*}
    \pi(x) := (x(f_1), \ldots, x(f_k)).
\end{align*}
Then for any $h' \in \BL1(\bR^k,\lVert \cdot\rVert_\infty)$, $h:= h' \circ \pi \in \BL1(\ell^\infty(\cF),\lVert \cdot \rVert_\cF)$, as
\begin{align*}
    \sup_{x_1 \neq x_2} \frac{\lvert h(x_1) - h(x_2) \rvert}{\lVert x_1 - x_2 \rVert_{\cF}}  = \sup_{x_1 \neq x_2} \frac{\lvert h' \circ \pi(x_1) - h' \circ \pi(x_2) \rvert}{\lVert x_1 - x_2 \rVert_{\cF}} 
    & \le \sup_{x_1 \neq x_2} \frac{\lVert \pi(x_1) - \pi(x_2) \rVert_{\infty}}{\lVert x_1 - x_2 \rVert_{\cF}} \\
    & \le \sup_{x_1 \neq x_2} \frac{\max_{1\le j \le k}\lvert x_1(f_j) - x_2(f_j) \rvert}{\lVert x_1 - x_2 \rVert_{\cF}} \le 1;
\end{align*}
we then have 
\begin{align*}
    & \sup_{h' \in \BL1(\bR^k,\lVert \cdot\rVert_\infty)} \Big\lvert \E_{\tilde U}\big[h'(\tilde \bG_n[f_1], \ldots, \tilde \bG_n[f_k])\mid \cO\big]
    - \E\big[h'(\bG_\P[f_1], \ldots, \bG_\P[f_k])\big] \Big\rvert \\ 
    & \le \sup_{h \in \BL1(\ell^\infty(\cF),\lVert \cdot \rVert_\cF)} \Big\lvert \E_{\tilde U}[h(\tilde \bG_n)\mid \cO] - \E[h(\bG_\P)] \Big\rvert = o_{\P_\cO}(1).
\end{align*}
To prove conditional asymptotic $\rho_{\P}$-equicontinuity, for any $x \in \ell^\infty(\cF)$ and fixed $\delta>0$, denote 
\begin{align*}
    \omega_{\delta}(x) := \sup_{\rho_{\P}(f,g)<\delta} \lvert x(f) - x(g) \rvert.
\end{align*}
We note $\omega_{\delta}$ is Lipschitz with respect to the sup norm on $\ell^\infty(\cF)$: for any $x_1,x_2 \in \ell^\infty(\cF)$ such that $x_1(f)-x_2(f) = (x_1-x_2)(f)$,
\begin{align*}
    \big\lvert \omega_{\delta}(x_1) - \omega_{\delta}(x_2) \big\rvert 
    &\le \sup_{\rho_{\P}(f,g)<\delta} \big\lvert \{ x_1(f) - x_1(g) \} - \{x_2(f) - x_2(g)\} \big\rvert \\ 
    & \le 
    \sup_{\rho_{\P}(f,g)<\delta}\big\{ \big\lvert x_1(f) - x_2(f) \big\rvert
    +
    \big\lvert x_1(g) - x_2(g) \big\rvert \big\}
    \le 2 \big\lVert x_1-x_2 \big\rVert_{\cF}.
\end{align*}
Moreover, we can construct $\psi_{\epsilon}:\bR_+ \to [0,1]$ such that
\begin{align*}
    \ind\big(t \ge \epsilon\big)
    \le
    \psi_{\epsilon}(t) 
    := \ind(t \ge \epsilon) 
    + \Big\{\frac{2t}{\epsilon}-1\Big\} \ind\big( \frac{\epsilon}{2}<t<\epsilon \big)
    \le 
    \ind\Big(t \ge \frac{\epsilon}{2}\Big).
\end{align*}
This way, we define $h_{\epsilon, \delta} := \psi_{\epsilon} \circ \omega_{\delta}$, which is a $1$-bounded Lipschitz function with
\begin{align*}
\sup_{x_1 \neq x_2} \frac{\lvert h_{\epsilon, \delta}(x_1) - h_{\epsilon, \delta}(x_2) \rvert}{\big\lVert x_1-x_2 \big\rVert_{\cF} } 
\le 
\frac{4}{\epsilon}.
\end{align*}
By scaling and re-scaling $h_{\epsilon, \delta}$ while applying \eqref{eq:CWC} gives
\begin{align*}
    \P_{\tilde U}
    \big( \omega_{\delta}(\tilde \bG_n) > \epsilon \mid \cO \big) 
    \le
    \E_{\tilde U} \big[ h_{\epsilon, \delta}(\tilde \bG_n) \mid \cO \big] = \E \big[ h_{\epsilon, \delta}(\bG_{\P}) \big] + o_{\P_\cO}(1).
    \yestag \label{eq:thm1,h_epdeltaWC}
\end{align*}
For the right hand side, by Example 1.5.10 of \cite{MR1385671}, the tight Gaussian random element $\bG_{\P}$ in $\ell^{\infty}(\cF)$ has almost all paths uniformly $\rho_2$-continuous, i.e.
\begin{align*}
    \sup_{\rho_2(f,g)<\delta} \big\lvert \bG_{\P}[f] - \bG_{\P}[g] \big\rvert \stackrel{\sf a.s.}{\longrightarrow} 0 ~\text{ as }~ \delta \downarrow 0;
\end{align*}
where
\begin{align*}
    \rho_2^2(f,g) 
    &:= \E \big[(\bG_{\P}[f] - \bG_{\P}[g])^2\big]
    = \Var \big[\bG_{\P}[f] - \bG_{\P}[g]\big] \\
    &= \Var_{\P}[f(Z)-g(Z)] 
    = \P\big[(f-g)^2\big] - (\P[(f-g)])^2 
    = \rho_{\P}^2(f,g);
\end{align*}
then $\bG_{\P}$ also has almost all paths uniformly $\rho_{\P}$-continuous:
\begin{align*}
    \omega_{\delta}( \bG_{\P}) = \sup_{\rho_{\P}(f,g)<\delta} \big\lvert \bG_{\P}[f] - \bG_{\P}[g] \big\rvert \stackrel{\sf a.s.}{\longrightarrow} 0 ~\text{ as }~ \delta \downarrow 0;
\end{align*}
which implies 
\begin{align*}
    \E \big[ h_{\epsilon, \delta}(\bG_{\P}) \big] 
    \le
    \P \big(\omega_{\delta}(\bG_{\P}) \ge \epsilon/2 \big) \to 0 ~\text{ as }~ \delta \downarrow 0.
    \yestag
    \label{eq:thm1,Gp_unif_cont_prob}
\end{align*}
Combining \eqref{eq:thm1,h_epdeltaWC}-\eqref{eq:thm1,Gp_unif_cont_prob} and applying the union bound to probability quantity in the left hand side of \eqref{eq:CAE} then yields the desired result.

For the opposite direction, assume \eqref{eq:FD-CWC} and \eqref{eq:CAE} hold. Note tightness of $\bG_{\P}$ in $\ell^{\infty}(\cF)$ also implies $(\cF, \rho_{\P})$ is totally bounded. For any $\delta>0$, we can take a finite $\rho_{\P}$-net $\{f_1, \ldots, f_N\}$ for $\cF$ such that an associated $\pi_{\delta}$ that maps $f$ to a element in $\rho_{\P}$-net with $\rho_{\P}(f,\pi_{\delta}f)<\delta$ is defined. Define the projection operator $\Pi_\delta: \ell^\infty(\cF) \to \ell^\infty(\cF)$ with $(\Pi_\delta x)[f] := x(\pi_{\delta}f)$; then 
\begin{align*}
    \big\lvert \E_{\tilde U}\big[h(\tilde \bG_n)\mid \cO\big] - \E\big[h(\bG_\P)\big] \big\rvert  
    &\le 
    \big\lvert \E_{\tilde U}\big[h(\tilde \bG_n)\mid \cO\big] - \E_{\tilde U}\big[h(\Pi_\delta \tilde \bG_n)\mid \cO\big] \big\rvert \\
    &~~ + 
    \big\lvert \E_{\tilde U}\big[h(\Pi_\delta \tilde \bG_n)\mid \cO\big]  - \E\big[h(\Pi_\delta \bG_\P)\big] \big\rvert \\
    &~~ + 
    \big\lvert  \E\big[h(\Pi_\delta \bG_\P)\big] - \E\big[h(\bG_\P)\big]  \big\rvert \\
    & := A_{n,\delta}(h) + B_{n,\delta}(h) + C_{\delta}(h).
    \yestag
    \label{eq:thm1,moment_tri_ineq}
\end{align*}
For the first term $A_{n,\delta}(h)$ in \eqref{eq:thm1,moment_tri_ineq}, by the definition of $h$, we have for any $t>0$,
\begin{align*}
    A_{n,\delta}(h) 
    \le 
    \E_{\tilde U}\big[ 2 \wedge \lVert \tilde \bG_n - \Pi_\delta \tilde \bG_n \rVert_{\cF} \mid \cO\big]
    \le 
    \E_{\tilde U}\big[ 2 \wedge \omega_{\delta}(\tilde \bG_n) \mid \cO \big] 
    \le 
    t + 2 \P_{\tilde U}\big( \omega_{\delta}(\tilde \bG_n) > t \mid \cO \big);
\end{align*}
then $A_{n,\delta}(h)$ can be made arbitrarily small with sufficiently large $n$ and sufficiently small $\delta$, uniformly for $h \in \BL1(\ell^\infty(\cF),\lVert \cdot \rVert_\cF)$.
For the second term $B_{n,\delta}(h)$ in \eqref{eq:thm1,moment_tri_ineq}, we can write
\begin{align*}
    h(\Pi_\delta \tilde \bG_n) = h \circ \pi_{N} (\tilde \bG_n[f_1], \ldots, \tilde \bG_n[f_N]),
    \yestag
    \label{eq:thm1,proj_to_finite}
\end{align*}
where $\pi_{N}: \bR^N \to \ell^\infty(\cF)$ and
\begin{align*}
    (\pi_{N}y)(f) := y_j ~\text{ if }~ \pi_{\delta}f = f_j.
\end{align*}
Note $h \circ \pi_{N} \in \BL1(\bR^N,\lVert \cdot\rVert_\infty)$ as for any $y_1 \neq y_2 \in \bR^N$ one has
\begin{align*}
    \big\lvert h \circ (\pi_{N}y_1) - h \circ (\pi_{N}y_2) \big\rvert 
    \le 
    \lVert \pi_{N}y_1 - \pi_{N}y_2 \rVert_{\cF} 
    \le 
    \lVert y_1 - y_2 \rVert_{\infty};
\end{align*}
then a similar result representation as \eqref{eq:thm1,proj_to_finite} for $h(\Pi_\delta \bG_\P)$ and \eqref{eq:FD-CWC} implies
\begin{align*}
    \sup_{h \in \BL1(\ell^\infty(\cF),\lVert \cdot \rVert_\cF)} B_{n,\delta}(h) = o_{\P_\cO}(1).
\end{align*}
For the third term $C_{\delta}(h)$ in \eqref{eq:thm1,moment_tri_ineq}, similarly we have uniformly for $h \in \BL1(\ell^\infty(\cF),\lVert \cdot \rVert_\cF)$, 
\begin{align*}
    C_{\delta}(h) 
    \le 
    \E \big[ 2 \wedge \lVert \bG_\P - \Pi_\delta \bG_\P \rVert_{\cF} \big] 
    \le 
    \E \big[ 2 \wedge \omega_{\delta}(\bG_\P)  \big], 
\end{align*}
then path uniform $\rho_\P$-continuity of $\bG_{\P}$ and Lebesgue dominated convergence theorem implies $C_{\delta}(h)$ can be arbitrarily small when $\delta \downarrow 0$. Combining these gives the desired \eqref{eq:CWC}.
\end{proof}

We next record the finite-dimensional reduction used in
Section~\ref{sec:CWC_characterization} to replace
\eqref{eq:FD-CWC} by covariance consistency. The result is stated under
a second-moment tail condition, which also makes explicit how the
bounded-envelope restriction adopted in the main text can be relaxed.

For a function class $\cF$, define its finite linear span by
\begin{align*}
\operatorname{span}(\cF)
:=
\Big\{
\sum_{j=1}^k
a_jf_j
:
k\in\bN,\ 
a_1,\ldots,a_k\in\bR,\ 
f_1,\ldots,f_k\in\cF
\Big\}.
\end{align*}

\begin{proposition}[Covariance characterization of finite-dimensional conditional weak convergence]
\label{prop:FD_cov_characterization}
Suppose $\cF\subset L^2(\P)$ and, for every
$h\in\operatorname{span}(\cF)$ and every $n$,
$\Q_n[h^2]<\infty$ $\P_\cO$-almost surely. Assume further that, for every
$h\in\operatorname{span}(\cF)$ and every $\eta>0$,
\begin{align*}
\lim_{K\to\infty}
\limsup_{n\to\infty}
\P_\cO
\Big(
\Q_n
\big\{ (h-\Q_n[h])^2
\ind\big(
\big\lvert h-\Q_n[h] \big\rvert > K \big)
\big\}
> \eta
\Big)
=
0.
\yestag
\label{eq:conditional_UI2}
\end{align*}
Then the finite-dimensional conditional weak convergence condition
\eqref{eq:FD-CWC} holds if and only if the covariance consistency condition
\eqref{eq:FD-cov} holds. In particular, this equivalence holds whenever
$\cF$ has an envelope bounded by a finite constant.
\end{proposition}

\begin{proof}
We first derive a conditional characteristic-function expansion. Fix
$h\in\operatorname{span}(\cF)$ and define
\begin{align*}
Y_{n,h} := h(\tilde Z)-\Q_n[h],
\qquad
v_{n,h}:= \Q_n[Y_{n,h}^2].
\end{align*}
For every $K>0$,
\begin{align*}
v_{n,h}
\le
K^2
+
\Q_n
\big\{
Y_{n,h}^2
\ind\big( \lvert Y_{n,h}\rvert>K \big)
\big\}.
\end{align*}
It follows from \eqref{eq:conditional_UI2} that
$v_{n,h}=O_{\P_\cO}(1)$.

Let $r(u) := \exp(iu)-1-iu+{u^2}/{2}$.
There exists a constant $C<\infty$ such that 
$\lvert r(u)\rvert
\le
C u^2 (1\wedge\lvert u\rvert )$ for every $u\in\bR$.
For fixed $t\in\bR$, set
\begin{align*}
R_{n,h}(t)
:= \Q_n
\Big\{
r\Big(
\frac{tY_{n,h}}{\sqrt n}
\Big)
\Big\}.
\end{align*}
Then for every $K>0$,
\begin{align*}
n
\lvert R_{n,h}(t)\rvert
&\le
Ct^2
\Q_n
\big\{
Y_{n,h}^2
\big(
1\wedge {\lvert tY_{n,h}\rvert}/{\sqrt n}
\big)
\big\}
\\
&\le
Ct^2
\Q_n
\big\{
Y_{n,h}^2
\big(
1\wedge {\lvert t\rvert K}/{\sqrt n}
\big) \ind(\lvert Y_{n,h}\rvert \le K)
+
Y_{n,h}^2
\big(
1\wedge {\lvert tY_{n,h}\rvert}/{\sqrt n}
\big) \ind(\lvert Y_{n,h}\rvert>K)
\big\}
\\
&\le
Ct^2
\Big(
\frac{\lvert t\rvert K}{\sqrt n}
v_{n,h}
+
\Q_n
\big\{
Y_{n,h}^2
\ind(\lvert Y_{n,h}\rvert>K)
\big\}
\Big).
\end{align*}
First letting $n\to\infty$ and then $K\to\infty$, using
\eqref{eq:conditional_UI2}, gives
\begin{align*}
nR_{n,h}(t)
=
o_{\P_\cO}(1).
\yestag
\label{eq:conditional_CF_remainder}
\end{align*}
Since $\Q_n[Y_{n,h}]=0$, conditionally on $\cO$ one has
\begin{align*}
\E_{\tilde U}
\big[
\exp\{ it\tilde\bG_n[h]\}
\mid \cO
\big]
& =
\Big(
\Q_n
\Big\{
\exp\Big(
\frac{itY_{n,h}}{\sqrt n}
\Big)
\Big\}
\Big)^n
\\
& =
\Big( 
1 
+ \frac{it}{\sqrt{n}} \Q_n[Y_{n,h}]
- \frac{t^2}{2n} \Q_n[Y_{n,h}^2]
+ R_{n,h}(t)
\Big)^n
\\
& =
\Big( 
1 - \frac{t^2v_{n,h}}{2n}
+ R_{n,h}(t)
\Big)^n
=
\exp\Big(
-
\frac{t^2v_{n,h}}{2}
\Big)
+
o_{\P_\cO}(1),
\yestag
\label{eq:conditional_CF_expansion}
\end{align*}
where the last equality follows from
$v_{n,h}=O_{\P_\cO}(1)$ and
\eqref{eq:conditional_CF_remainder}. Note conditionally on $\cO$, the leading term on the right hand side is the characteristic function of ${\rm N}(0,v_{n,h})$.

Now, suppose first that \eqref{eq:FD-cov} holds. Fix
$k\in\bN$ and $f_1,\ldots,f_k\in\cF$, and write
\begin{align*}
S_n
:=
\big(
\tilde\bG_n[f_1],\ldots,\tilde\bG_n[f_k]
\big)^\top,
\qquad
G:=
\big(
\bG_\P[f_1],\ldots,\bG_\P[f_k]
\big)^\top,
\end{align*}
with 
\begin{align*}
\Sigma_n
:=
\big\{
\Q_n[f_if_j]-\Q_n[f_i]\Q_n[f_j]
\big\}_{i,j=1}^k,
\qquad
\Sigma :=
\big\{
\P[f_if_j]-\P[f_i]\P[f_j]
\big\}_{i,j=1}^k.
\end{align*}
Then $\Sigma_n=\Sigma+o_{\P_\cO}(1)$ entrywise. For any
$a=(a_1,\ldots,a_k)^\top\in\bR^k$, write $h_a := \sum_{j=1}^k a_jf_j$.
Since $a^\top S_n=\tilde\bG_n[h_a]$ and
$v_{n,h_a}=a^\top\Sigma_na$, \eqref{eq:conditional_CF_expansion} and continuous mapping theorem gives,
for every $t\in\bR$,
\begin{align*}
\E_{\tilde U}
\big[
\exp\{ it a^\top S_n\}
\mid 
\cO
\big]
=
\exp\Big(
-
\frac{t^2}{2}
a^\top\Sigma_n a
\Big)
+
o_{\P_\cO}(1)
=
\exp\Big(
-
\frac{t^2}{2}
a^\top\Sigma a
\Big)
+
o_{\P_\cO}(1).
\yestag
\label{eq:conditional_CW_characteristic_convergence}
\end{align*}
To conclude \eqref{eq:FD-CWC}, take an arbitrary subsequence. By
Lemma~\ref{lemma:subseq_prob_cvg} and a diagonal extraction over a
countable dense subset of $\bR^k$, there exists a further subsequence along
which $\Sigma_n\to\Sigma$ and the conditional characteristic functions in
\eqref{eq:conditional_CW_characteristic_convergence} converge almost surely
on that dense subset. Along this further subsequence,
\begin{align*}
\E_{\tilde U}
\big[
\lVert S_n\rVert^2
\mid
\cO
\big]
=
\tr(\Sigma_n)
\end{align*}
is almost surely bounded. Hence the conditional laws of $S_n$ are tight.
Every subsequential weak limit of the conditional laws along this further subsequence has the characteristic function
of ${\rm N}(0,\Sigma)$ on a dense subset, and therefore everywhere. Thus,
for almost every realization along the further subsequence, the conditional
law of $S_n$ converges weakly to ${\rm N}(0,\Sigma)$. Since the
bounded-Lipschitz distance metrizes weak convergence on $\bR^k$, the
bounded-Lipschitz distance in \eqref{eq:FD-CWC} converges to zero almost
surely along this further subsequence. Applying
Lemma~\ref{lemma:subseq_prob_cvg} again proves \eqref{eq:FD-CWC}.

Conversely, suppose \eqref{eq:FD-CWC} holds. Fix
$h\in\operatorname{span}(\cF)$, and choose a representation $h = \sum_{j=1}^k a_jf_j$
with $a=(a_1,\ldots,a_k)^\top\in\bR^k$ and
$f_1,\ldots,f_k\in\cF$. With $S_n$ and $G$ defined as above,
\eqref{eq:FD-CWC} implies, for every $t\in\bR$,
\begin{align*}
\E_{\tilde U}
\big[
\exp\{ it\tilde\bG_n[h]\}
\mid \cO
\big] 
&= 
\E_{\tilde U}
\big[
\exp\big\{ 
 it a^\top S_n
\big\}
\mid
\cO
\big]
\\
&=
\E_{\tilde U}
\big[
\cos (t a^\top S_n)
\mid \cO
\big]
+
i \E_{\tilde U}
\big[
\sin (t a^\top S_n)
\mid \cO
\big]
\\
&=
\E
\big[
\cos (t a^\top G)
\big]
+
o_{\P_\cO}(1)
+
i \big \{ 
\E
\big[
\sin (t a^\top G)
\big]
+
o_{\P_\cO}(1)
\big\}
\\
&=
\E
\big[
\exp\{ it a^\top G \}
\big]
+
o_{\P_\cO}(1).
\yestag
\label{eq:FD_CWC_implies_CF}
\end{align*}
In the third equality here, the real and imaginary parts follow by applying
\eqref{eq:FD-CWC} to suitably rescaled cosine and sine functions. Writing $v_h := \P[\{ h-\P[h]\}^2]$,
one has $a^\top G \sim {\rm N}(0,v_h)$ and 
\begin{align*}
\E \big[ \exp\{  it a^\top G \} \big]
=
\exp\Big(
- \frac{t^2v_h}{2}
\Big).
\end{align*}
Combining \eqref{eq:conditional_CF_expansion} and
\eqref{eq:FD_CWC_implies_CF} gives
\begin{align*}
\exp\Big(
-
\frac{t^2v_{n,h}}{2}
\Big)
=
\exp\Big(
-
\frac{t^2v_h}{2}
\Big)
+
o_{\P_\cO}(1).
\end{align*}
For any fixed $t\ne0$, continuity of the inverse map
$x\mapsto-2t^{-2}\log x$ on $(0,1]$ therefore yields
\begin{align*}
\Var_{\Q_n}(h)
=
v_{n,h}
=
v_h
+
o_{\P_\cO}(1)
=
\Var_\P(h)
+
o_{\P_\cO}(1).
\yestag
\label{eq:span_variance_consistency}
\end{align*}
Applying \eqref{eq:span_variance_consistency} to $h=f$, $g$, and $f+g$, and using the polarization identity, gives
\begin{align*}
\Q_n[fg]-\Q_n[f]\Q_n[g] 
&=
\frac{1}{2}
\{
\Var_{\Q_n}(f+g)
-
\Var_{\Q_n}(f)
-
\Var_{\Q_n}(g)
\}
\\
&=
\frac{1}{2}
\{
\Var_\P(f+g)
-
\Var_\P(f)
-
\Var_\P(g)
\}
+
o_{\P_\cO}(1)
\\
&=
\P[fg]-\P[f]\P[g]
+
o_{\P_\cO}(1),
\end{align*}
which is \eqref{eq:FD-cov}.

Finally, suppose $\cF$ has an envelope bounded by $B<\infty$. For every
$h=\sum_{j=1}^k a_jf_j\in\operatorname{span}(\cF)$,
\begin{align*}
\big\lvert
h-\Q_n[h]
\big\rvert
\le
2B
\sum_{j=1}^k
\lvert a_j\rvert.
\end{align*}
Thus \eqref{eq:conditional_UI2} holds. This proves the final
assertion.
\end{proof}

Condition~\eqref{eq:conditional_UI2} is formulated directly in terms of
finite linear combinations of functions in $\cF$. A more convenient envelope-level sufficient
condition is the existence of a measurable envelope $F$ of $\cF$ such that
\begin{align*}
\P[F^2]
<
\infty,
\qquad
\Q_n[F^2]
<
\infty
\quad
\P_\cO\text{-almost surely for every }n,
\end{align*}
and, for every $\eta>0$,
\begin{align*}
\lim_{K\to\infty}
\limsup_{n\to\infty}
\P_\cO
\Big\{
\Q_n
\big[
F^2
\ind(F>K)
\big]
>
\eta
\Big\}
=
0.
\yestag
\label{eq:envelope_UI2}
\end{align*}
To see this, fix
$h=\sum_{j=1}^k a_jf_j\in\operatorname{span}(\cF)$ and write
$A_h:=\sum_{j=1}^k\lvert a_j\rvert$. Then $\lvert h\rvert\le A_hF$, so
\eqref{eq:envelope_UI2} implies
$V_{n,h}:=\Q_n[h^2]=O_{\P_\cO}(1)$. On the event
$\{V_{n,h}\le M\}$, one has
$\lvert\Q_n[h]\rvert\le\sqrt M$. Hence for $K>2\sqrt M$,
\begin{align*}
\big\{
\lvert h-\Q_n[h]\rvert>K
\big\}
\subseteq
\big\{
\lvert h\rvert>K/2
\big\}.
\end{align*}
Moreover, by Markov's inequality,
\begin{align*}
& \Q_n
\big[
\{h-\Q_n[h]\}^2
\ind\big(
\lvert h-\Q_n[h] \rvert
>K \big)
\big]
\\
&~~ \le
2 \Q_n
\big[
h^2
\ind\big(
\lvert h-\Q_n[h] \rvert
>K \big)
\big]
+
2 \Q_n [h^2]
\Q_n\big(
\lvert h-\Q_n[h] \rvert >K \big)
\\
&~~ \le
2
\Q_n
\big[
h^2
\ind\big(
\lvert h\rvert> K/2
\big)
\big]
+
\frac{8M^2}{K^2}
\\
&~~ \le
2A_h^2
\Q_n
\big[
F^2
\ind\big( F>{K}/{2A_h} \big)
\big]
+
\frac{8M^2}{K^2},
\end{align*}
with the case $A_h=0$ being trivial. First choose $M$ so that
$\P_\cO(V_{n,h}>M)$ is asymptotically small, and then let $K\to\infty$.
This proves \eqref{eq:conditional_UI2}. Thus, the bounded-envelope assumption used in the main text may be
replaced by
\eqref{eq:envelope_UI2}. The condition $\P[F^2]<\infty$ alone is not
sufficient when the random laws $\Q_n$ vary with $n$; the corresponding
second-moment tails must also be controlled under $\Q_n$.

\subsection{Proofs of the results in Section~\ref{sec:suff_cond_FD_cov}}

For any real-valued function $h$, let $D_h$ denote the set of discontinuity points of $h$.

\subsubsection{Proof of Lemma~\ref{lem:continuous_approximation_principle}}

\begin{proof}
Fix $h\in\cG$ and let $\{h_m:m\ge1\}$  be the sequence of described bounded continuous functions that satisfies
\eqref{eq:CAP_P_approximation} and
\eqref{eq:CAP_Qn_approximation}. For every $m\ge1$,
\begin{align*}
\lvert(\Q_n-\P)h\rvert
\le
\lvert(\Q_n-\P)h_m\rvert
+
\Q_n[\lvert h-h_m\rvert]
+
\P[\lvert h-h_m\rvert].
\end{align*}
Hence, for every $\epsilon>0$ and every fixed $m$,
\begin{align*}
&\P_\cO\big(
\lvert(\Q_n-\P)h\rvert>\epsilon
\big)
\\
&\quad\le
\P_\cO\Big(
\lvert(\Q_n-\P)h_m\rvert>\frac{\epsilon}{3}
\Big)
+
\P_\cO\Big(
\Q_n[\lvert h-h_m\rvert]>\frac{\epsilon}{3}
\Big)
+
\ind\Big\{
\P[\lvert h-h_m\rvert]>\frac{\epsilon}{3}
\Big\}.
\end{align*}
For every fixed $m$, weak convergence of the random measures gives
$(\Q_n-\P)h_m=o_{\P_\cO}(1)$. Therefore,
\begin{align*}
\limsup_{n\to\infty}
\P_\cO\Big(
\lvert(\Q_n-\P)h\rvert>\epsilon
\Big)
\le
\limsup_{n\to\infty}
\P_\cO\Big(
\Q_n[\lvert h-h_m\rvert]>\frac{\epsilon}{3}
\Big)
+
\ind\Big\{
\P[\lvert h-h_m\rvert]>\frac{\epsilon}{3}
\Big\}.
\end{align*}
The left-hand side does not depend on $m$. Letting $m\to\infty$ and using
\eqref{eq:CAP_P_approximation}-\eqref{eq:CAP_Qn_approximation} shows that it
equals zero. Thus $(\Q_n-\P)h=o_{\P_\cO}(1)$.
\end{proof}

\subsubsection{Proof of Theorem~\ref{thm:FD-cov}}

\begin{proof}
We first prove the basic implication. Define
\begin{align*}
\cG_\cF
:=
\cF\cup\{fg:f,g\in\cF\},
\end{align*}
then that the finite-dimensional covariance consistency condition \eqref{eq:FD-cov} holds if for every $h\in\cG_\cF$, 
\begin{align*}
(\Q_n-\P)h=o_{\P_\cO}(1).
\yestag
\label{eq:pointwise_moment_consistency}
\end{align*}
As one observe, fix any $f,g\in\cF$, write $\Gamma_\Q(f,g) := \Q[fg]-\Q[f]\Q[g]$
for any probability measure $\Q$, then
\begin{align*}
\big\lvert \Gamma_{\Q_n}(f,g)-\Gamma_\P(f,g)\big\rvert
&\le
\big\lvert(\Q_n-\P)[fg]\big\rvert
+
\big\lvert \Q_n[f]\Q_n[g]-\P[f]\P[g]\big\rvert
\\
&\le
\big\lvert(\Q_n-\P)[fg]\big\rvert
+
\big\lvert\Q_n[f]\big\rvert
\big\lvert(\Q_n-\P)[g]\big\rvert
+
\big\lvert\P[g]\big\rvert
\big\lvert(\Q_n-\P)[f]\big\rvert
\\
&\le
\big\lvert(\Q_n-\P)[fg]\big\rvert
+
B\big\lvert(\Q_n-\P)[g]\big\rvert
+
B\big\lvert(\Q_n-\P)[f]\big\rvert,
\end{align*}
where $B<\infty$ satisfies $\sup_{f\in\cF}
\lvert f(z)\rvert \le B$ for every $z\in\bR^p$. Thus \eqref{eq:pointwise_moment_consistency} implies \eqref{eq:FD-cov}.

Suppose condition~(i) holds, then $\Q_n \rightsquigarrow \P$ in $\P_\cO$-probability. Let $h$ be a bounded measurable function such that $\P(D_h)=0$. Then
\begin{align*}
    (\Q_n-\P)h=o_{\P_\cO}(1).
    \yestag
    \label{eq:bounded_P_as_cont_integral_cvg}
\end{align*}
To see this, take any subsequence of $n$. 
Since $\Q_n\rightsquigarrow\P$ in $\P_\cO$-probability, a diagonal
application of Lemma~\ref{lemma:subseq_prob_cvg} to a countable
convergence-determining class of bounded continuous functions yields a
further subsequence along which
$\Q_n\rightsquigarrow\P$ almost surely. For almost every realization along this further subsequence, Lemma~\ref{lemma:portmanteau_as_continuous} gives
\begin{align*}
    \Q_n h\to \P h.
\end{align*}
Thus every subsequence admits a further subsequence along which $(\Q_n-\P)h\to0$ almost surely. Applying the subsequence criterion again proves \eqref{eq:bounded_P_as_cont_integral_cvg}.

For any $f,g\in\cF$. By assumption, $\P(D_f)=\P(D_g)=0$. Since $f$ and $g$ are bounded, the product $fg$ is continuous at every point at which both $f$ and $g$ are continuous. Therefore $D_{fg}\subset D_f\cup D_g$ and hence $\P(D_{fg})=0$.
Applying \eqref{eq:bounded_P_as_cont_integral_cvg} gives
\begin{align*}
\big\lvert(\Q_n-\P)f\big\rvert
\vee \big\lvert(\Q_n-\P)g\big\rvert
\vee \big\lvert(\Q_n-\P)(fg)\big\rvert=o_{\P_\cO}(1),
\end{align*}
which proves \eqref{eq:pointwise_moment_consistency} and hence \eqref{eq:FD-cov}.

Now suppose condition~(ii) holds. Fix $h\in\cG_\cF$ and set
$M:=\lVert h\rVert_\infty$. The case $M=0$ is trivial, so assume $M>0$.
By Lusin's theorem
\cite[Theorems~7.1.3 and~7.5.2]{Dudley_2002}, for every $m\ge1$ there exists a closed set $K_m\subset\bR^p$ such that
$\nu(K_m^c) \le {1}/{m}$
and the restriction of $h$ to $K_m$ is continuous. By the Tietze extension
theorem \cite[Theorem~2.6.4(a)]{Dudley_2002}, there exists a bounded
continuous function $h_m$ on $\bR^p$ satisfying $h_m=h$ on $K_m$ and $\lVert h_m\rVert_\infty\le M$.
Since $\P\ll\nu$ and $\nu(K_m^c)\to0$, one has
$\P(K_m^c)\to0$. Therefore,
\begin{align*}
\P[\lvert h-h_m\rvert]
&\le
2M\P(K_m^c)
\longrightarrow0.
\end{align*}
Moreover,
\begin{align*}
\Q_n[\lvert h-h_m\rvert]
\le
2M\Q_n(K_m^c)
\le
2M
\sup_{A:\nu(A)\le1/m}
\Q_n(A).
\end{align*}
Consequently, for every $\epsilon>0$,
\begin{align*}
\lim_{m\to\infty}
\limsup_{n\to\infty}
\P_\cO\big(
\Q_n[\lvert h-h_m\rvert]>\epsilon
\big)
\le
\lim_{m\to\infty}
\limsup_{n\to\infty}
\P_\cO\Big(
\sup_{A:\nu(A)\le1/m}
\Q_n(A)>
\frac{\epsilon}{2M}
\Big)
=0
\end{align*}
by \eqref{eq:AUAC_Qn_nu}. Lemma~\ref{lem:continuous_approximation_principle}
therefore gives $(\Q_n-\P)h=o_{\P_\cO}(1)$ for every
$h\in\cG_\cF$, which proves \eqref{eq:pointwise_moment_consistency} and hence \eqref{eq:FD-cov}.
\end{proof}

\subsubsection{Proof of Corollary~\ref{cor:divergence_FD_cov}}

\begin{proof}

We first prove the total-variation assertion. Since
\begin{align*}
\sup_{h\in\BL1(\bR^p,\lVert \cdot\rVert)}
\big\lvert
(\Q_n-\P)h
\big\rvert
\le
2d_{\rm TV}(\Q_n,\P),
\end{align*}
total-variation convergence implies
$\Q_n\rightsquigarrow\P$ in $\P_\cO$-probability. Moreover, for every
$\delta>0$,
\begin{align*}
\sup_{A:\P(A)\le\delta}
\Q_n(A)
\le
\sup_{A:\P(A)\le\delta}
\big\{
\P(A)+\lvert\Q_n(A)-\P(A)\rvert
\big\}
\le
\delta+d_{\rm TV}(\Q_n,\P).
\end{align*}
Consequently, for every $\epsilon>0$ and every
$0<\delta<\epsilon/2$,
\begin{align*}
\P_\cO\Big(
\sup_{A:\P(A)\le\delta}
\Q_n(A)>\epsilon
\Big)
\le
\P_\cO\Big(
d_{\rm TV}(\Q_n,\P)>\frac{\epsilon}{2}
\Big)
\longrightarrow0.
\end{align*}
Thus \eqref{eq:AUAC_Qn_nu} holds with $\nu=\P$. Since $\P$ is a finite
Borel measure and $\P\ll\P$, Theorem~\ref{thm:FD-cov}(ii) yields
\eqref{eq:FD-cov}.

It remains to verify that each displayed divergence condition implies
total-variation convergence. Pinsker's inequality gives
\begin{align*}
d_{\rm TV}(\Q_n,\P)
\le
\Big\{\frac{1}{2}D_{\rm KL}(\Q_n\|\P)\Big\}^{1/2} \wedge
\Big\{\frac{1}{2}D_{\rm KL}(\P\|\Q_n)\Big\}^{1/2}.
\end{align*}
Moreover, write $\mu:=\P+\Q_n$, $p:=\d\P/\d\mu$, and $q_n:=\d\Q_n/\d\mu$. Since
\begin{align*}
H^2(\Q_n,\P)
:=
\frac{1}{2}\int \big(\sqrt{q_n}-\sqrt p\big)^2\d\mu,
\end{align*}
we have by Cauchy–Schwarz inequality,
\begin{align*}
d_{\rm TV}(\Q_n,\P)
&=
\frac{1}{2}\int \lvert q_n-p \rvert \d\mu
=
\frac{1}{2}\int
\lvert \sqrt{q_n}-\sqrt p \rvert (\sqrt{q_n}+\sqrt p) \d\mu
\\
&\le
\frac{1}{2}
\Big\{ \int (\sqrt{q_n}-\sqrt p)^2\d\mu \Big\}^{1/2}
\Big\{ \int (\sqrt{q_n}+\sqrt p)^2\d\mu \Big\}^{1/2}
\le
\sqrt 2H(\Q_n,\P).
\end{align*}
Finally, on the event $\{\Q_n\ll\P\}$,
\begin{align*}
2d_{\rm TV}(\Q_n,\P)
=
\int
\Big\lvert
\frac{\d\Q_n}{\d\P}-1
\Big\rvert
\d\P
\le
\Big\{
\int
\Big(
\frac{\d\Q_n}{\d\P}-1
\Big)^2
\d\P
\Big\}^{1/2}
=
\big\{\chi^2(\Q_n\|\P)\big\}^{1/2}.
\end{align*}

Thus convergence to zero of any of the displayed divergences implies
$d_{\rm TV}(\Q_n,\P)=o_{\P_\cO}(1)$ and hence the \eqref{eq:FD-cov}.
\end{proof}

\subsubsection{Proof of Lemma~\ref{lem:nu_UI_AUAC}}

\begin{proof}
We first prove that \eqref{eq:nu_UI_density} implies
\eqref{eq:AUAC_Qn_nu}. On the event $\{\Q_n\ll\nu\}$, for every Borel
set $A$ and every $L>0$,
\begin{align*}
\Q_n(A)
=
\int_A q_n\d\nu
\le
L\nu(A)
+
\int q_n\ind\{q_n>L\}\d\nu.
\end{align*}
Hence
\begin{align*}
\sup_{A:\nu(A)\le\delta}
\Q_n(A)
\le
L\delta
+
\int q_n\ind\{q_n>L\}\d\nu
\end{align*}
on the same event. Fix $\epsilon>0$. For every fixed $L>0$ and every
$0<\delta\le\epsilon/(2L)$,
\begin{align*}
\P_\cO\Big(
\sup_{A:\nu(A)\le\delta}
\Q_n(A)
>
\epsilon
\Big)
\le
\P_\cO(\Q_n\not\ll\nu)
+
\P_\cO\Big(
\int q_n\ind\{q_n>L\}\d\nu
>
\frac{\epsilon}{2}
\Big).
\end{align*}
Therefore, for every fixed $\epsilon>0$ and $L>0$,
\begin{align*}
\lim_{\delta\downarrow0}
\limsup_{n\to\infty}
\P_\cO\Big(
\sup_{A:\nu(A)\le\delta}
\Q_n(A)
>
\epsilon
\Big)
\le
\limsup_{n\to\infty}
\P_\cO\Big(
\int q_n\ind\{q_n>L\}\d\nu
>
\frac{\epsilon}{2}
\Big).
\end{align*}
Letting $L\to\infty$ proves
\eqref{eq:AUAC_Qn_nu}.

Conversely, suppose \eqref{eq:AUAC_Qn_nu} holds. For $L>0$, define
\begin{align*}
A_{n,L}
:=
\big\{z:q_n(z)>L\big\}.
\end{align*}
On the event $\{\Q_n\ll\nu\}$,
\begin{align*}
L\nu(A_{n,L})
\le
\int_{A_{n,L}}q_n\d\nu
\le
1,
\end{align*}
so that $\nu(A_{n,L})\le1/L$. Moreover,
\begin{align*}
\int q_n\ind\{q_n>L\}\d\nu
&=
\Q_n(A_{n,L})
\le
\sup_{A:\nu(A)\le1/L}
\Q_n(A).
\end{align*}
Thus, for every $\epsilon>0$,
\begin{align*}
&\P_\cO\Big(
\int q_n\ind\{q_n>L\}\d\nu
>
\epsilon
\Big)
\le
\P_\cO(\Q_n\not\ll\nu)
+
\P_\cO\Big(
\sup_{A:\nu(A)\le1/L}
\Q_n(A)
>
\epsilon
\Big).
\end{align*}
Taking the $\limsup$ as $n\to\infty$ and then letting $L\to\infty$ proves
\eqref{eq:nu_UI_density}.
\end{proof}

\subsubsection{Proof of Corollary~\ref{cor:superlinear_AUAC}}

\begin{proof}
Suppose first that \eqref{eq:reference_orlicz_bound} holds. For all sufficiently large $L$, set
$a_L
:=
\sup_{t>L} {t}/{\Phi(t)}$.
The assumption gives $a_L\to0$. On the event
$\{\Q_n\ll\nu\}$,
\begin{align*}
\int q_n\ind\{q_n>L\}\d\nu
\le
 a_L
\int \Phi(q_n)\d\nu.
\end{align*}
Fix $\epsilon,\eta>0$. By stochastic boundedness, there exists $M<\infty$ such that, for
all sufficiently large $n$,
\begin{align*}
\P_\cO\Big(
\Q_n\not\ll\nu
\ \text{or}\ 
\int\Phi(q_n)\d\nu>M
\Big)
<
\eta.
\end{align*}
Choose $L$ sufficiently large that $a_L<M^{-1}\epsilon$. Then
\begin{align*}
\limsup_{n\to\infty}
\P_\cO\Big(
\int q_n\ind\{q_n>L\}\d\nu
>
\epsilon
\Big)
\le
\eta.
\end{align*}
Since $\eta>0$ is arbitrary, \eqref{eq:nu_UI_density} follows.

For $r\in(1,\infty)$, take $\Phi(t)=t^r$ and \eqref{eq:reference_orlicz_bound} yields the claim. For the case where $r=\infty$, notice on the event
\begin{align*}
\big\{
\Q_n\ll\nu,
\ \lVert q_n\rVert_{L^\infty(\nu)}\le L
\big\},
\end{align*}
one has
\begin{align*}
\int q_n\ind\{q_n>L\}\d\nu
=0.
\end{align*}
Therefore, for every $\epsilon>0$,
\begin{align*}
\P_\cO\Big(
\int q_n\ind\{q_n>L\}\d\nu
>
\epsilon
\Big)
\le
\P_\cO(\Q_n\not\ll\nu)
+
\P_\cO\big(
\lVert q_n\rVert_{L^\infty(\nu)}>L
\big).
\end{align*}
Taking the $\limsup$ as $n\to\infty$ and then letting $L\to\infty$ proves
\eqref{eq:nu_UI_density}.
\end{proof}

\subsubsection{Proof of Corollary~\ref{cor:localized_Lebesgue_FD_cov}}

\begin{proof}
We verify the assumptions of Lemma~\ref{lem:continuous_approximation_principle}. 
Fix $h \in \cG_\cF$ and let $M:= \lVert h \rVert_{\infty}$. The case $M=0$ is trivial, so we suppose $M>0$.
Choose a sequence $R_m\uparrow\infty$ such that
\begin{align*}
    \P(B_{R_m}^c)\le \frac{1}{m}
    ~~\text{ and }~~
    \P(\partial B_{R_m})=0.
\end{align*}
Such a sequence exists as $\P$ has continuity from below and is absolutely continuous to $\lambda$. 

Since $\P\ll\lambda$ and $\P$ is a finite measure, the integral with density $\d\P/\d\lambda$ is uniformly absolutely continuous with respect to $\lambda$. Then for every $m\ge1$ there exists $\delta_m^{(1)}>0$ such that $\lambda(A)\le\delta_m^{(1)}$ implies $\P(A)\le{1}/{m}$ for all $A$. 

Condition~\eqref{eq:local_Lebesgue_UI} implies the corresponding
local asymptotic uniform absolute continuity statement. Indeed, fix
$R>0$ and $\epsilon>0$. On the event $\{\Q_n\ll\lambda\}$, for every
$L>0$ and every $\delta>0$,
\begin{align*}
\sup_{A\subseteq B_R,\lambda(A)\le\delta}
\Q_n(A)
\le
L\delta
+
\int_{B_R}
q_n\ind\{q_n>L\}
\d\lambda.
\end{align*}
Hence whenever $0<\delta\le\epsilon/(2L)$,
\begin{align*}
\P_\cO\Big(
\Q_n\not\ll\lambda
\ \text{or}\
\sup_{\substack{A\subseteq B_R,\lambda(A)\le\delta}}
\Q_n(A)>\epsilon
\Big)
\le
\P_\cO(\Q_n\not\ll\lambda)
+
\P_\cO\Big(
\int_{B_R}
q_n\ind\{q_n>L\}
\d\lambda
>
\frac{\epsilon}{2}
\Big).
\end{align*}
Taking the $\limsup$ as $n\to\infty$, then letting
$\delta\downarrow0$ and finally $L\to\infty$, gives
\begin{align*}
\lim_{\delta\downarrow0}
\limsup_{n\to\infty}
\P_\cO\Big(
\Q_n\not\ll\lambda
\ \text{or}\
\sup_{A\subseteq B_R,\lambda(A)\le\delta}
\Q_n(A)>\epsilon
\Big)
=0.
\end{align*}
Applying this statement with $R=R_m$ and $\epsilon=1/m$, choose
$\delta_m^{(2)}>0$ such that
\begin{align*}
\limsup_{n\to\infty}
\P_\cO\Big(
\Q_n\not\ll\lambda
\ \text{or}\
\sup_{A\subseteq B_{R_m},
\lambda(A)\le\delta_m^{(2)}}
\Q_n(A)>\frac{1}{m}
\Big)
\le
\frac{1}{m}.
\yestag
\label{eq:localized_Qn_lambda_AUAC}
\end{align*}
Set $\delta_m:=\delta_m^{(1)}\wedge\delta_m^{(2)}$, by Lusin's theorem \cite[Theorems~7.1.3 and~7.5.2]{Dudley_2002} applied to the restriction of Lebesgue measure to $B_{R_m}$, there exists a compact set $K_m\subset B_{R_m}$ such that $\lambda(B_{R_m}\setminus K_m)\le\delta_m$
and the restriction of $h$ to $K_m$ is continuous. Applying the Tietze extension theorem \cite[Theorem~2.6.4(a)]{Dudley_2002}, there exists a bounded continuous function $h_m$ on $\bR^p$ such that $h_m=h$ on $K_m$ and $\lVert h_m\rVert_\infty\le M$. For such $h_m$,
\begin{align*}
\P[\lvert h-h_m\rvert]
\le
2M\P(K_m^c)
\le
2M\big\{ \P(B_{R_m}^c) + \P(B_{R_m}\setminus K_m) \big\}
\le
\frac{4M}{m}
\to0
\end{align*}
and
\begin{align*}
\Q_n[\lvert h-h_m\rvert]
\le
2M\Q_n(K_m^c)
\le
2M\big\{ \Q_n(B_{R_m}^c) + \Q_n(B_{R_m}\setminus K_m)
\big\}.
\end{align*}
For the first term here, since $\P(\partial B_{R_m})=0$, the weak convergence $\Q_n\rightsquigarrow\P$ in $\P_\cO$-probability implies, for every fixed $m$,
\begin{align*}
\Q_n(B_{R_m}^c)-\P(B_{R_m}^c)
=
o_{\P_\cO}(1);
\end{align*}
together with $\P(B_{R_m}^c)\le 1/m$, this implies
\begin{align*}
    \limsup_{n\to\infty}
    \P_\cO\Big(\Q_n(B_{R_m}^c)>\frac{2}{m}\Big)=0.
\end{align*}
For the second term, by the choice of $K_m$, as $B_{R_m}\setminus K_m\subset B_{R_m}$ and $\lambda(B_{R_m}\setminus K_m)\le\delta_m^{(2)}$, condition \eqref{eq:localized_Qn_lambda_AUAC} gives
\begin{align*}
\limsup_{n\to\infty}
\P_\cO\Big(
\Q_n(B_{R_m}\setminus K_m)>\frac{1}{m}
\Big)
\le
\frac{1}{m}.
\end{align*}
Therefore, for every $m$ such that $6M/m<\epsilon$,
\begin{align*}
    &\limsup_{n\to\infty}
    \P_\cO\big(\Q_n[\lvert h-h_m\rvert]>\epsilon\big)
    \\
    &~~\le 
    \limsup_{n\to\infty} \P_\cO\Big(2M\big\{ \Q_n(B_{R_m}^c) + \Q_n(B_{R_m}\setminus K_m)
    \big\}>\frac{6M}{m}\Big)
    \\
    &~~\le
    \limsup_{n\to\infty}
    \P_\cO\Big(\Q_n(B_{R_m}^c)>\frac{2}{m}\Big)
    +
    \limsup_{n\to\infty}
    \P_\cO\Big(\Q_n(B_{R_m}\setminus K_m)>\frac{1}{m}\Big)
    \le \frac{1}{m}.
\end{align*}
Letting $m\to\infty$ yields for every $\epsilon>0$,
\begin{align*}
\lim_{m\to\infty}
\limsup_{n\to\infty}
\P_\cO\big(
\Q_n[\lvert h-h_m\rvert]>\epsilon
\big)
=0.
\end{align*}
Thus the assumptions of Lemma~\ref{lem:continuous_approximation_principle} hold and the finite-dimensional covariance consistency condition follows.
\end{proof}

\subsection{Proofs of the results in Section~\ref{sec:suff_cond_CAE}}

\subsubsection{Proof of Theorem~\ref{thm:reference_measure}}

\begin{proof}
For $\delta>0$, write
\begin{align*}
Z_{n,\delta}
:=
\sup_{\rho_\P(f,g)<\delta}
\big\lvert
\tilde\bG_n[f]-\tilde\bG_n[g]
\big\rvert,
\end{align*}
and let
\begin{align*}
M_{n,\delta}(\cO)
:=
\E_{\tilde U}\big[
Z_{n,\delta}
\mid \cO
\big]
=
\E_{\Q_n}\big[
\big\lVert
\bG_n^{\Q_n}
\big\rVert_{\cF_\delta} \big].
\end{align*}
Conditional Markov's inequality gives, for every
$\epsilon,\eta>0$,
\begin{align*}
\P_\cO\big\{
\P_{\tilde U}\big(
Z_{n,\delta}>\epsilon
\mid \cO
\big)>\eta \big\}
\le
\P_\cO\big(
M_{n,\delta}(\cO)>\epsilon\eta
\big).
\yestag
\label{eq:CAE_by_local_mean}
\end{align*}
Fix $1\le C<\infty$ and let $\Omega_n(C) := \{\Q_n\le C\nu\}$.
On $\Omega_n(C)$, apply
Lemma~\ref{lem:centered_empirical_process_comparison} conditionally on $\cO$
with $\mu=\Q_n$, reference measure $\nu$, and
$\cA=\cF_\delta$. Since $\cF_\delta$ has envelope bounded by
$2B$,
\begin{align*}
M_{n,\delta}(\cO)
\le
K_C
\E_\nu\big[
\big\lVert\bG_n^\nu\big\rVert_{\cF_\delta}
\big]
+
8B\sqrt n
\exp\Big(
\frac{-n}{4}
\Big).
\end{align*}
Therefore,
\begin{align*}
\P_\cO\big(
M_{n,\delta}(\cO)>\epsilon\eta
\big)
\le
\P_\cO\big(
\Q_n\nleq C\nu
\big)
+
\ind\Big\{
K_C
\E_\nu\big[
\big\lVert\bG_n^\nu\big\rVert_{\cF_\delta} \big]
+
8B\sqrt n
\exp\Big(
\frac{-n}{4}
\Big)
>
\epsilon\eta
\Big\}.
\end{align*}
For each fixed $C$, condition \eqref{eq:local_mod_mean_ref_meas} implies
\begin{align*}
\lim_{\delta\downarrow0}
\limsup_{n\to\infty}
\ind\Big\{
K_C
\E_\nu\big[
\big\lVert\bG_n^\nu\big\rVert_{\cF_\delta} \big]
+
8B\sqrt n
\exp\Big(
\frac{-n}{4}
\Big)
>
\epsilon\eta
\Big\}
=0.
\end{align*}
Combining this display with \eqref{eq:CAE_by_local_mean} gives
\begin{align*}
&\lim_{\delta\downarrow0}
\limsup_{n\to\infty}
\P_\cO\big\{
\P_{\tilde U}\big(
Z_{n,\delta}>\epsilon
\mid \cO
\big)>\eta
\big\}
\le
\limsup_{n\to\infty}
\P_\cO\big(
\Q_n\nleq C\nu
\big).
\end{align*}
Letting $C\to\infty$ and using
\eqref{eq:stochastic_reference_domination} proves \eqref{eq:CAE}.
\end{proof}

\subsubsection{Proof of Corollary~\ref{cor:domination}}

\begin{proof}
If $\cF$ is $\P$-Donsker, then Corollary~2.3.12 of
\cite{MR1385671} gives
\begin{align*}
\lim_{\delta\downarrow0}
\limsup_{m\to\infty}
\E_\P\big[ \big\lVert\bG_m^\P\big\rVert_{\cF_\delta} \big]
=0.
\end{align*}
Thus condition \eqref{eq:local_mod_mean_ref_meas} holds with $\nu=\P$. The conclusion follows from Theorem~\ref{thm:reference_measure}.
\end{proof}

\subsubsection{Proof of Corollary~\ref{cor:Lebegue_density}}

\begin{proof}
Let $c=\operatorname*{essinf}_{z\in S}p(z)>0$. For any fixed $M<\infty$, define
\begin{align*}
A_n(M)
:=
\big\{
\Q_n\ll\lambda,\ 
\Q_n(S)=1,\ 
\lVert q_n\rVert_{L^\infty(\lambda)}\le M
\big\}.
\end{align*}
On $A_n(M)$, one has $q_n=0$ Lebesgue-almost everywhere on $S^c$ and for
every Borel set $A$,
\begin{align*}
\Q_n(A)
=
\int_{A\cap S}q_n(z)\d\lambda(z)
\le
M\lambda(A\cap S)
\le
\frac{M}{c}
\int_{A\cap S}p(z)\d\lambda(z)
\le
\frac{M}{c}\P(A).
\end{align*}
Thus $\Q_n\le(M/c)\P$ on $A_n(M)$, and hence
\begin{align*}
\P_\cO\Big(\Q_n\nleq \frac{M}{c}\P \Big)
&\le
\P_\cO\big(A_n(M)^c\big)
\\
&\le
\P_\cO\big( 
\big\{\Q_n\ll\lambda,\Q_n(S)=1\big\}^c 
\big)
+
\P_\cO\big( 
\lVert q_n\rVert_{L^\infty(\lambda)}>M
\big)
\end{align*}
The assumed stochastic boundedness of
$\lVert q_n\rVert_{L^\infty(\lambda)}$ with
$\P_\cO(\Q_n\ll\lambda,\Q_n(S)=1)\to1$ implies
\begin{align*}
\lim_{M\to\infty}
\limsup_{n\to\infty}
\P_\cO\big(
A_n(M)^c
\big)
=0;
\end{align*}
consequently,
\begin{align*}
\lim_{C\to\infty}
\limsup_{n\to\infty}
\P_\cO\big(\Q_n\nleq C\P\big)=0.
\end{align*}
The result follows from Corollary~\ref{cor:domination}.
\end{proof}

\subsubsection{Proof of Theorem~\ref{thm:unif_entropy}}

\begin{proof}
For a deterministic probability measure $\Q$, define
\begin{align*}
\cC_{\Q,\delta}
:=
\big\{
(f-g)-\Q[f-g]:
f,g\in\cF,\ \rho_\P(f,g)<\delta
\big\}.
\end{align*}
Then write $\mathbb P_n^\Q$ as the empirical measure of an i.i.d. $\Q$-sample and $\bG_n^\Q:=\sqrt n(\mathbb P_n^\Q-\Q)$,
\begin{align*}
\sup_{\rho_\P(f,g)<\delta}
\big\lvert \bG_n^\Q[f]-\bG_n^\Q[g]\big\rvert
=
\lVert \bG_n^\Q\rVert_{\cC_{\Q,\delta}}.
\end{align*}
The class $\cC_{\Q,\delta}$ has envelope bounded by $4B$.

We first record a covering number bound for $\cC_{\Q,\delta}$. Fix an arbitrary probability measure $\Lambda$ and set $\bar\Lambda:=\frac{1}{2}(\Lambda+\Q)$.
Let $\{f_1,\ldots,f_N\}$ be an $\epsilon B$-net of $\cF$ in $L^2(\bar\Lambda)$. For any $f,g\in\cF$, choose $f_j,f_k$ such that $\lVert f-f_j\rVert_{\bar\Lambda,2}\le \epsilon B$ and $\lVert g-f_k\rVert_{\bar\Lambda,2}\le \epsilon B$. Then
\begin{align*}
&\big\lVert
\{f-g-\Q[f-g]\}
-
\{f_j-f_k-\Q[f_j-f_k]\}
\big\rVert_{\Lambda,2}
\\
&~~\le
\lVert f-f_j\rVert_{\Lambda,2}
+
\lVert g-f_k\rVert_{\Lambda,2}
+
\big\lvert \Q[f-f_j]\big\rvert
+
\big\lvert \Q[g-f_k]\big\rvert
\\
&~~\le
\sqrt 2\lVert f-f_j\rVert_{\bar\Lambda,2}
+
\sqrt 2\lVert g-f_k\rVert_{\bar\Lambda,2}
+
\sqrt 2\lVert f-f_j\rVert_{\bar\Lambda,2}
+
\sqrt 2\lVert g-f_k\rVert_{\bar\Lambda,2}
\\
&~~\le
4\sqrt 2\epsilon B.
\end{align*}
Therefore,
\begin{align*}
\cN \big(4\sqrt2\epsilon B,\cC_{\Q,\delta},L^2(\Lambda)\big)
\le
\cN \big(\epsilon B,\cF,L^2(\bar\Lambda)\big)^2.
\end{align*}
In particular, the uniform entropy integral of $\cC_{\Q,\delta}$ is controlled by that of $\cF$, uniformly in $\Q$ and $\delta$. To see this, let the normalized class $\cC_{\Q,\delta}':= (4B)^{-1}\cC_{\Q,\delta}$ such that it has envelope bounded by one, then for any $0<u\le 1$ such that $\epsilon=u/\sqrt 2$,
\begin{align*}
\cN(u,\cC_{\Q,\delta}',L^2(\Lambda))
=
\cN(4Bu,\cC_{\Q,\delta},L^2(\Lambda)) \le
\cN({uB}/{\sqrt 2},\cF, L^2(\bar\Lambda))^2,
\end{align*}
therefore
\begin{align*}
\sqrt{1+\log \cN(u,\cC_{\Q,\delta}',L^2(\Lambda))}
\le
\sqrt 2
\sqrt{1+\log \cN({uB}/{\sqrt 2},\mathcal F,L^2(\bar\Lambda))}.
\end{align*}
and consequently for any $0<u\le 1$ one has
\begin{align*}
J(u,\cC_{\Q,\delta}',1)
\le
2J(u/\sqrt2,\mathcal F,F)
\le
2J(u,\mathcal F,F).
\end{align*}
Denote
\begin{align*}
\sigma_\Q(\delta)
:=
\sup_{\rho_\P(f,g)<\delta}\rho_\Q(f,g)
=
\sup_{h\in\cC_{\Q,\delta}}\lVert h\rVert_{\Q,2};
\end{align*}
Theorem~2.1 in \cite{van2011local} applied to the normalized class $\cC_{\Q,\delta}'$ then implies there exists a universal constant $K<\infty$ such that, for every $0<\gamma<4B$, on the event $\sigma_\Q(\delta)\le \gamma$,
\begin{align*}
\E_\Q\big[
\lVert \bG_n^\Q\rVert_{\cC_{\Q,\delta}}\big]
\le
K B
J({\gamma}/{4B},\cF,F)
\Big(
1+\frac{B^2J({\gamma}/{4B},\cF,F)}{\gamma^2\sqrt n}
\Big).
\yestag
\label{eq:global_entropy_maximal_bound}
\end{align*}
Here $\E_\Q$ denotes expectation over an i.i.d. $\Q$-sample of size $n$. Now define
\begin{align*}
A_{n,\delta}(\cO)
:=
\P_{\tilde U}\Big(
\sup_{\rho_\P(f,g)<\delta}
\big\lvert \tilde\bG_n[f]-\tilde\bG_n[g]\big\rvert>\epsilon
\mid \cO
\Big);
\end{align*}
On the event $E_{n,\delta}(\gamma)
:=\{ \sup_{\rho_\P(f,g)<\delta}\rho_{\Q_n}(f,g)\le \gamma\}$, apply the bound \eqref{eq:global_entropy_maximal_bound} with $\Q=\Q_n$ gives
\begin{align*}
A_{n,\delta}(\cO)
&\le
\frac{1}{\epsilon}
\E_{\tilde U}\big[
\lVert \tilde\bG_n\rVert_{\cC_{\Q_n,\delta}}
\mid \cO
\big]
\le
\frac{K B}{\epsilon}
J({\gamma}/{4B},\cF,F)
\Big(1+ \frac{B^2 J({\gamma}/{4B},\cF,F)}{\gamma^2\sqrt n} \Big).
\end{align*}
Hence, for every $\eta>0$,
\begin{align*}
\P_\cO\big(A_{n,\delta}(\cO)>\eta \big)
&\le
\P_\cO\big(E_{n,\delta}(r)^c\big)
+
\ind\Big\{
\frac{K B}{\epsilon}
J({\gamma}/{4B},\cF,F)
\Big( 1+\frac{B^2 J({\gamma}/{4B},\cF,F)}{\gamma^2\sqrt n} \Big)>\eta
\Big\};
\end{align*}
taking the $n$-limit on both sides then gives
\begin{align*}
\limsup_{n\to\infty}\P_\cO \big(A_{n,\delta}(\cO)>\eta \big)
&\le
\limsup_{n\to\infty}
\P_\cO\Big(
\sup_{\rho_\P(f,g)<\delta}\rho_{\Q_n}(f,g)>r
\Big) +
\ind\Big( \frac{K B}{\epsilon}
J({\gamma}/{4B},\cF,F)
>\eta \Big).
\end{align*}
Now first let $\delta$ goes to $0$. The first term vanishes by \eqref{eq:local_metric_transfer_Qn}. Then let $\gamma$ goes to $0$ and $J({\gamma}/{4B},\cF,F)\to0$ by \eqref{eq:unif_entropy_F}. Therefore we obtained the desired  conditional asymptotic $\rho_\P$-equicontinuity \eqref{eq:CAE}.
\end{proof}

\subsubsection{Proof of Lemma~\ref{lem:local_density_metric_transfer}}

\begin{proof}
Let $h_{f,g} := f-g-\P[f-g]$ for $f,g\in\cF$.
Then
\begin{align*}
\lVert h_{f,g}\rVert_\infty
\le 4B,
\qquad
\P h_{f,g}^2
=
\rho_\P(f,g)^2.
\end{align*}
Moreover, since $c\mapsto\Q_n(f-g-c)^2$ is minimized at
$c=\Q_n[f-g]$,
\begin{align*}
\rho_{\Q_n}(f,g)^2
=
\Q_n\big\{
 f-g-\Q_n[f-g]
\big\}^2
\le
\Q_n h_{f,g}^2.
\end{align*}
Fix $R<\infty$. On the event $\{\Q_n\ll\lambda\}$,
\begin{align*}
\Q_n h_{f,g}^2
&=
\int_{B_R}h_{f,g}(z)^2q_n(z)\d z
+
\int_{B_R^c}h_{f,g}(z)^2\d\Q_n(z)
\\
&\le
\int_{B_R}h_{f,g}(z)^2q_n(z)\d z
+
16B^2\Q_n(B_R^c).
\end{align*}
Write $r_R':=r_R/(r_R-1)$. Hölder's inequality bounds the first term by
\begin{align*}
\int_{B_R}h_{f,g}^2q_n\d z
&\le
\Big\{
\int_{B_R}q_n^{r_R}\d z
\Big\}^{1/r_R}
\Big\{
\int_{B_R}\lvert h_{f,g}\rvert^{2r_R'}\d z
\Big\}^{1/r_R'}
\\
&\le
\Big\{
\int_{B_R}q_n^{r_R}\d z
\Big\}^{1/r_R}
\Big\{(4B)^{2r_R'-2}
\int_{B_R}h_{f,g}^2\d z\Big\}^{1/r_R'}.
\end{align*}
Since $p(z) \ge c_R :=\operatorname*{essinf}_{z\in B_R}p(z)$ Lebesgue-almost everywhere on $B_R$, whenever
$\rho_\P(f,g)<\delta$,
\begin{align*}
\int_{B_R}h_{f,g}^2\d z
\le
c_R^{-1}\P h_{f,g}^2
<
c_R^{-1}\delta^2.
\end{align*}
Then for every fixed $R<\infty$,
\begin{align*}
\sup_{\rho_\P(f,g)<\delta}
\int_{B_R}h_{f,g}^2q_n\d z
\le
\Big\{
\int_{B_R}q_n^{r_R}\d z
\Big\}^{1/r_R}
\big\{
(4B)^{2r_R'-2}
c_R^{-1}\delta^2
\big\}^{1/r_R'}
=
O_{\P_\cO}\big(
\delta^{2/r_R'}
\big);
\end{align*}
which implies, for every $\eta>0$ and fixed $R$,
\begin{align*}
\lim_{\delta\downarrow0}
\limsup_{n\to\infty}
\P_\cO\Big(
\sup_{\rho_\P(f,g)<\delta}
\int_{B_R}h_{f,g}^2q_n\d z
>\eta
\Big)
=0.
\end{align*}
Consequently, for any fixed $\epsilon>0$, 
\begin{align*}
&\P_\cO\Big(
\sup_{\rho_\P(f,g)<\delta}
\rho_{\Q_n}(f,g)>\epsilon
\Big)
\\
&~~\le
\P_\cO(\Q_n\not\ll\lambda)
+
\P_\cO\Big(
\sup_{\rho_\P(f,g)<\delta}
\int_{B_R}h_{f,g}^2q_n\d z
>
\frac{\epsilon^2}{2}
\Big)
+
\P_\cO\Big(
\Q_n(B_R^c)
>
\frac{\epsilon^2}{32B^2}
\Big).
\end{align*}
Therefore for every fixed $R$,
\begin{align*}
\lim_{\delta\downarrow0}
\limsup_{n\to\infty}
\P_\cO\Big(
\sup_{\rho_\P(f,g)<\delta}
\rho_{\Q_n}(f,g)>\epsilon
\Big)
\le
\limsup_{n\to\infty}
\P_\cO\Big(
\Q_n(B_R^c)
>
\frac{\epsilon^2}{32B^2}
\Big).
\end{align*}
Letting $R\to\infty$ and applying
\eqref{eq:Qn_asymptotic_tightness} proves
\eqref{eq:local_metric_transfer_Qn}.
\end{proof}

\subsection{Proofs of the results in Section~\ref{subsec:relation_Donsker}}

\subsubsection{Proof of Theorem~\ref{thm:CAE_to_Donsker}}

\begin{proof}
For $\delta>0$, define
\begin{align*}
L_{n,\delta}
:=
\E_\P\big[
\big\lVert
\bG_n^\P
\big\rVert_{\cF_\delta}
\big],
\qquad
M_{n,\delta}
:=
\E_{\Q_n}\big[
\big\lVert
\bG_n^{\Q_n}
\big\rVert_{\cF_\delta}
\big],
\qquad
r_n
:=
8B\sqrt n
\exp\Big(
-\frac n4
\Big).
\end{align*}
Here $L_{n,\delta}$ is deterministic, whereas $M_{n,\delta}$ is random
through its dependence on $\cO$.

By Lemma~\ref{lem:CAE_local_mean_equivalence}, condition \eqref{eq:CAE}
implies that, for every $a>0$,
\begin{align*}
\lim_{\delta\downarrow0}
\limsup_{n\to\infty}
\P_\cO\big(
M_{n,\delta}>a
\big)
=
0.
\yestag
\label{eq:reverse_transfer_Q_local_mean}
\end{align*}
Write $\Omega_n(D) := \{\P\le D\Q_n\}$ and denote $\alpha:=\liminf_{n\to\infty}
\P_\cO(\Omega_n(D))>0$. On $\Omega_n(D)$, Lemma~\ref{lem:centered_empirical_process_comparison},
applied conditionally on $\cO$ with
$\mu=\P$, $\nu=\Q_n$, and $\cA=\cF_\delta$, gives
\begin{align*}
L_{n,\delta}
\le
K_D M_{n,\delta}
+
r_n.
\end{align*}
Fix $a>0$. For all sufficiently large $n$,
\begin{align*}
\P_\cO\big(
\Omega_n(D)
\big)
\ge
\frac{\alpha}{2},
\qquad
r_n
\le
K_Da.
\end{align*}
For such $n$, since $L_{n,\delta}$ is deterministic,
\begin{align*}
\frac{\alpha}{2}
\ind\big\{
L_{n,\delta}>2K_Da
\big\}
\le
\P_\cO\Big(
\Omega_n(D)
\cap
\big\{
L_{n,\delta}>2K_Da
\big\}
\Big)
\le
\P_\cO\big(
M_{n,\delta}>a
\big).
\end{align*}
Taking the $\limsup$ as $n\to\infty$ gives
\begin{align*}
\frac{\alpha}{2}
\limsup_{n\to\infty}
\ind\big\{
L_{n,\delta}>2K_Da
\big\}
\le
\limsup_{n\to\infty}
\P_\cO\big(
M_{n,\delta}>a
\big).
\end{align*}
The limsup of the indicator on the left belongs to $\{0,1\}$.
Equation \eqref{eq:reverse_transfer_Q_local_mean} therefore implies that,
for all sufficiently small $\delta$,
\begin{align*}
\limsup_{n\to\infty}
L_{n,\delta}
\le
2K_Da.
\end{align*}
Hence
\begin{align*}
\lim_{\delta\downarrow0}
\limsup_{n\to\infty}
L_{n,\delta}
\le
2K_Da.
\end{align*}
Since $a>0$ is arbitrary,
\begin{align*}
\lim_{\delta\downarrow0}
\limsup_{n\to\infty}
L_{n,\delta}
=
0.
\yestag
\label{eq:reverse_transfer_P_local_mean}
\end{align*}
By Markov's inequality, \eqref{eq:reverse_transfer_P_local_mean} implies
asymptotic $\rho_\P$-equicontinuity of the ordinary $\P$-empirical process.
Its finite-dimensional distributions converge to those of $\bG_\P$ by the
multivariate central limit theorem. Moreover, the canonical semimetric of
$\bG_\P$ is $\rho_\P$, since
\begin{align*}
\E\big[
\big\{
\bG_\P[f]
-
\bG_\P[g]
\big\}^2
\big]
=
\rho_\P^2(f,g).
\end{align*}
The standing tightness assumption on $\bG_\P$ implies that
$(\cF,\rho_\P)$ is totally bounded and that $\bG_\P$ has uniformly
$\rho_\P$-continuous sample paths; see Example~1.5.10 of
\cite{MR1385671}. Theorems~1.5.4 and~1.5.7 of \cite{MR1385671}
therefore yield $\bG_n^\P \rightsquigarrow \bG_\P$ in $\ell^\infty(\cF)$.
Thus $\cF$ is $\P$-Donsker.
\end{proof}

\subsubsection{Proof of Corollary~\ref{cor:Donsker_equivalence_weak_consistency}}

\begin{proof}
Suppose first that $\cF$ is $\P$-Donsker. By
Corollary~\ref{cor:domination} and
\eqref{eq:han-lrb0}, the conditional asymptotic
$\rho_\P$-equicontinuity condition \eqref{eq:CAE} holds.

It remains to verify finite-dimensional covariance consistency. Fix
$\epsilon>0$ and $C<\infty$. On the event $\{\Q_n\le C\P\}$,
\begin{align*}
\sup_{A:\P(A)\le\delta}
\Q_n(A)
\le
C\delta.
\end{align*}
Consequently, whenever $0<\delta<\epsilon/C$,
\begin{align*}
\P_\cO\Big(
\sup_{A:\P(A)\le\delta}
\Q_n(A)
>
\epsilon
\Big)
\le
\P_\cO\big(
\Q_n\nleq C\P
\big).
\end{align*}
Taking the $\limsup$ as $n\to\infty$, then letting
$\delta\downarrow0$ and finally $C\to\infty$, gives
\begin{align*}
\lim_{\delta\downarrow0}
\limsup_{n\to\infty}
\P_\cO\Big(
\sup_{A:\P(A)\le\delta}
\Q_n(A)
>
\epsilon
\Big)
=
0.
\end{align*}
Thus the asymptotic uniform absolute continuity condition
\eqref{eq:AUAC_Qn_nu} holds with $\nu=\P$. Together with
$\Q_n\rightsquigarrow\P$ in $\P_\cO$-probability,
Theorem~\ref{thm:FD-cov}(ii) yields
\eqref{eq:FD-cov}. Theorem~\ref{thm:cvg_emp_proc} now gives
\eqref{eq:CWC}.

Conversely, suppose that \eqref{eq:CWC} holds. By
Theorem~\ref{thm:cvg_emp_proc}, condition \eqref{eq:CAE} holds.
The lower-coverage condition
\eqref{eq:lower_domination_nonvanishing} and
Theorem~\ref{thm:CAE_to_Donsker} then imply that $\cF$ is
$\P$-Donsker.
\end{proof}

\section{Proofs of the application results}

\subsection{Proofs of the results in Section~\ref{sec:triangular_normalizing_flows}}

\subsubsection{Proof of Lemma~\ref{lem:triangular_KL_consistency}}

\begin{proof}
By the approximate empirical minimization condition,
\begin{align*}
\P[\ell_{\hat S_n}]
\le
\bP_n[\ell_{\hat S_n}]
+
\Delta_n^{\rm NF}
\le
\inf_{S\in\cT_n}
\bP_n[\ell_S]
+
\varepsilon_n^{\rm NF}
+
\Delta_n^{\rm NF}
\le
\inf_{S\in\cT_n}
\P[\ell_S]
+
\varepsilon_n^{\rm NF}
+
2\Delta_n^{\rm NF}.
\end{align*}

For every $S\in\cT_n$, Definition~\ref{def:triangular_normalizing_flow}
and Assumption~\ref{ass:triangular_model}(i) give
\begin{align*}
q_S(z)
=
p_U\big(S(z)\big)
\prod_{j=1}^p D_jS_j(z)
>
0
\end{align*}
for Lebesgue-almost every
$z\in\operatorname{int}(\cZ)$. Because $\cZ$ is convex with nonempty
interior, $\lambda(\partial\cZ)=0$ and $\P(\partial\cZ)=0$. Since $\P\ll\lambda$ and is supported
on $\cZ$, it follows that
\begin{align*}
    \P\big(\{q_S=0\}\big)
    = 
    \P\big(\{q_S=0\}\cap \cZ\big)
    \le 
    \P\big(\partial\cZ\big) + \P\big(\{q_S=0\}\cap \operatorname{int}(\cZ)\big) = 0;
\end{align*}
which is equivalent to $\P\ll\Q_S$ as they are both absolutely continuous to $\lambda$.

Consequently, Assumption~\ref{ass:triangular_data}(iii) and
$\ell_S\in L^1(\P)$ give
\begin{align*}
D_{\rm KL}(\P\|\Q_S)
=
\int_{\cZ} p(z)\log\frac{p(z)}{q_S(z)}dz
=
\P[\log p]
+
\P[\ell_S].
\end{align*}
Adding the common constant $\P[\log p]$ to the preceding oracle inequality
gives
\begin{align*}
D_{\rm KL}(\P\|\Q_n)
\le
\inf_{S\in\cT_n}
D_{\rm KL}(\P\|\Q_S)
+
\varepsilon_n^{\rm NF}
+
2\Delta_n^{\rm NF}
=
\alpha_n^{\rm NF}
+
2\Delta_n^{\rm NF}
+
\varepsilon_n^{\rm NF}.
\end{align*}
The final assertion follows from
Assumption~\ref{ass:triangular_model}(iii)-(v).
\end{proof}

\subsubsection{Proof of Theorem~\ref{thm:triangular_flow_bootstrap_CWC}}

\begin{proof}
We verify the two conditions in Theorem~\ref{thm:cvg_emp_proc}.

First, we verify conditional asymptotic $\rho_\P$-equicontinuity. A version of
the density of $\Q_n=\Q_{\hat S_n}$ is
\begin{align*}
q_n(z)
=
p_U\big(\hat S_n(z)\big)
\prod_{j=1}^pD_j\hat S_{n,j}(z)
\ind\{z\in\cZ\}.
\end{align*}
By Assumption~\ref{ass:triangular_model}(i)-(ii),
\begin{align*}
q_n(z)
\le
\overline p_U M_S^p\ind\{z\in\cZ\}.
\end{align*}
Since $p(z)\ge\underline p$ for Lebesgue-almost every $z\in\cZ$
and $\P(\cZ^c) = \int_{\cZ^c} p(z)\d\lambda = 0$ implies $p(z)=0$ for Lebesgue-almost every $z\in\cZ^c$, it follows that
\begin{align*}
q_n(z)
\le
\frac{\overline p_U M_S^p}{\underline p}p(z)
\end{align*}
for Lebesgue-almost every $z\in\bR^p$. Consequently, $\Q_n
\le
\frac{\overline p_U M_S^p}{\underline p}\P$
and hence
\begin{align*}
\Q_n\ll\P,
\qquad
\Big\lVert
\frac{\d\Q_n}{\d\P}
\Big\rVert_{\infty}
\le
\frac{\overline p_U M_S^p}{\underline p}.
\end{align*}
Corollary~\ref{cor:domination} then gives the conditional asymptotic
$\rho_\P$-equicontinuity condition~\eqref{eq:CAE}.

Second, Lemma~\ref{lem:triangular_KL_consistency} gives $D_{\rm KL}(\P\|\Q_n) = o_{\P_\cO}(1)$.
Corollary~\ref{cor:divergence_FD_cov} therefore yields the finite-dimensional
covariance consistency condition~\eqref{eq:FD-cov}. The conclusion follows
from Theorem~\ref{thm:cvg_emp_proc}.
\end{proof}

\subsection{Proofs of the results in Section~\ref{sec:FM}}

\subsubsection{Proof of Lemma~\ref{lem:FM_training_implies_weak_convergence}}

\begin{proof}
For every $\tilde v\in\cV_n$, the definitions of $\Delta_n^{\rm FM}$ and $\varepsilon_n^{\rm FM}$
give
\begin{align*}
\cR^{\rm FM}(\hat v_n)
\le
\hat \cR^{\rm FM}_n(\hat v_n)+\Delta_n^{\rm FM}
\le
\hat \cR^{\rm FM}_n(\tilde v)
+\varepsilon_n^{\rm FM}+\Delta_n^{\rm FM}
\le
\cR^{\rm FM}(\tilde v)
+\varepsilon_n^{\rm FM}+2\Delta_n^{\rm FM}.
\end{align*}
Taking the infimum over $\tilde v\in\cV_n$ and using
\eqref{eq:CFM_projection_identity} yields
\begin{align*}
\int_0^1
\lVert
\hat v_{n,t}-v_t
\rVert_{L^2(\P_t)}^2
\d t
=
\cR^{\rm FM}(\hat v_n)-\cR^{\rm FM}(v)
\le
\alpha_n^{\rm FM}+2\Delta_n^{\rm FM}+\varepsilon_n^{\rm FM}.
\end{align*}

For the Wasserstein bound, let $U^\circ\sim\P_U$ be independent
of $\cO$, and write
\begin{align*}
Y_t:=\Phi_t(U^\circ),
\qquad
\hat Y_t:=\hat\Phi_{n,t}(U^\circ).
\end{align*}
Conditionally on $\cO$, we have $Y_t\sim\P_t$ and
$\hat Y_1\sim\Q_n$. Thus $(Y_1,\hat Y_1)$ is a coupling of
$\P$ and $\Q_n$. The linear-growth bound gives
\begin{align*}
\lVert \hat Y_t\rVert
\le
\exp(Lt)\big(\lVert U^\circ\rVert+Mt\big),
\quad 0\le t\le1.
\end{align*}
Since $\P_U$ has a finite second moment, so does $\Q_n$.
Together with the second-moment assumption on $\P$, this ensures
that $W_2(\Q_n,\P)$ is finite.

Since $\hat v_n\in\cV_n$, for every $t\in[0,1]$,
\begin{align*}
\lVert \hat Y_t-Y_t\rVert
&\le
\int_0^t
\lVert \hat v_{n,s}(\hat Y_s)-\hat v_{n,s}(Y_s)\rVert
\d s
+
\int_0^t
\lVert \hat v_{n,s}(Y_s)-v_s(Y_s)\rVert
\d s
\\
&\le
L\int_0^t \lVert \hat Y_s-Y_s\rVert \d s
+
\int_0^t
\lVert \hat v_{n,s}(Y_s)-v_s(Y_s)\rVert
\d s.
\end{align*}
Gronwall's inequality therefore gives
\begin{align*}
\lVert \hat Y_1-Y_1\rVert
\le
\exp(L)
\int_0^1
\lVert \hat v_{n,t}(Y_t)-v_t(Y_t)\rVert
\d t.
\end{align*}
Squaring this inequality, applying Cauchy-Schwarz in time,
and taking conditional expectation over $U^\circ$ yield
\begin{align*}
W_2^2(\Q_n,\P)
&\le
\E_{U^\circ}\big[
\lVert \hat Y_1-Y_1\rVert^2
\mid\cO
\big]
\\
&\le
\exp(2L)
\E_{U^\circ}\Big[
\Big(
\int_0^1
\lVert \hat v_{n,t}(Y_t)-v_t(Y_t)\rVert
\d t
\Big)^2
\mid\cO
\Big]
\\
&\le
\exp(2L)
\int_0^1
\E_{U^\circ}\big[
\lVert \hat v_{n,t}(Y_t)-v_t(Y_t)\rVert^2
\mid\cO
\big]
\d t
\\
&=
\exp(2L)
\int_0^1
\lVert \hat v_{n,t}-v_t\rVert_{L^2(\P_t)}^2
\d t
\le
\exp(2L)
\Big(
\alpha_n^{\rm FM}
+
2\Delta_n^{\rm FM}
+
\varepsilon_n^{\rm FM}
\Big).
\end{align*}
Taking square roots proves the claimed Wasserstein bound.
The consistency conclusion follows from the approximation,
generalization, and optimization conditions, and the
weak-convergence conclusion follows from
Lemma~\ref{lemma:wq-implies-weak}.

\end{proof}

\subsubsection{Proof of Lemma~\ref{lem:FM_density_bound}}

\begin{proof}
By Assumption~\ref{def:controlled_neural_FM_class}, the learned time-$t$ map
$\hat\Phi_{n,t}$ is a $C^1$-diffeomorphism. Let
\begin{align*}
J_{n,t}(u)
:=
D_u\hat\Phi_{n,t}(u).
\end{align*}
For almost every $t$, the Jacobian satisfies
\begin{align*}
\frac{\d}{\d t}J_{n,t}(u)
=
D \hat v_{n,t}(\hat\Phi_{n,t}(u))
J_{n,t}(u)
\quad\text{with initial condition}\quad
J_{n,0}(u)=I_p.
\end{align*}
By Liouville's formula,
\begin{align*}
\det J_{n,t}(u)
=
\exp\Big\{
\int_0^t
\tr\big(D \hat v_{n,s}(\hat\Phi_{n,s}(u))\big) \d s \Big\}
=
\exp\Big\{
\int_0^t
\nabla\cdot\hat v_{n,s}(\hat\Phi_{n,s}(u))\d s \Big\}>0.
\end{align*}
The change-of-variables formula therefore yields, for Lebesgue-almost every
$u$,
\begin{align*}
q_n\big(
\hat\Phi_{n,1}(u)
\big)
=
p_0(u)
\exp\Big\{
-
\int_0^1
\nabla\cdot\hat v_{n,t}\big(
\hat\Phi_{n,t}(u)
\big)
\d t
\Big\}.
\end{align*}
Since
$\lvert\nabla\cdot\hat v_{n,t}(x)\rvert\le pL$, the result follows.
\end{proof}

\subsubsection{Proof of Theorem~\ref{thm:FM_bootstrap_CWC}}

\begin{proof}
Lemma~\ref{lem:FM_training_implies_weak_convergence} gives
$W_2(\Q_n,\P)=o_{\P_\cO}(1)$ and hence
$\Q_n\rightsquigarrow\P$ in $\P_\cO$-probability. In particular,
$\Q_n$ is asymptotically tight in $\P_\cO$-probability. Moreover,
Lemma~\ref{lem:FM_density_bound} gives a deterministic constant $C<\infty$
such that
\begin{align*}
\lVert q_n\rVert_{L^\infty(\lambda)}
&\le C
\end{align*}
realization-wise. Hence, for every $R<\infty$,
\begin{align*}
\int_{B_R}
q_n(z)^2
\d z
&\le
\lVert q_n\rVert_{L^\infty(\lambda)}
\int_{B_R}
q_n(z)
\d z
\le C.
\end{align*}
The target-density lower bound and
Lemma~\ref{lem:local_density_metric_transfer} therefore imply that, for every
$\epsilon >0$,
\begin{align*}
\lim_{\delta\downarrow0}
\limsup_{n\to\infty}
\P_\cO\big(
\sup_{\rho_\P(f,g)<\delta}
\rho_{\Q_n}(f,g)>\epsilon
\big)
=0.
\end{align*}
Theorem~\ref{thm:unif_entropy} then gives conditional asymptotic
$\rho_\P$-equicontinuity.

The same density bound implies the localized density-uniform-integrability
condition in Corollary~\ref{cor:localized_Lebesgue_FD_cov}. Together with
$\Q_n\rightsquigarrow\P$ in $\P_\cO$-probability, this yields finite-dimensional
covariance consistency. The conclusion now follows from
Theorem~\ref{thm:cvg_emp_proc}.
\end{proof}

\subsection{Proofs of the results in Section~\ref{sec:diffusion}}

\subsubsection{Proof of Lemma~\ref{lem:diffusion_TV_consistency}}

\begin{proof}
Conditionally on $\cO$,
Corollary~\ref{cor:diffusion_fitted_reverse_KL} applied
realization-wise gives \eqref{eq:diffusion_KL_bound}. 
Pinsker's inequality and the triangle
inequality then give \eqref{eq:diffusion_TV_decomposition}.

It remains to verify that the two deterministic endpoint errors vanish. With a slight abuse of notations, let
$p$ also denote a Lebesgue density of $\P$ and write $a_t := e^{-t}$ with $\sigma_t := \sqrt{1-e^{-2t}}$. The density of $\P_t$ can be
written as 
\begin{align*}
p_t
=
\mathsf D_{a_t}p
*
\varphi_{\sigma_t},
\qquad
(\mathsf D_a p)(x)
:=
a^{-p}p(x/a),
\end{align*}
where the convolution follows from the independence of $Z$ and $\xi$, and
$\varphi_\sigma$ denotes the density of
${\rm N}(0,\sigma^2I_p)$.
The triangle inequality and Young's convolution inequality give
\begin{align*}
\lVert p_t-p\rVert_{L^1(\lambda)}
&\le
\big\lVert
(D_{a_t}p-p)*\varphi_{\sigma_t}
\big\rVert_{L^1(\lambda)}
+
\big\lVert p*\varphi_{\sigma_t}-p\big\rVert_{L^1(\lambda)}
\\
&\le
\lVert D_{a_t}p-p\rVert_{L^1(\lambda)}
+
\lVert p*\varphi_{\sigma_t}-p\rVert_{L^1(\lambda)}.
\end{align*}
We now verify that both terms here vanish. Fix
$\epsilon>0$ and choose a continuous compactly supported
function $g:\bR^p\to\bR$ such that
$\lVert p-g\rVert_{L^1(\lambda)} < \epsilon$. For $y\in\bR^p$, define $(\tau_y f)(x) := f(x-y)$.
Both $\mathsf D_a$ and $\tau_y$ are isometries on $L^1(\lambda)$.
Consequently,
\begin{align*}
\lVert \mathsf D_a p-p\rVert_{L^1(\lambda)}
&\le
\lVert \mathsf D_a(p-g)\rVert_{L^1(\lambda)}
+
\lVert \mathsf D_ag-g\rVert_{L^1(\lambda)}
+
\lVert p-g\rVert_{L^1(\lambda)}
\\
&=
\lVert \mathsf D_ag-g\rVert_{L^1(\lambda)}
+
2\lVert p-g\rVert_{L^1(\lambda)},
\end{align*}
where the first term vanishes as $a\to 1$ by the dominated convergence theorem. Similarly 
\begin{align*}
\lVert \tau_yp-p\rVert_{L^1(\lambda)}
&\le
\lVert \tau_y(p-g)\rVert_{L^1(\lambda)}
+
\lVert \tau_yg-g\rVert_{L^1(\lambda)}
+
\lVert p-g\rVert_{L^1(\lambda)}
\\
&=
\lVert \tau_yg-g\rVert_{L^1(\lambda)}
+
2\lVert p-g\rVert_{L^1(\lambda)}.
\end{align*}
and the first term on the right hand side vanishes as $y\to0$. Since $\epsilon>0$ is arbitrary, it follows that
\begin{align*}
\lVert \mathsf D_ap-p\rVert_{L^1(\lambda)}
&\longrightarrow
0
\qquad
\text{as }a\to1,
\yestag
\label{eq:L1_dilation_continuity}
\\
\lVert \tau_yp-p\rVert_{L^1(\lambda)}
&\longrightarrow
0
\qquad
\text{as }y\to0.
\yestag
\label{eq:L1_translation_continuity}
\end{align*}
For the Gaussian convolution, Minkowski's integral inequality gives
\begin{align*}
\lVert p*\varphi_\sigma-p\rVert_{L^1(\lambda)}
&\le
\int_{\bR^p}
\lVert \tau_yp-p\rVert_{L^1(\lambda)}
\varphi_\sigma(y)\d y.
\end{align*}
Given $\epsilon>0$, equation
\eqref{eq:L1_translation_continuity} yields some $\delta>0$ such that
\begin{align*}
\sup_{\lVert y\rVert<\delta}
\lVert \tau_yp-p\rVert_{L^1(\lambda)}
<
\epsilon.
\end{align*}
Since $\lVert \tau_yp-p\rVert_{L^1(\lambda)}
\le 2\lVert p\rVert_{L^1(\lambda)}=2$, we obtain
\begin{align*}
\lVert p*\varphi_\sigma-p\rVert_{L^1(\lambda)}
\le
\epsilon
+
2
\int_{\lVert y\rVert\ge\delta}
\varphi_\sigma(y)\d y
&=
\epsilon
+
2
\int_{\lVert z\rVert\ge\delta/\sigma}
\varphi_1(z)\d z
\longrightarrow
\epsilon
\end{align*}
as $\sigma\downarrow0$. Letting $\epsilon\downarrow0$ gives
\begin{align*}
\lVert p*\varphi_\sigma-p\rVert_{L^1(\lambda)}
\longrightarrow
0
\qquad
\text{as }\sigma\downarrow0.
\yestag
\label{eq:Gaussian_approximate_identity}
\end{align*}
Because $a_t\to1$ and $\sigma_t\to0$ as $t\downarrow0$,
\eqref{eq:L1_dilation_continuity} and
\eqref{eq:Gaussian_approximate_identity} imply
$\lVert p_t-p\rVert_{L^1(\lambda)}
\to 0$.
Hence
$d_{\rm TV}(\P_{\tau_n},\P)\to0$.

For the other endpoint, convexity of KL divergence in its first argument and
the Gaussian KL formula give, for every $t>0$,
\begin{align*}
D_{\rm KL}(\P_t\|\gamma_p)
&\le
\E\big[
D_{\rm KL}\big(
{\rm N}\big(
e^{-t}Z,(1-e^{-2t})I_p
\big)
\big\|
{\rm N}(0,I_p)
\big)
\big]
\\
&=
\frac{e^{-2t}}{2}
\P\big[
\lVert Z\rVert^2
\big]
+
\frac{p}{2}
\big\{
-e^{-2t}
-
\log(1-e^{-2t})
\big\}
\longrightarrow
0
\end{align*}
as $t\to\infty$. Thus
$D_{\rm KL}(\P_{T_n}\|\gamma_p)\to0$. Combining these endpoint limits with
\eqref{eq:diffusion_score_consistency} yields the result.
\end{proof}

\subsubsection{Proof of Lemma~\ref{lem:TV_uniform_entropy_CWC}}

\begin{proof}
By Corollary~\ref{cor:divergence_FD_cov}, total-variation consistency implies
the finite-dimensional covariance consistency condition
\eqref{eq:FD-cov}. It remains to verify the metric-transfer condition
\eqref{eq:local_metric_transfer_Qn}.

Fix $f,g\in\cF$ and write $h_{f,g} := f-g- \P[f-g]$. By definition,
\begin{align*}
\P[h_{f,g}^2]
=
\rho_\P(f,g)^2,
\qquad
\lVert h_{f,g}\rVert_\infty
\le
4B.
\end{align*}
Moreover, since one has $\rho_{\Q_n}(f,g)^2
\le
\Q_n[h_{f,g}^2]$,
\begin{align*}
\rho_{\Q_n}(f,g)^2
\le
\rho_\P(f,g)^2
+
\big\lvert
(\Q_n-\P)[h_{f,g}^2]
\big\rvert
\le
\rho_\P(f,g)^2
+
32B^2
d_{\rm TV}(\Q_n,\P).
\end{align*}
Consequently, for every $\delta>0$,
\begin{align*}
\sup_{\rho_\P(f,g)<\delta}
\rho_{\Q_n}(f,g)^2
\le
\delta^2
+
32B^2
d_{\rm TV}(\Q_n,\P).
\end{align*}
This proves \eqref{eq:local_metric_transfer_Qn}. Theorem~\ref{thm:unif_entropy}
therefore gives conditional asymptotic $\rho_\P$-equicontinuity, and the
conclusion follows from Theorem~\ref{thm:cvg_emp_proc}.
\end{proof}

\subsubsection{Proof of Theorem~\ref{thm:diffusion_bootstrap_CWC}}

\begin{proof}
Lemma~\ref{lem:diffusion_TV_consistency} gives
$d_{\rm TV}(\Q_n,\P)=o_{\P_\cO}(1)$. The conclusion follows from
Lemma~\ref{lem:TV_uniform_entropy_CWC}.
\end{proof}

\subsection{Proofs of the results in Section~\ref{sec:WGAN}}

\subsubsection{Proof of Lemma~\ref{lem:WGAN_training_implies_W1}}

\begin{proof}
We first note that $\Q_n$ has a finite first moment: conditionally
on $\cO$, Assumption~\ref{ass:WGAN_model}(i) gives
\begin{align*}
\P_U
\big[
\big\lVert
\hat G_n(U)
\big\rVert
\big]
\le
\big\lVert
\hat G_n(0)
\big\rVert
+
{\rm Lip}(\hat G_n)
\P_U
\big[
\lVert U\rVert
\big]
<
\infty.
\end{align*}

We next use the contraction property of Wasserstein distance under a
Lipschitz map. Fix a realization of $\cO$, and let $\pi$ be any coupling
of $\bP_{U,n}$ and $\P_U$. The image of $\pi$ under the map
$(u,v)
\mapsto
(\hat G_n(u), \hat G_n(v))$
is a coupling of
$\hat G_n\#\bP_{U,n}$ and $\hat G_n\#\P_U$, hence
\begin{align*}
W_1
\big(
\hat G_n\#\bP_{U,n},
\hat G_n\#\P_U
\big)
\le
\int
\big\lVert
\hat G_n(u)
-
\hat G_n(v)
\big\rVert
\d\pi(u,v)
\le
{\rm Lip}(\hat G_n)
\int
\lVert u-v\rVert
\d\pi(u,v).
\end{align*}
Taking the infimum over all such couplings gives
\begin{align*}
W_1
\big(
\hat G_n\#\bP_{U,n},
\hat G_n\#\P_U
\big)
\le
{\rm Lip}(\hat G_n)
W_1
\big(
\bP_{U,n},
\P_U
\big).
\yestag
\label{eq:WGAN_pushforward_contraction}
\end{align*}
Since $\Q_n=\hat G_n\#\P_U$, the triangle inequality and
\eqref{eq:WGAN_pushforward_contraction} yield
\eqref{eq:WGAN_consistency_decomposition}. By
Lemma~\ref{lem:empirical_W1_consistency}, $W_1(\bP_n,\P) \to0$ almost surely. 
For the second term in
\eqref{eq:WGAN_consistency_decomposition}, by Lemma~\ref{lem:Kantorovich_Rubinstein} and
\eqref{eq:empirical_WGAN_objective},
\begin{align*}
&W_1
\big(
\bP_n,
\hat G_n\#\bP_{U,n}
\big)
=
\sup_{\varphi\in{\rm Lip}_1^0(\bR^p)}
\widehat{\mathcal L}_n
\big(
\hat G_n,\varphi
\big)
\\
&~~\le
\Big\{
\sup_{\varphi\in{\rm Lip}_1^0(\bR^p)}
\widehat{\mathcal L}_n
\big(
\hat G_n,\varphi
\big)
-
\widehat{\mathcal L}_n
\big(
\hat G_n,\hat D_n
\big)
\Big\}
+
\big\lvert
\widehat{\mathcal L}_n
\big(
\hat G_n,\hat D_n
\big)
\big\rvert
=
o_{\P_\cO}(1)
\end{align*}
by Assumption~\ref{ass:WGAN_model}(ii)-(iii).
The third is $o_{\P_\cO}(1)$ by Assumption~\ref{ass:WGAN_model}(iv). Thus
$W_1(\Q_n,\P)=o_{\P_\cO}(1)$. The weak-convergence conclusion follows from
Lemma~\ref{lemma:wq-implies-weak}.
\end{proof}

\subsubsection{Proof of Lemma~\ref{lem:W1_modulus_CWC}}

\begin{proof}
Write $d_n
:=
W_1(\Q_n,\P)$.
We first record the transport bound used below. Suppose a function
$\psi:\bR^p\to\bR$ satisfies
$\lvert\psi(x)-\psi(y)\rvert
\le
a\omega (\lVert x-y\rVert)$ for some $a<\infty$. For every $\epsilon>0$, choose a coupling
$\pi_{n,\epsilon}$ of $\Q_n$ and $\P$ satisfying
\begin{align*}
\int
\lVert x-y\rVert
\d\pi_{n,\epsilon}(x,y)
\le
d_n+\epsilon.
\end{align*}
By monotonicity, concavity, and Jensen's inequality,
\begin{align*}
\big\lvert
(\Q_n-\P)[\psi]
\big\rvert
\le
a
\int
\omega\big(
\lVert x-y\rVert
\big)
\d\pi_{n,\epsilon}(x,y)
\le
a\omega\Big(
\int
\lVert x-y\rVert
\d\pi_{n,\epsilon}(x,y)
\Big)
\le
a\omega(d_n+\epsilon).
\end{align*}
Letting $\epsilon\downarrow0$ gives
\begin{align*}
\big\lvert
(\Q_n-\P)[\psi]
\big\rvert
\le
a\omega(d_n).
\yestag
\label{eq:W1_modulus_bound}
\end{align*}

We first verify finite-dimensional covariance consistency. Fix
$f,g\in\cF$. Since $\lVert f\rVert_\infty\vee
\lVert g\rVert_\infty\le B$, the product satisfies
\begin{align*}
\lvert f(x)g(x)-f(y)g(y)\rvert
\le
\lvert f(x)\rvert
\lvert g(x)-g(y)\rvert
+
\lvert g(y)\rvert
\lvert f(x)-f(y)\rvert
\le
2B\omega\big(
\lVert x-y\rVert
\big).
\end{align*}
Applying \eqref{eq:W1_modulus_bound} to $f$, $g$, and $fg$ gives
\begin{align*}
&
\big\lvert
\big\{
\Q_n[fg]-\Q_n[f]\Q_n[g]
\big\}
-
\big\{
\P[fg]-\P[f]\P[g]
\big\}
\big\rvert
\\
&\quad\le
\big\lvert
(\Q_n-\P)[fg]
\big\rvert
+
\big\lvert\Q_n[f]\big\rvert
\big\lvert(\Q_n-\P)[g]\big\rvert
+
\big\lvert\P[g]\big\rvert
\big\lvert(\Q_n-\P)[f]\big\rvert
\\
&\quad\le
4B\omega(d_n)
=
o_{\P_\cO}(1).
\end{align*}
Thus \eqref{eq:FD-cov} holds.

For the local semimetric, write $u:=f-g$. Then
$\lVert u\rVert_\infty\le2B$, $u$ has modulus $2\omega$, and
\begin{align*}
\lvert u(x)^2-u(y)^2\rvert
=
\lvert u(x)-u(y)\rvert
\lvert u(x)+u(y)\rvert
\le
8B\omega\big(
\lVert x-y\rVert
\big).
\end{align*}
Therefore,
\begin{align*}
\big\lvert
\rho_{\Q_n}(f,g)^2
-
\rho_\P(f,g)^2
\big\rvert
&\le
\big\lvert
(\Q_n-\P)[u^2]
\big\rvert
+
\big\lvert
\Q_n[u]^2-\P[u]^2
\big\rvert
\\
&\le
8B\omega(d_n)
+
2\omega(d_n)
\big\{
\lvert\Q_n[u]\rvert
+
\lvert\P[u]\rvert
\big\}
\le
16B\omega(d_n).
\end{align*}
Consequently, for every $\delta>0$,
\begin{align*}
\sup_{\rho_\P(f,g)<\delta}
\rho_{\Q_n}(f,g)^2
\le
\delta^2
+
16B\omega(d_n).
\end{align*}
Since $d_n=o_{\P_\cO}(1)$ and $\omega$ is continuous at zero, $\omega(d_n)=o_{\P_\cO}(1)$. For every $\epsilon>0$ and every
$0<\delta<\epsilon$,
\begin{align*}
\P_\cO\Big(
\sup_{\rho_\P(f,g)<\delta}
\rho_{\Q_n}(f,g)>\epsilon
\Big)
\le
\P_\cO\Big(
\omega(d_n)
>
\frac{\epsilon^2-\delta^2}{16B}
\Big)
\to0.
\end{align*}
Thus \eqref{eq:local_metric_transfer_Qn} holds.
Theorem~\ref{thm:unif_entropy} gives \eqref{eq:CAE}, then the
conclusion follows from Theorem~\ref{thm:cvg_emp_proc}.
\end{proof}

\subsubsection{Proof of Theorem~\ref{thm:WGAN_bootstrap_CWC}}

\begin{proof}
Lemma~\ref{lem:WGAN_training_implies_W1} gives $W_1(\Q_n,\P) = o_{\P_\cO}(1)$.
The conclusion follows from Lemma~\ref{lem:W1_modulus_CWC}. For the last
statement, take $\omega(t)=Lt$.
\end{proof}

\section{Additional technical results}

This section provides model-specific verification results for the
applications in Section~\ref{sec:app}. For triangular normalizing
flows, we give a concrete support-matched class satisfying the model
assumptions. For flow matching and WGANs, we provide controlled neural
templates that reduce the corresponding assumptions to approximation,
generalization, and optimization conditions. For score-based diffusion
models, we establish the reverse-diffusion regularity and KL-divergence
bounds used in the main text.

\subsection{Triangular normalizing flows}
\label{sec:NF-apdx}

We give a example satisfying
Assumptions~\ref{ass:triangular_data} and
\ref{ass:triangular_model}. The normalized integral construction below is an exponential-link instance
of the parameterizations in \citet[equation~(4.9)]{zech2022sparse} and
\citet[equation~(4.1)]{wang2022minimax}. It is restricted to a
hypercube so that all marginalizations are over fixed domains; this avoids the
boundary degeneration that can occur for general convex supports. Any known hyperrectangle can be reduced to the unit cube by an affine
bijection, with the fitted generator subsequently conjugated back to the
original coordinates.

For any integer $r\ge1$, let $C^1([0,1]^r)$ be the class of continuous functions on
$[0,1]^r$ that are continuously differentiable on $(0,1)^r$ and whose
first-order partial derivatives admit continuous extensions to $[0,1]^r$.

For $j\in\zahl{p}$ and
$g_j\in C^1([0,1]^{p-j+1})$, define the conditional log-normalizer
\begin{align*}
A_jg_j(z_{j+1:p})
:={}
\log
\int_0^1
\exp\big\{g_j(t,z_{j+1:p})\big\}
\d t,
\end{align*}
and the $j$-th triangular coordinate
\begin{align*}
S_{g,j}(z_j,z_{j+1:p})
:={}
\int_0^{z_j}
\exp\big\{
g_j(t,z_{j+1:p})
-
A_jg_j(z_{j+1:p})
\big\}
\d t.
\yestag
\label{eq:integrated_exponential_coordinate}
\end{align*}
For $j=p$, the conditioning argument $z_{p+1:p}$ is void and $A_p g_p$ is a
scalar. Write $S_g := (S_{g,1},\ldots,S_{g,p})^\top$.
For constants $B,L<\infty$, define
\begin{align*}
\cG_j(B,L)
:=
\big\{
g\in C^1([0,1]^{p-j+1}):
\lVert g\rVert_\infty\le B, 
\sup_x\lVert\nabla g(x)\rVert\le L
\big\},
\end{align*}
and
$\cT(B,L)
:= \{ S_g: g_j\in\cG_j(B,L), j\in\zahl{p}\}$.

\begin{proposition}
\label{prop:integrated_exponential_flow}
Suppose $\cZ=\cU=[0,1]^p$ and $\P_U={\rm Unif}([0,1]^p)$.
Suppose also that the target density satisfies $p\in C^1([0,1]^p)$ and $\inf_{z\in[0,1]^p}p(z)>0$. Then Assumption~\ref{ass:triangular_data} holds. Moreover, there exist
finite $B,L>0$ such that, upon taking
$\cT_n:=\cT(B,L)$ for every $n$,
Assumption~\ref{ass:triangular_model}(i)-(iv) holds.
\end{proposition}

\begin{proof}
We first verify Assumption~\ref{ass:triangular_data}.
The set $[0,1]^p$ is compact and convex with nonempty interior.
Moreover, continuity and strict positivity of $p$ on the compact cube
give constants
\begin{align*}
0<\underline p
\le
p(z)
\le
\overline p
<
\infty,
\qquad
z\in[0,1]^p.
\end{align*}
In particular, $\log p$ is bounded and hence
$\P[\lvert\log p\rvert]<\infty$.
Thus Assumption~\ref{ass:triangular_data} holds.

We next verify
Assumption~\ref{ass:triangular_model}(i)-(iv).

First, the uniform noise distribution satisfies
Assumption~\ref{ass:triangular_model}(i) with
$\underline p_U=\overline p_U=1$. 
For every $g_j$, the construction in
\eqref{eq:integrated_exponential_coordinate} gives 
\begin{align*}
S_{g,j}(0,z_{j+1:p})=0, \qquad
S_{g,j}(1,z_{j+1:p})=1, 
\end{align*}
and
\begin{align*}
D_j S_{g,j}(z_j,z_{j+1:p})
=
\exp\big\{
g_j(z_j,z_{j+1:p})
-
A_jg_j(z_{j+1:p})
\big\}
>0.
\yestag
\label{eq:integrated_exponential_derivative}
\end{align*}
Thus for every fixed $z_{j+1:p}$, the map
$z_j\mapsto S_{g,j}(z_j,z_{j+1:p})$ is a strictly increasing bijection from
$[0,1]$ onto $[0,1]$. Starting with the last coordinate and proceeding
backward therefore shows that $S_g$ is a bijection from $[0,1]^p$ onto
$[0,1]^p$. The endpoint identities also show that
$S_g((0,1)^p)=(0,1)^p$. The map is continuous, so compactness of the domain
implies that its inverse is continuous. Moreover, $S_g$ is $C^1$, and its
triangular Jacobian has strictly positive determinant on $(0,1)^p$. The
inverse function theorem therefore shows that its restriction to $(0,1)^p$
is a $C^1$-diffeomorphism.
Hence every $S_g\in\cT(B,L)$ satisfies
Definition~\ref{def:triangular_normalizing_flow}.

If $\lVert g_j\rVert_\infty\le B$, then
$-B \le A_jg_j(z_{j+1:p}) \le B$.
Consequently, \eqref{eq:integrated_exponential_derivative} gives $e^{-2B} \le D_j S_{g,j}(z) \le e^{2B}$ uniformly over $j$, $z$, and $S_g\in\cT(B,L)$. Thus Assumption~\ref{ass:triangular_model}(ii) holds with
$M_S=e^{2B}$. The lower bound $D_j S_{g,j}\ge e^{-2B}$ will also be useful
below.

Second, we construct an exact population map in the class. For
$j\in\zahl{p}$, define the trailing marginal density
\begin{align*}
p_{j:p}(z_{j:p})
:={}
\int_{[0,1]^{j-1}}
p(x_{1:j-1},z_{j:p})
\d x_{1:j-1},
\end{align*}
and set $p_{p+1:p}\equiv1$. Define the reverse conditional density
\begin{align*}
r_j^\star(z_j\mid z_{j+1:p})
:={}
\frac{p_{j:p}(z_{j:p})}
{p_{j+1:p}(z_{j+1:p})}.
\end{align*}
All integrations are over fixed cubes. Hence differentiation under the
integral sign shows that every $p_{j:p}$ is $C^1$. Since $p$ is continuous
and bounded away from zero on the compact unit cube, the functions
$p_{j:p}$ are uniformly bounded above and away from zero. It follows that
$r_j^\star$ is positive and $C^1$, and that both
$\log r_j^\star$ and its gradient are uniformly bounded over
$j\in\zahl{p}$. Therefore, for some finite constants $B$ and $L$,
\begin{align*}
g_j^\star(z_j,z_{j+1:p})
:={}
\log r_j^\star(z_j\mid z_{j+1:p})
\in
\cG_j(B,L)
\end{align*}
for every $j\in\zahl{p}$. Because $r_j^\star(\cdot\mid z_{j+1:p})$ is a conditional density,
\begin{align*}
A_jg_j^\star(z_{j+1:p})
=
\log
\int_0^1
r_j^\star(t\mid z_{j+1:p})
\d t
=
0.
\end{align*}
Thus $D_j S_{g^\star,j}(z)
=
r_j^\star(z_j\mid z_{j+1:p})$.
Since the base density is one, the density induced by $S_{g^\star}$ is
\begin{align*}
q_{S_{g^\star}}(z)
=
\prod_{j=1}^p
r_j^\star(z_j\mid z_{j+1:p})
=
\prod_{j=1}^p
\frac{p_{j:p}(z_{j:p})}
{p_{j+1:p}(z_{j+1:p})}
=
p(z).
\end{align*}
Therefore $\Q_{S_{g^\star}}=\P$, and
\begin{align*}
\inf_{S\in\cT(B,L)}
D_{\rm KL}(\P\|\Q_S)
=
0.
\end{align*}
This verifies Assumption~\ref{ass:triangular_model}(iii).

It remains to verify the uniform law of large numbers. Under the uniform base
distribution,
\begin{align*}
\ell_{S_g}(z)
=
-
\sum_{j=1}^p
\big\{
g_j(z_{j:p})
-
A_jg_j(z_{j+1:p})
\big\},
\end{align*}
so $\lVert\ell_{S_g}\rVert_\infty\le2pB$. Moreover, for any two functions
$g$ and $\widetilde g$ on the same domain,
\begin{align*}
\big\lVert
A_jg-A_j\widetilde g
\big\rVert_\infty
\le
\lVert g-\widetilde g\rVert_\infty;
\yestag
\label{eq:log_normalizer_lipschitz}
\end{align*}
since with $\delta:=\lVert g-\widetilde g\rVert_\infty$,
\begin{align*}
e^{-\delta}
\int_0^1e^{\widetilde g(t,z_{j+1:p})}\d t
\le
\int_0^1e^{g(t,z_{j+1:p})}\d t
\le
e^{\delta}
\int_0^1e^{\widetilde g(t,z_{j+1:p})}\d t,
\end{align*}
and taking logarithms gives
\eqref{eq:log_normalizer_lipschitz}. Consequently,
\begin{align*}
\big\lVert
\ell_{S_g}-\ell_{S_{\widetilde g}}
\big\rVert_\infty
\le
2
\sum_{j=1}^p
\lVert g_j-\widetilde g_j\rVert_\infty.
\yestag
\label{eq:triangular_loss_lipschitz}
\end{align*}

Let $K_j:=[0,1]^{p-j+1}$.
Fix an arbitrary sequence
$\{\tilde g_m\}_{m\ge1}\subset\cG_j(B,L)$. The uniform sup-norm bound
implies that this sequence is pointwise bounded. Moreover, since $K_j$
is convex, the common gradient bound gives
$\lvert \tilde g_m(x)- \tilde g_m(y)\rvert
\le
L\lVert x-y\rVert$
for every $m\ge1$ and every $x,y\in K_j$. Hence the sequence is
equicontinuous. By the Arzelà-Ascoli theorem,
it has a subsequence that converges uniformly on $K_j$.

We now show that this implies that $\cG_j(B,L)$ is totally bounded
under the sup norm. Otherwise, there would exist $\epsilon_0>0$ and a
sequence $\{\tilde g_m\}_{m\ge1}\subset\cG_j(B,L)$ such that $\lVert \tilde g_{m_1}-\tilde g_{m_2}\rVert_\infty
\ge
\epsilon_0$ for all $m_1\ne m_2$.
The preceding Arzelà-Ascoli argument gives a uniformly convergent
subsequence. Such a subsequence is Cauchy under the sup norm, contradicting
the displayed separation. Therefore, $\cG_j(B,L)$ is totally
bounded under the sup norm.

The finite product of these classes is
totally bounded, and \eqref{eq:triangular_loss_lipschitz} implies that the loss
class
\begin{align*}
\mathcal L(B,L)
:=
\big\{
\ell_{S_g}:S_g\in\cT(B,L)
\big\}
\end{align*}
is also totally bounded under the sup norm.
For any $\epsilon>0$, let
$\ell_1,\ldots,\ell_N$ be a finite $\epsilon$-net for
$\mathcal L(B,L)$. Then
\begin{align*}
\sup_{\ell\in\mathcal L(B,L)}
\big\lvert
(\bP_n-\P)\ell
\big\rvert
\le
\max_{1\le k\le N}
\big\lvert
(\bP_n-\P)\ell_k
\big\rvert
+
2\epsilon.
\end{align*}
For each integer $m\ge1$, choose a finite $m^{-1}$-net
$\ell_{m,1},\ldots,\ell_{m,N_m}$ for $\mathcal L(B,L)$.
By the strong law of large numbers, on an event of probability one,
\begin{align*}
\max_{1\le k\le N_m}
\big\lvert
(\bP_n-\P)\ell_{m,k}
\big\rvert
\longrightarrow
0
\end{align*}
for every $m\ge1$ simultaneously.
On this event,
\begin{align*}
\limsup_{n\to\infty}
\sup_{\ell\in\mathcal L(B,L)}
\big\lvert
(\bP_n-\P)\ell
\big\rvert
\le
\frac{2}{m}
\end{align*}
for every $m\ge1$.
Letting $m\to\infty$ proves
\begin{align*}
\sup_{S\in\cT(B,L)}
\big\lvert
(\bP_n-\P)\ell_S
\big\rvert
\longrightarrow
0
\end{align*}
almost surely and completes the proof.
\end{proof}

Proposition~\ref{prop:integrated_exponential_flow} verifies
Assumption~\ref{ass:triangular_data} and
Assumption~\ref{ass:triangular_model}(i)-(iv).
If, in addition, the fitted map satisfies
Assumption~\ref{ass:triangular_model}(v), then
Theorem~\ref{thm:triangular_flow_bootstrap_CWC} yields conditional
weak convergence of the triangular generative bootstrap process for
every bounded $\P$-Donsker class.

In fact, this class satisfies a stronger two-sided
likelihood-ratio bound, which yields the converse as well.
For every $S_g\in\cT(B,L)$, the preceding proof gives for Lebesgue-almost every $z\in[0,1]^p$,
\begin{align*}
e^{-2Bp}
\le
q_{S_g}(z)
\le
e^{2Bp}.
\end{align*}
Then for every fitted map $\hat S_n\in\cT(B,L)$,
\begin{align*}
\Big\{\sup_{z\in[0,1]^p}p(z) \Big\}^{-1}
e^{-2Bp}
\le
{\d\Q_n}/{\d\P}
\le
\Big\{\inf_{z\in[0,1]^p}p(z)\Big\}^{-1}
e^{2Bp}
\end{align*}
$\P$-almost surely.
Thus the model satisfies the two directional domination conditions in
Section~\ref{subsec:relation_Donsker}. Under Assumption~\ref{ass:triangular_model}(v),
Lemma~\ref{lem:triangular_KL_consistency} also gives
$\Q_n\rightsquigarrow\P$ in $\P_\cO$-probability.
Therefore, Corollary~\ref{cor:Donsker_equivalence_weak_consistency}
implies that, for every measurable class $\cF$ with a bounded envelope,
the triangular-flow bootstrap process converges conditionally if and
only if $\cF$ is $\P$-Donsker.

\subsection{Flow matching}
\label{sec:FM-apdx}

This subsection gives a finite-dimensional neural template for
verifying the flow-matching conditions in
Section~\ref{sec:FM}. The construction uses an affine skip connection
to accommodate linear growth, a smooth residual network for nonlinear
flexibility, and piecewise-linear interpolation in time. Deterministic
norm constraints yield the flow regularity in
Assumption~\ref{def:controlled_neural_FM_class}. We then verify
approximation through a target-specific closure condition and
generalization through a finite-dimensional covering argument.
Numerical optimization remains a separate condition on the fitting
procedure.

\begin{definition}[Controlled affine-residual neural vector fields]
\label{def:FM_neural_template}
Fix an integer $D\ge1$. For each $n$, let
\begin{align*}
0=t_{0,n}<t_{1,n}<\cdots<t_{K_n,n}=1
\end{align*}
be a deterministic partition of $[0,1]$, and let
$\phi_{0,n},\ldots,\phi_{K_n,n}$ be its nodal hat functions. Therefore, $\phi_{k,n}(t)\ge0$ and $\sum_{k=0}^{K_n}\phi_{k,n}(t)=1$
for $0\le t\le1$. 

Fix constants $L_A,M_b<\infty$ and
$\kappa_\ell,\beta_\ell<\infty$ for $1\le\ell\le D+1$, all independent of $n$. Define
\begin{align*}
c_0:=1,
\qquad
c_\ell:=\kappa_\ell c_{\ell-1}+\beta_\ell
\end{align*}
for $1\le\ell\le D+1$.
Set 
\begin{align*}
L_r := \prod_{\ell=1}^{D+1}\kappa_\ell,
\qquad
M_r := c_{D+1}.
\end{align*}
Assume $L_A+L_r \le L$ and $M_b+M_r \le M$,
where $L$ and $M$ are the constants in
Assumption~\ref{def:controlled_neural_FM_class}.

A parameter $\theta\in\Theta_n$ consists of matrices and vectors
\begin{align*}
&\big\{(A_k,b_k): \lVert A_k\rVert_{\rm op}
\le L_A, \lVert b_k\rVert \le M_b, ~0\le k\le K_n\big\},
\\
&\big\{(W_\ell,a_\ell): \lVert W_\ell\rVert_{\rm op}
\le \kappa_\ell, \lVert a_\ell\rVert
\le \beta_\ell, ~1\le\ell\le D+1 \big\},
\end{align*}
where all quantities are vectorized in the usual way. Define
\begin{align*}
A_\theta(t)
:=
\sum_{k=0}^{K_n}A_k\phi_{k,n}(t),
\qquad
b_\theta(t)
:=
\sum_{k=0}^{K_n}b_k\phi_{k,n}(t).
\end{align*}
Let
$\sigma_{\rm act}:\bR\to\bR$ be continuously differentiable, applied coordinatewise,
and satisfy $\sigma_{\rm act}(0)=0$ with $\lVert\sigma_{\rm act}'\rVert_\infty\le1$. For $y=(t,x)\in[0,1]\times\bR^p$, the residual network $r_\theta(t,x)$ is defined as
\begin{align*}
h_{\theta,0}(y)
&:=y,
\\
h_{\theta,\ell}(y)
&:=
\sigma_{\rm act}\big(
W_\ell h_{\theta,\ell-1}(y)+a_\ell
\big), \quad 1\le\ell\le D,
\\
r_\theta(t,x)
&:=
W_{D+1}h_{\theta,D}(t,x)+a_{D+1},
\end{align*}
where $m_{n,0}=p+1$, $m_{n,D+1}=p$, and $m_{n,1},\ldots,m_{n,D}\ge1$ are arbitrary hidden-layer widths. 
Define the vector field by
\begin{align*}
\tilde v_{\theta,t}(x)
:=
A_\theta(t)x+b_\theta(t)+r_\theta(t,x).
\end{align*}
The resulting vector-field class is
\begin{align*}
\cV_n
:=
\big\{
\tilde v_\theta:\theta\in\Theta_n
\big\}.
\end{align*}
\end{definition}

The class $\cV_n$ is a norm-controlled
abstraction of neural vector-field
parametrizations used in continuous-time flow-based generative models
\citep{chen2018neural,grathwohl2018ffjord,lipman2022flow}. The following proposition verifies the velocity regularity conditions required in the
main text. It also records a parameter-sensitivity bound that will be used for
a uniform law of large numbers for the sampled loss. Let
\begin{align*}
N_n^{\rm par}
:=
(K_n+1)(p^2+p)
+
\sum_{\ell=1}^{D+1}
 m_{n,\ell}
\big( m_{n,\ell-1}+1 \big)
\end{align*}
denote the total number of parameters.

\begin{proposition}[Regularity and parameter sensitivity]
\label{prop:FM_template_regularity}
For the class in Definition~\ref{def:FM_neural_template}, the parameter set
$\Theta_n$ is compact and every $\tilde v_\theta\in\cV_n$ satisfies
\begin{align*}
\sup_{t\in[0,1]}
\lVert
\tilde v_{\theta,t}(0)
\rVert
\le M,
\qquad
\sup_{t\in[0,1],x\in\bR^p}
\lVert
\nabla_x\tilde v_{\theta,t}(x)
\rVert_{\rm op}
\le L.
\end{align*}
Consequently, $\cV_n$ satisfies
Assumption~\ref{def:controlled_neural_FM_class}.

Moreover, there exists a constant $C<\infty$, independent of $n$, such that
for every $\theta,\theta'\in\Theta_n$ and every $x\in\bR^p$,
\begin{align*}
\sup_{t\in[0,1]}
\big\lVert
\tilde v_{\theta,t}(x)
-
\tilde v_{\theta',t}(x)
\big\rVert
\le
C\big(1+\lVert x\rVert\big)
\lVert\theta-\theta'\rVert_2;
\yestag
\label{eq:FM_parameter_sensitivity}
\end{align*}
and there exists a constant
$C_0<\infty$, independent of $n$, such that
\begin{align*}
\log \cN \big(
\eta,
\Theta_n,
\lVert\cdot\rVert_2
\big)
\le
N_n^{\rm par}
\log\Big(
1+
\frac{C_0\sqrt{N_n^{\rm par}}}{\eta}
\Big),
\qquad
0<\eta\le1.
\yestag
\label{eq:FM_parameter_entropy}
\end{align*}
\end{proposition}

\begin{proof}
Because the hat functions are nonnegative and form a partition of unity,
\begin{align*}
\sup_{t\in[0,1]}
\lVert A_\theta(t)\rVert_{\rm op}
&\le L_A,
\qquad
\sup_{t\in[0,1]}
\lVert b_\theta(t)\rVert
\le M_b.
\end{align*}
Since $\sigma_{\rm act}(0)=0$ and $\sigma_{\rm act}$ is $1$-Lipschitz, induction over the layers gives
\begin{align*}
\lVert h_{\theta,\ell}(t,x)\rVert
\le
c_\ell\big(1+\lVert x\rVert\big),
\quad
0\le\ell\le D,
\end{align*}
and in particular,
\begin{align*}
\sup_{t\in[0,1]}
\lVert r_\theta(t,0)\rVert
\le M_r.
\end{align*}
The chain rule and $\lVert\sigma_{\rm act}'\rVert_\infty\le1$ also give
\begin{align*}
\sup_{t\in[0,1],x\in\bR^p}
\lVert
\nabla_x r_\theta(t,x)
\rVert_{\rm op}
\le
\prod_{\ell=1}^{D+1}
\lVert W_\ell\rVert_{\rm op}
\le L_r.
\end{align*}
The first conclusion now follows from $L_A+L_r\le L$ and
$M_b+M_r\le M$.

We next prove \eqref{eq:FM_parameter_sensitivity}. Write
$d:=\lVert\theta-\theta'\rVert_2$. Every individual matrix difference has
operator norm at most $d$, and every individual vector difference has
Euclidean norm at most $d$. The partition-of-unity property gives
\begin{align*}
\sup_{t\in[0,1]}
\lVert A_\theta(t)-A_{\theta'}(t)\rVert_{\rm op}
&\le d,
\qquad
\sup_{t\in[0,1]}
\lVert b_\theta(t)-b_{\theta'}(t)\rVert
\le d.
\end{align*}
For the residual network, the $1$-Lipschitz property of $\sigma_{\rm act}$ yields, for
$1\le\ell\le D$,
\begin{align*}
&
\lVert
h_{\theta,\ell}(t,x)
-
h_{\theta',\ell}(t,x)
\rVert
\\
&\le
\lVert
W_\ell h_{\theta,\ell-1}(t,x)+a_\ell
-
W'_\ell h_{\theta',\ell-1}(t,x)-a'_\ell
\rVert
\\
&\le
\lVert W_\ell\rVert_{\rm op}
\lVert
h_{\theta,\ell-1}(t,x)
-
h_{\theta',\ell-1}(t,x)
\rVert
+
\lVert W_\ell-W'_\ell\rVert_{\rm op}
\lVert h_{\theta',\ell-1}(t,x)\rVert
+
\lVert a_\ell-a'_\ell\rVert
\\
&\le
\kappa_\ell
\lVert
h_{\theta,\ell-1}(t,x)
-
h_{\theta',\ell-1}(t,x)
\rVert
+
\big\{
c_{\ell-1}(1+\lVert x\rVert)+1
\big\}d.
\end{align*}
Starting from the common input
$h_{\theta,0}=h_{\theta',0}$ and iterating this inequality, followed by the
same calculation for the output layer gives
\begin{align*}
\sup_{t\in[0,1]}
\lVert
r_\theta(t,x)-r_{\theta'}(t,x)
\rVert
&\le
C_1\big(1+\lVert x\rVert\big)d
\end{align*}
for a constant $C_1$ depending only on the fixed depth and the fixed
layerwise bounds. Combining the affine and residual parts proves
\eqref{eq:FM_parameter_sensitivity}.

Because the hat functions are continuous, the maps
$t\mapsto A_\theta(t)$ and $t\mapsto b_\theta(t)$ are continuous.
Moreover, an induction over the network layers shows that
$(t,x)\mapsto r_\theta(t,x)$ is continuous. Since
$\sigma_{\rm act}'$ is continuous, the chain-rule representation of
$\nabla_x r_\theta(t,x)$ and another induction over the layers show that
$(t,x)\mapsto\nabla_x r_\theta(t,x)$ is continuous. Consequently,
both
\begin{align*}
(t,x)
\longmapsto
\tilde v_{\theta,t}(x),
\qquad
(t,x)
\longmapsto
\nabla_x\tilde v_{\theta,t}(x)
\end{align*}
are continuous on $[0,1]\times\bR^p$.
In particular, they satisfy the joint
measurability requirements in
Assumption~\ref{def:controlled_neural_FM_class}.

Finally, for each fixed $n$, the set $\Theta_n$ is a closed subset of the
finite-dimensional Euclidean space $\bR^{N_n^{\rm par}}$, since all
of its defining norm constraints are closed. Moreover, every scalar
parameter coordinate is bounded by a constant independent of $n$. Hence $\Theta_n
\subset
[-C_0,C_0]^{N_n^{\rm par}}$ for some fixed $C_0<\infty$ and is compact.

Partition each coordinate interval
$[-C_0,C_0]$ into subintervals of length at most
$\eta/\sqrt{N_n^{\rm par}}$. Each resulting grid cell has Euclidean diameter at
most $\eta$, and the number of such cells is bounded by $(1+{2C_0\sqrt{N_n^{\rm par}}}/{\eta})^{N_n^{\rm par}}$.
Choosing one point of $\Theta_n$ from each grid cell that intersects
$\Theta_n$ gives an $\eta$-net. Therefore, after increasing $C_0$ if
necessary,
\begin{align*}
\log \cN \big(
\eta,\Theta_n,\lVert\cdot\rVert_2
\big)
\le
N_n^{\rm par}
\log\Big(
1+
\frac{C_0\sqrt{N_n^{\rm par}}}{\eta}
\Big),
\end{align*}
which proves~\eqref{eq:FM_parameter_entropy}.
\end{proof}

We next give a target-specific sufficient condition for the approximation
part of Assumption~\ref{ass:controlled_neural_FM_training}. The statement is
formulated as a closure condition.

\begin{proposition}[Approximation by the affine-residual sieve]
\label{prop:FM_template_approximation}
Assume
$\E[
\lVert Z\rVert^2
+
\lVert U\rVert^2]
<\infty$.
Suppose the mesh of the partition in
Definition~\ref{def:FM_neural_template} converges to zero and the true
marginal velocity admits a decomposition
\begin{align*}
v_t(x)
=
A_\star(t)x+b_\star(t)+r_\star(t,x),
\end{align*}
where
$A_\star:[0,1]\to\bR^{p\times p}$ and
$b_\star:[0,1]\to\bR^p$ are continuous, with
$\bR^{p\times p}$ equipped with the operator norm and
$\bR^p$ equipped with the Euclidean norm, and satisfy
\begin{align*}
\sup_{t\in[0,1]}
\lVert A_\star(t)\rVert_{\rm op}
\le L_A,
\qquad
\sup_{t\in[0,1]}
\lVert b_\star(t)\rVert
\le M_b.
\end{align*}
Assume also that there exists $C_\star<\infty$ such that for all
$(t,x)\in[0,1]\times\bR^p$,
\begin{align*}
\lVert r_\star(t,x)\rVert
\le
C_\star\big(1+\lVert x\rVert\big),
\end{align*}
and that there exist radii $R_n\uparrow\infty$ and residual networks $r_n$
obeying the constraints in Definition~\ref{def:FM_neural_template} such that
\begin{align*}
\sup_{t\in[0,1],x\in B_{R_n}}
\lVert
r_n(t,x)-r_\star(t,x)
\rVert
\longrightarrow0.
\yestag
\label{eq:FM_residual_closure}
\end{align*}
Then
\begin{align*}
\alpha_n^{\rm FM}
=
\inf_{\tilde v\in\cV_n}
\big\{
\cR^{\rm FM}(\tilde v)-\cR^{\rm FM}(v)
\big\}
\longrightarrow0.
\end{align*}
\end{proposition}

\begin{proof}
Let
\begin{align*}
A^{(n)}(t)
:=
\sum_{k=0}^{K_n}
A_\star(t_{k,n})\phi_{k,n}(t),
\qquad
b^{(n)}(t)
:=
\sum_{k=0}^{K_n}
b_\star(t_{k,n})\phi_{k,n}(t).
\end{align*}
Write
$h_n
:=
\max_{1\le k\le K_n}
(t_{k,n}-t_{k-1,n})
\to 0$ 
to denote the mesh size of the partition. For any fixed
$t\in[t_{j,n},t_{j+1,n}]$, since the nodal hat functions are
nonnegative and form a partition of unity, only
$\phi_{j,n}(t)$ and $\phi_{j+1,n}(t)$ are nonzero with $\phi_{j,n}(t)+\phi_{j+1,n}(t)=1$, and
\begin{align*}
A^{(n)}(t)
=
\phi_{j,n}(t)A_\star(t_{j,n})
+
\phi_{j+1,n}(t)A_\star(t_{j+1,n}).
\end{align*}
Consequently,
\begin{align*}
\big\lVert
A^{(n)}(t)-A_\star(t)
\big\rVert_{\rm op}
&=
\big\lVert
\phi_{j,n}(t)
\big\{
A_\star(t_{j,n})-A_\star(t)
\big\}
+
\phi_{j+1,n}(t)
\big\{
A_\star(t_{j+1,n})-A_\star(t)
\big\}
\big\rVert_{\rm op}
\\
&\le
\phi_{j,n}(t)
\big\lVert
A_\star(t_{j,n})-A_\star(t)
\rVert_{\rm op}
+
\phi_{j+1,n}(t)
\lVert
A_\star(t_{j+1,n})-A_\star(t)
\big\rVert_{\rm op}
\\ 
&\le 
\sup_{s,t\in[0,1]:
\lvert s-t\rvert\le h_n}
\lVert
A_\star(s)-A_\star(t)
\rVert_{\rm op}
\longrightarrow 0;
\end{align*}
where the convergence follows because $A_\star$ is uniformly
continuous on the compact interval $[0,1]$ and $h_n\to0$.
Similarly, 
\begin{align*}
\sup_{t\in[0,1]}
\lVert
b^{(n)}(t)-b_\star(t)
\rVert
\le 
\sup_{s,t\in[0,1],
\lvert s-t\rvert\le h_n}
\lVert
b_\star(s)-b_\star(t)
\rVert
\longrightarrow 0.
\end{align*}
Define
\begin{align*}
\tilde v_{n,t}(x)
:=
A^{(n)}(t)x+b^{(n)}(t)+r_n(t,x).
\end{align*}
Then $\tilde v_n\in\cV_n$. Since
$X_t=\mu_tZ+\sigma_tU$ and $\mu,\sigma$ are bounded on $[0,1]$,
\begin{align*}
\int_0^1
\int_{\bR^p}
\lVert x\rVert^2
\d\P_t(x)
\d t
<\infty.
\yestag
\label{eq:FM_path_second_moment_apdx}
\end{align*}
By the projection identity
\eqref{eq:CFM_projection_identity}, it is enough to prove
\begin{align*}
\int_0^1
\lVert
\tilde v_{n,t}-v_t
\rVert_{L^2(\P_t)}^2
\d t
\longrightarrow0.
\end{align*}
The affine terms converge by
\eqref{eq:FM_path_second_moment_apdx} and the two uniform approximation
bounds above. For the residual term, the class constraints imply
\begin{align*}
\lVert r_n(t,x)\rVert
&\le
M_r+L_r\lVert x\rVert.
\end{align*}
Together with the affine-growth assumption on $r_\star$, this gives
\begin{align*}
&\int_0^1
\int_{\bR^p}
\lVert
r_n(t,x)-r_\star(t,x)
\rVert^2
\d\P_t(x)
\d t
\notag\\
&\le
\sup_{t\in[0,1],x\in B_{R_n}}
\lVert
r_n(t,x)-r_\star(t,x)
\rVert^2
+
C
\int_0^1
\int_{B_{R_n}^c}
\big(1+\lVert x\rVert^2\big)
\d\P_t(x)
\d t.
\end{align*}
The first term converges to zero by \eqref{eq:FM_residual_closure}; the second
converges to zero by \eqref{eq:FM_path_second_moment_apdx}. The conclusion
follows.
\end{proof}

Proposition~\ref{prop:FM_template_approximation} shows that if the
constrained residual classes approximate $r_\star$ uniformly on expanding
balls while retaining the fixed norm bounds in
Definition~\ref{def:FM_neural_template}, then this local uniform
approximation condition extends to the approximation required in the
main theorem. Related classical universal-approximation results give uniform approximation on compact sets for unconstrained neural networks \citep{cybenko1989approximation,kidger2020universal}.

Condition~\eqref{eq:FM_residual_closure} is an architecture and
target specific closure requirement. For example, suppose
\begin{align*}
\P_U
    =
{\rm N}(m_U,\Sigma_U),
\qquad
\P
=
{\rm N}(m_Z,\Sigma_Z),
\end{align*}
where $\Sigma_U$ and $\Sigma_Z$ are positive definite. Then
$(X_t,V_t)$ is jointly Gaussian for every $t\in[0,1]$, and hence
$v_t(x)=\E[V_t\mid X_t=x]$ is affine in $x$. More precisely,
\begin{align*}
v_t(x)
&=
A_\star(t)x+b_\star(t),
\\
A_\star(t)
&=
\big(
\dot\mu_t\mu_t\Sigma_Z
+
\dot\sigma_t\sigma_t\Sigma_U
\big)
\big(
\mu_t^2\Sigma_Z
+
\sigma_t^2\Sigma_U
\big)^{-1},
\\
b_\star(t)
&=
\dot\mu_t m_Z+\dot\sigma_t m_U
-
A_\star(t)
\big(
\mu_t m_Z+\sigma_t m_U
\big).
\end{align*}
The coefficient functions $A_\star$ and $b_\star$ are continuous on
$[0,1]$. Thus one may take $r_\star\equiv0$, choose $L_A$ and $M_b$
large enough to contain these coefficients, and verify the approximation
condition using only the piecewise-linear time sieves.

It remains to verify the generalization condition for the objective actually
used in Assumption~\ref{def:FM_training_objective}. For
$\theta\in\Theta_n$, define the augmented-observation loss
\begin{align*}
\ell_\theta(z,t,u)
:=
\big\lVert
\tilde v_{\theta,t}
\big(
\mu_tz+\sigma_tu
\big)
-
\big(
\dot\mu_tz+\dot\sigma_tu
\big)
\big\rVert^2.
\end{align*}
Then we can write
\begin{align*}
\cR^{\rm FM}(\tilde v_\theta)
=
\E\big[
\ell_\theta(Z,T,U)
\big],
\qquad
\hat \cR^{\rm FM}_n(\tilde v_\theta)
=
\frac1n
\sum_{i=1}^n
\ell_\theta(Z_i,T_i,U_i).
\end{align*}
The next proposition gives a simple finite-dimensional sufficient condition
for the corresponding uniform law of large numbers.

\begin{proposition}[Generalization for the sampled flow-matching loss]
\label{prop:FM_template_generalization}
Suppose that $\lVert Z\rVert$ and $\lVert U\rVert$ are sub-Gaussian random variables.
If
\begin{align*}
\frac{
N_n^{\rm par}
\log\big(
2+nN_n^{\rm par}
\big)
}{n}
&\longrightarrow0,
\yestag
\label{eq:FM_parameter_growth_apdx}
\end{align*}
then
\begin{align*}
\sup_{\tilde v\in\cV_n}
\big\lvert
\hat \cR^{\rm FM}_n(\tilde v)-\cR^{\rm FM}(\tilde v)
\big\rvert
&=
o_{\P_\cO}(1).
\end{align*}
\end{proposition}

\begin{proof}
Because $\mu$, $\sigma$, $\dot\mu$, and $\dot\sigma$ are bounded on
$[0,1]$, the affine-growth bound from
Proposition~\ref{prop:FM_template_regularity} gives a constant $C<\infty$ such
that, uniformly over $\theta\in\Theta_n$,
\begin{align*}
\ell_\theta(z,t,u)
\le
C\big(
1+\lVert z\rVert^2+\lVert u\rVert^2
\big).
\yestag
\label{eq:FM_loss_envelope_apdx}
\end{align*}
Combining \eqref{eq:FM_parameter_sensitivity} with
$\lvert\lVert a\rVert^2-\lVert b\rVert^2\rvert
\le(\lVert a\rVert+\lVert b\rVert)\lVert a-b\rVert$ also gives
\begin{align*}
\big\lvert
\ell_\theta(z,t,u)-\ell_{\theta'}(z,t,u)
\big\rvert
&\le
C\big(
1+\lVert z\rVert^2+\lVert u\rVert^2
\big)
\lVert\theta-\theta'\rVert_2.
\yestag
\label{eq:FM_loss_parameter_Lipschitz_apdx}
\end{align*}
By assumption, the common envelope in
\eqref{eq:FM_loss_envelope_apdx} is sub-exponential, uniformly in
$n$ and $\theta$. Hence the centered
losses are uniformly sub-exponential, and the ordinary sub-exponential
Bernstein inequality applies with constant $C'$ independent of $n$ and $\theta$: for every
$x>0$,
\begin{align*}
\P_\cO\Big(
\big\lvert
\hat \cR^{\rm FM}_n(\tilde v_\theta)
-
\cR^{\rm FM}(\tilde v_\theta)
\big\rvert
>
C'\Big\{
\sqrt{\frac{x}{n}}
+
\frac{x}{n}
\Big\}
\Big)
\le
2e^{-x}.
\end{align*}

Let $\eta_n:=n^{-1}$ and let
$\theta_1,\ldots,\theta_{M_n}$ be an $\eta_n$-net of $\Theta_n$ under the
Euclidean metric. By \eqref{eq:FM_parameter_entropy},
\begin{align*}
\log M_n
\le
N_n^{\rm par}
\log\big(
1+C_0n\sqrt{N_n^{\rm par}}
\big)
\lesssim
N_n^{\rm par}
\log\big(
2+nN_n^{\rm par}
\big).
\end{align*}
For each $\theta$, choose a net point $\pi_n(\theta)$ satisfying
$\lVert\theta-\pi_n(\theta)\rVert_2\le n^{-1}$. Then
\eqref{eq:FM_loss_parameter_Lipschitz_apdx} gives
\begin{align*}
&\sup_{\theta\in\Theta_n}
\big\lvert
\big\{
\hat \cR^{\rm FM}_n(\tilde v_\theta)-\cR^{\rm FM}(\tilde v_\theta)
\big\}
-
\big\{
\hat \cR^{\rm FM}_n(\tilde v_{\pi_n(\theta)})
-
\cR^{\rm FM}(\tilde v_{\pi_n(\theta)})
\big\}
\big\rvert
\\
& \le
\sup_{\theta\in\Theta_n}
\big\{
\frac1n
\sum_{i=1}^n
\big\lvert
\ell_\theta(Z_i,T_i,U_i) - \ell_{\pi_n(\theta)}(Z_i,T_i,U_i) 
\big\rvert
+
\E \big[
\big\lvert
\ell_\theta(Z,T,U) - \ell_{\pi_n(\theta)}(Z,T,U) 
\big\rvert
\big]
\big\}
\\
&\le
\frac{C}{n}
\big\{
\frac1n
\sum_{i=1}^n
\big(
1+\lVert Z_i\rVert^2+\lVert U_i\rVert^2
\big)
+
\E\big[
1+\lVert Z\rVert^2+\lVert U\rVert^2
\big]
\big\}
=
O_{\P_\cO}(n^{-1}).
\end{align*}
The preceding Bernstein inequality and a union bound give, for every
$x>0$,
\begin{align*}
&\P_\cO\Big(
\max_{1\le j\le M_n}
\big\lvert \hat \cR^{\rm FM}_n(\tilde v_{\theta_j})
-
\cR^{\rm FM}(\tilde v_{\theta_j}) \big\rvert
>
C'
\Big\{
\sqrt{\frac{x}{n}}
+
\frac{x}{n}
\Big\}
\Big)
\\
&\le
\sum_{j=1}^{M_n}
\P_\cO\Big(
\big\lvert \hat \cR^{\rm FM}_n(\tilde v_{\theta_j})
-
\cR^{\rm FM}(\tilde v_{\theta_j}) \big\rvert
>
C'
\Big\{
\sqrt{\frac{x}{n}}
+
\frac{x}{n}
\Big\}
\Big)
\le
2M_ne^{-x}.
\end{align*}
Taking $x=\log(2M_n)+t$ yields, for every $t>0$,
\begin{align*}
\P_\cO\Big(
\max_{1\le j\le M_n}
\big\lvert \hat \cR^{\rm FM}_n(\tilde v_{\theta_j})
-
\cR^{\rm FM}(\tilde v_{\theta_j}) \big\rvert
>
C'
\Big\{
\sqrt{\frac{\log(2M_n)+t}{n}}
+
\frac{\log(2M_n)+t}{n}
\Big\}
\Big)
\le
e^{-t}.
\end{align*}
Consequently,
\begin{align*}
\max_{1\le j\le M_n}
\big\lvert
\hat \cR^{\rm FM}_n(\tilde v_{\theta_j})
-
\cR^{\rm FM}(\tilde v_{\theta_j})
\big\rvert
=
O_{\P_\cO}\Big(
\sqrt{\frac{\log(2M_n)}{n}}
+
\frac{\log(2M_n)}{n}
\Big).
\end{align*}
The right-hand side is $o_{\P_\cO}(1)$ by
\eqref{eq:FM_parameter_growth_apdx}. Combining the last two displays proves
the result.
\end{proof}

The preceding propositions now combine directly with the main-text training
assumption.

\begin{corollary}[Verification of the flow-matching conditions]
\label{cor:FM_template_verification}
Suppose the class $\cV_n$ is constructed as in
Definition~\ref{def:FM_neural_template}. Assume the conditions of
Propositions~\ref{prop:FM_template_approximation} and
\ref{prop:FM_template_generalization} hold. If the fitted vector field also
satisfies $\varepsilon_n^{\rm FM} = o_{\P_\cO}(1)$,
then Assumption~\ref{def:controlled_neural_FM_class} and
Assumption~\ref{ass:controlled_neural_FM_training} hold. 
\end{corollary}

Together with Assumption~\ref{ass:FM_data} and the remaining conditions of
Theorem~\ref{thm:FM_bootstrap_CWC}, Corollary~\ref{cor:FM_template_verification} yields the conditional weak
convergence asserted in that theorem.

The same argument applies to other parametrizations that provide
comparable flow-regularity, parameter-sensitivity, and covering bounds.

\subsection{Score-based diffusion models}
\label{sec:diffusion-apdx}

This subsection gives primitive conditions under which the
reverse-diffusion regularity requirement in
Assumption~\ref{ass:diffusion_score_consistency}(i) and the path-space KL
bound used in Lemma~\ref{lem:diffusion_TV_consistency} hold. For the latter,
we apply the diffusion KL estimate of
\citet[Lemma~4.4(i) and Remark~4.5]{lacker2023hierarchies}.
Proposition~\ref{prop:diffusion_OU_reverse_path} records the required
properties of the exact reverse process, and
Proposition~\ref{prop:diffusion_admissible_regularity} gives well-posedness
and finite-energy properties for the learned reverse process.
Proposition~\ref{prop:diffusion_reverse_KL_bounds} then establishes the KL
bound. Finally, Corollary~\ref{cor:diffusion_fitted_reverse_KL} specializes
it to the fitted diffusion model of Section~\ref{sec:diffusion}, and
Proposition~\ref{prop:diffusion_Gaussian_plugin} gives a simple Gaussian
plug-in verification of the score-consistency condition.

Since the KL comparison below is between the laws of entire reverse
trajectories, it is convenient to place the exact and learned processes
on a common canonical path space. For fixed $0<\tau<T<\infty$, let
\begin{align*}
\Omega_{\tau,T}
:=
C\big([0,T-\tau];\bR^p\big),
\end{align*}
equipped with the uniform topology and its Borel $\sigma$-field, and
write $\mathsf X_r(\omega) := \omega(r)$ with $0\le r\le T-\tau$
for the canonical coordinate process on $\Omega_{\tau,T}$.
Under the path laws introduced below, $\mathsf X$ represent the exact and learned reverse diffusion, respectively.

We first record the regularity of the early-stopped Ornstein--Uhlenbeck reverse path.

\begin{proposition}[Early-stopped Ornstein-Uhlenbeck reverse path]
\label{prop:diffusion_OU_reverse_path}
Suppose Assumption~\ref{ass:diffusion_data_path} holds. Then the following
statements hold.
\begin{enumerate}[label=(\roman*),itemsep=-.5ex]
\item For every $t>0$, the law $\P_t$ admits a strictly positive
$C^\infty$ Lebesgue density $p_t$, and the map
\begin{align*}
(t,x)
\longmapsto
s_t(x)
=
\nabla\log p_t(x)
\end{align*}
is continuous on $(0,\infty)\times\bR^p$.

\item For every $t>0$,
\begin{align*}
s_t(X_t)
=
\E\big[
\psi_t(X_t,Z)
\mid
X_t
\big]
\end{align*}
almost surely, and for every fixed $0<\tau<T<\infty$,
\begin{align*}
\int_\tau^T
\lVert s_t\rVert_{L^2(\P_t)}^2
\d t
<
\infty.
\end{align*}

\item For every fixed $0<\tau<T<\infty$, the time reversal of the forward
OU process over $[\tau,T]$ induces a probability measure
$\mathbf P_{\tau,T}^{\rm ex}$ on $\Omega_{\tau,T}$ under which the
canonical process is a weak solution of
\begin{align*}
\d \mathsf X_r
=
\big\{
\mathsf X_r
+
2s_{T-r}(\mathsf X_r)
\big\}
\d r
+
\sqrt{2}\d W_r^{\rm ex},
\qquad
\mathsf X_0\sim\P_T.
\yestag
\label{eq:appendix_exact_reverse_OU}
\end{align*}
Moreover,
\begin{align*}
\mathsf X_r
\sim
\P_{T-r},
\qquad
0\le r\le T-\tau.
\yestag
\label{eq:appendix_exact_reverse_marginals}
\end{align*}
In particular, $\mathsf X_{T-\tau}\sim\P_\tau$.
\end{enumerate}
\end{proposition}

\begin{proof}
For every $t>0$, the law $\P_t$ is the convolution of a finite measure with
a nondegenerate Gaussian density. Standard differentiation under the integral
therefore gives a strictly positive $C^\infty$ density $p_t$ and joint
continuity of $s_t$ on $(0,\infty)\times\bR^p$.

Now we verify the conditional-expectation identity. Since $t>0$, differentiation under the integral sign gives
\begin{align*}
\nabla p_t(x)
=
\int_{\mathbb R^p}
\nabla_x p_t(x\mid z)\d \P(z)
=
\int_{\mathbb R^p}
\psi_t(x,z)p_t(x\mid z)\d \P(z).
\end{align*}
Because $p_t(x)>0$, it follows that
\begin{align*}
s_t(x)
=
\nabla\log p_t(x)
=
\frac{
\int_{\mathbb R^p}
\psi_t(x,z)p_t(x\mid z)\d \P(z)
}{
p_t(x)
}.
\end{align*}
The conditional distribution of $Z$ given $X_t=x$ is characterized by
\begin{align*}
\P(Z\in \d z\mid X_t=x)
=
\frac{p_t(x\mid z)}{p_t(x)}\P(\d z).
\end{align*}
Consequently,
\begin{align*}
s_t(x)
=
\E\big[
\psi_t(X_t,Z)
\mid
X_t=x
\big]
\end{align*}
for every $x$, and therefore $s_t(X_t)
= \E[ \psi_t(X_t,Z) \mid X_t]$ almost surely.
Conditional Jensen's inequality gives
\begin{align*}
\lVert s_t\rVert_{L^2(\P_t)}^2
=
\E\big[
\big\lVert
\E\big[
\psi_t(X_t,Z)
\mid
X_t
\big]
\big\rVert^2
\big]
\le
\E\big[
\lVert\psi_t(X_t,Z)\rVert^2
\big]
=
\frac{\E\lVert\xi\rVert^2}{1-e^{-2t}}
=
\frac{p}{1-e^{-2t}},
\end{align*}
Hence its integral is finite
on every interval bounded away from zero.

Finally, fix $0<\tau<T<\infty$ and write
$\sigma_t := \sqrt{1-e^{-2t}}$.
Let $\mu_t$ denote the Lebesgue density of $e^{-t}Z$. Then
$p_t = \mu_t*\phi_{\sigma_t}$ and $\lVert \mu_t\rVert_{L^1(\lambda)}=1$.
Lemma~\ref{lem:L1_Lr_convolution} then gives
\begin{align*}
\lVert p_t\rVert_{L^2(\lambda)}
\le
\lVert\phi_{\sigma_t}\rVert_{L^2(\lambda)},
\qquad
\lVert\nabla p_t\rVert_{L^2(\lambda)}
\le
\lVert\nabla\phi_{\sigma_t}\rVert_{L^2(\lambda)}.
\end{align*}
Since
$\inf_{t\in[\tau,T]}\sigma_t
= \sigma_\tau >0$,
the right-hand sides are uniformly bounded over
$t\in[\tau,T]$. Hence
\begin{align*}
\int_\tau^T
\Big\{
\lVert p_t\rVert_{L^2(\lambda)}^2
+
\lVert\nabla p_t\rVert_{L^2(\lambda)}^2
\Big\}
\d t
<
\infty.
\end{align*}
Thus the density regularity condition in
\citet[Theorem~2.1]{haussmann1986time} holds for the forward OU
process restricted to $[\tau,T]$. Its drift and diffusion
coefficients are globally Lipschitz with linear growth. The
time-reversal theorem therefore applies, and because the diffusion
matrix is $2I_p$, the reversed drift is
\begin{align*}
-(-x)+2\nabla\log p_{T-r}(x)
=
x+2s_{T-r}(x).
\end{align*}
This gives \eqref{eq:appendix_exact_reverse_OU}. The marginal identity
\eqref{eq:appendix_exact_reverse_marginals} follows directly from the
definition of the time-reversed path.
\end{proof}

We next give a simple sufficient condition for the learned reverse equation to
be well posed. The condition is imposed separately on each finite interval
$[\tau,T]$; its controlling functions need not be uniform over $n$.

\begin{definition}[Diffusion-admissible score field]
\label{def:diffusion_admissible_score}
Fix $0<\tau<T<\infty$. A jointly Borel measurable map
\begin{align*}
\tilde s:
[\tau,T]\times\bR^p
\to
\bR^p
\end{align*}
is called diffusion-admissible on $[\tau,T]$ if, for Lebesgue-almost every
$t\in[\tau,T]$, the map $x\mapsto\tilde s_t(x)$ is continuous and there exist
nonnegative functions $L_{\tilde s},M_{\tilde s}\in L^2([\tau,T])$ such that
\begin{align*}
\lVert
\tilde s_t(x)-\tilde s_t(y)
\rVert
\le
L_{\tilde s}(t)
\lVert x-y\rVert,
\qquad
\lVert\tilde s_t(0)\rVert
\le
M_{\tilde s}(t)
\end{align*}
for all $x,y\in\bR^p$ and Lebesgue-almost every $t\in[\tau,T]$.
\end{definition}

The preceding condition implies the linear-growth bound
\begin{align*}
\lVert\tilde s_t(x)\rVert
\le
M_{\tilde s}(t)
+
L_{\tilde s}(t)\lVert x\rVert.
\yestag
\label{eq:diffusion_admissible_linear_growth}
\end{align*}
It permits the spatial Lipschitz and growth envelopes to depend on the fitted
score and hence, in an application below, on the fitting information $\cO$.

\begin{proposition}[Well-posedness and finite pathwise score energy]
\label{prop:diffusion_admissible_regularity}
Suppose Assumption~\ref{ass:diffusion_data_path} holds. Fix
$0<\tau<T<\infty$, and let $\tilde s$ be diffusion-admissible on
$[\tau,T]$. Then the following statements hold.
\begin{enumerate}[label=(\roman*),itemsep=-.5ex]
\item One has
\begin{align*}
\int_\tau^T
\lVert\tilde s_t\rVert_{L^2(\P_t)}^2
\d t
<
\infty,
\qquad
\int_\tau^T
\E\big[
\lVert\psi_t(X_t,Z)\rVert^2
\big]
\d t
<
\infty.
\yestag
\label{eq:diffusion_population_risk_finiteness}
\end{align*}
Consequently, the population denoising-score risks in
\eqref{eq:diffusion_DSM_objectives} are finite for $\tilde s$ and $s$ on
$[\tau,T]$.

\item For every probability measure $\nu$ on $\bR^p$, the reverse equation
\begin{align*}
\hat X_r^{\tilde s,\nu}
=
U_0
+
\int_0^r
\big\{
\hat X_u^{\tilde s,\nu}
+
2\tilde s_{T-u}
\big(
\hat X_u^{\tilde s,\nu}
\big)
\big\}
\d u
+
\sqrt{2}W_r^{\rm rev},
\qquad
U_0\sim\nu,
\yestag
\label{eq:appendix_learned_reverse_general_initial}
\end{align*}
for $0\le r\le T-\tau$, where $U_0$ is independent of
$W^{\rm rev}$, admits a unique non-explosive strong solution. The same
conclusion holds after starting the equation at any deterministic
reverse time from any deterministic initial state.

\item Let $\mathbf Q_{\tau,T}^{\tilde s,\nu}$ denote the law of the solution
of \eqref{eq:appendix_learned_reverse_general_initial} on
$\Omega_{\tau,T}$. Under both $\mathbf P_{\tau,T}^{\rm ex}$ and
$\mathbf Q_{\tau,T}^{\tilde s,\nu}$,
\begin{align*}
\int_0^{T-\tau}
\big\lVert
s_{T-r}(\mathsf X_r)
-
\tilde s_{T-r}(\mathsf X_r)
\big\rVert^2
\d r
<
\infty
\end{align*}
almost surely.
\end{enumerate}
\end{proposition}

\begin{proof}
The OU representation gives
\begin{align*}
\int_{\bR^p}
\lVert x\rVert^2
\d\P_t(x)
=
\E\big[
\lVert X_t\rVert^2
\big]
\le
2\E\lVert Z\rVert^2
+
2p
\end{align*}
uniformly over $t\ge0$. Hence
\eqref{eq:diffusion_admissible_linear_growth}, Cauchy-Schwarz, and
$L_{\tilde s},M_{\tilde s}\in L^2([\tau,T])$ imply the first finiteness
statement in \eqref{eq:diffusion_population_risk_finiteness}. The second
follows from
\begin{align*}
\E\big[
\lVert\psi_t(X_t,Z)\rVert^2
\big]
=
\frac{p}{1-e^{-2t}},
\end{align*}
whose integral is finite on $[\tau,T]$. The squared-loss projection identity
then also gives finiteness of the risk at $s$.

The learned reverse drift
\begin{align*}
b_r^{\tilde s}(x)
:=
x+2\tilde s_{T-r}(x),
\qquad
0\le r\le T-\tau,
\end{align*}
satisfies, for Lebesgue-almost every $r$,
\begin{align*}
\lVert
b_r^{\tilde s}(x)-b_r^{\tilde s}(y)
\rVert
\le
\ell(r)\lVert x-y\rVert,
\qquad
\lVert b_r^{\tilde s}(x)\rVert
\le
m(r)+\ell(r)\lVert x\rVert,
\end{align*}
where
\begin{align*}
\ell(r)
:=
1+2L_{\tilde s}(T-r),
\qquad
m(r)
:=
2M_{\tilde s}(T-r).
\end{align*}
Since $L_{\tilde s},M_{\tilde s}\in L^2([\tau,T])$ and the time
interval is finite,
\begin{align*}
\int_0^{T-\tau}
\big\{
\ell(r)+m(r)
\big\}
\d r
<
\infty.
\end{align*}
Thus the drift is jointly measurable, globally Lipschitz in the
spatial variable with a time-integrable Lipschitz envelope, and has a
time-integrable linear-growth envelope. Since the diffusion
coefficient is the constant matrix $\sqrt{2}I_p$, the standard
existence and pathwise-uniqueness theorem for for SDEs with
time-dependent coefficients
gives a unique strong solution of
\eqref{eq:appendix_learned_reverse_general_initial} on
$[0,T-\tau]$.

For completeness, the solution is non-explosive. To see this, writing
\begin{align*}
Y_r
:=
\sup_{0\le u\le r}
\lVert
\widehat X_u^{\tilde s,\nu}
\rVert,
\end{align*}
the integral equation and the preceding growth bound give
\begin{align*}
Y_r
&\le
\lVert U_0\rVert
+
\sqrt{2}
\sup_{0\le u\le r}
\lVert W_u^{\rm rev}\rVert
+
\int_0^r m(u)\d u
+
\int_0^r \ell(u)Y_u\d u.
\end{align*}
Gronwall's inequality therefore yields
\begin{align*}
Y_{T-\tau}
\le
\Big\{
\lVert U_0\rVert
+
\sqrt{2}
\sup_{0\le r\le T-\tau}
\lVert W_r^{\rm rev}\rVert
+
\int_0^{T-\tau}m(r)\d r
\Big\}
\exp\Big(
\int_0^{T-\tau}\ell(r)\d r
\Big)
<
\infty
\end{align*}
almost surely.
The same argument applies after starting the equation at any
deterministic reverse time from any deterministic initial state.

It remains to verify the pathwise energy statement. Every continuous path on
$[0,T-\tau]$ is bounded. By Proposition~\ref{prop:diffusion_OU_reverse_path},
$(t,x)\mapsto s_t(x)$ is continuous on
$[\tau,T]\times\bR^p$, and hence $s_{T-r}(\mathsf X_r)$ is bounded along each
continuous path. On the other hand,
\eqref{eq:diffusion_admissible_linear_growth} and the boundedness of the path
give
\begin{align*}
\int_0^{T-\tau}
\lVert
\tilde s_{T-r}(\mathsf X_r)
\rVert^2
\d r
<
\infty.
\end{align*}
The asserted finiteness follows from the triangle inequality.
\end{proof}

We can now establish KL-divergence bounds comparing the exact and
learned reverse diffusions. Recall that
$\mathbf P_{\tau,T}^{\rm ex}$ and
$\mathbf Q_{\tau,T}^{\tilde s,\nu}$ are probability measures on the
path space $\Omega_{\tau,T}$: respectively, they are the laws of the
exact and learned reverse paths over the entire interval
$[0,T-\tau]$. 
Write $\Q_{\tau,T}^{\tilde s,\nu}$ as the distribution of $\hat X_{T-\tau}^{\tilde s,\nu}$ for the terminal law of the learned reverse process \eqref{eq:appendix_learned_reverse_general_initial}. The terminal
marginal of the exact reverse process \eqref{eq:appendix_exact_reverse_marginals} is $P_\tau$.

\begin{proposition}[KL-divergence bounds for the exact and learned reverse diffusions]
\label{prop:diffusion_reverse_KL_bounds}
Suppose Assumption~\ref{ass:diffusion_data_path} holds. Fix
$0<\tau<T<\infty$, let $\tilde s$ be diffusion-admissible on $[\tau,T]$,
and let $\nu$ be any probability measure on $\bR^p$ satisfying $D_{\rm KL}(\P_T\|\nu) < \infty$.
Then
\begin{align*}
D_{\rm KL}
\big(
\mathbf P_{\tau,T}^{\rm ex}
\|
\mathbf Q_{\tau,T}^{\tilde s,\nu}
\big)
&\le
D_{\rm KL}(\P_T\|\nu)
+
\int_\tau^T
\lVert
\tilde s_t-s_t
\rVert_{L^2(\P_t)}^2
\d t,
\yestag
\label{eq:diffusion_path_space_KL_bound}
\\
D_{\rm KL}
\big(
\P_\tau
\|
\Q_{\tau,T}^{\tilde s,\nu}
\big)
&\le
D_{\rm KL}(\P_T\|\nu)
+
\int_\tau^T
\lVert
\tilde s_t-s_t
\rVert_{L^2(\P_t)}^2
\d t.
\yestag
\label{eq:diffusion_terminal_KL_bound_general}
\end{align*}
\end{proposition}

\begin{proof}
Under $\mathbf P_{\tau,T}^{\rm ex}$, the canonical process has drift
\begin{align*}
b_r^{\rm ex}(x)
:={}
x+2s_{T-r}(x),
\end{align*}
whereas under $\mathbf Q_{\tau,T}^{\tilde s,\nu}$ it has drift
\begin{align*}
b_r^{\tilde s}(x)
:={}
x+2\tilde s_{T-r}(x).
\end{align*}
Both processes have diffusion coefficient matrix $\sqrt{2}I_p$. By
Proposition~\ref{prop:diffusion_OU_reverse_path}, the exact reverse equation
has a weak solution. By
Proposition~\ref{prop:diffusion_admissible_regularity}, the learned reverse
equation is well posed from every deterministic starting time and state, and
the quadratic drift discrepancy is finite almost surely under both path laws.
We may therefore apply the localized KL-divergence estimate of
\citet[Lemma~4.4(i) and Remark~4.5]{lacker2023hierarchies}. It gives
\begin{align*}
D_{\rm KL}
\big(
\mathbf P_{\tau,T}^{\rm ex}
\|
\mathbf Q_{\tau,T}^{\tilde s,\nu}
\big)
&\le
D_{\rm KL}(\P_T\|\nu)
+
\frac12
\E_{\mathbf P_{\tau,T}^{\rm ex}}
\int_0^{T-\tau}
\big\lVert
(\sqrt{2}I_p)^{-1}
\big\{
 b_r^{\rm ex}(\mathsf X_r)
-
b_r^{\tilde s}(\mathsf X_r)
\big\}
\big\rVert^2
\d r
\\
&=
D_{\rm KL}(\P_T\|\nu)
+
\E_{\mathbf P_{\tau,T}^{\rm ex}}
\int_0^{T-\tau}
\big\lVert
s_{T-r}(\mathsf X_r)
-
\tilde s_{T-r}(\mathsf X_r)
\big\rVert^2
\d r.
\end{align*}
Using \eqref{eq:appendix_exact_reverse_marginals} and the change of variables
$t=T-r$ gives
\begin{align*}
\E_{\mathbf P_{\tau,T}^{\rm ex}}
\int_0^{T-\tau}
\big\lVert
s_{T-r}(\mathsf X_r)
-
\tilde s_{T-r}(\mathsf X_r)
\big\rVert^2
\d r
=
\int_\tau^T
\lVert
\tilde s_t-s_t
\rVert_{L^2(\P_t)}^2
\d t.
\end{align*}
This proves \eqref{eq:diffusion_path_space_KL_bound}. Define the terminal
evaluation map
\begin{align*}
\mathsf e_{T-\tau}
:
\Omega_{\tau,T}
\longrightarrow
\bR^p,
\qquad
\mathsf e_{T-\tau}(\omega)
:=
\omega(T-\tau).
\end{align*}
For any $\omega,\omega'\in\Omega_{\tau,T}$,
\begin{align*}
\big\lVert
\mathsf e_{T-\tau}(\omega)
-
\mathsf e_{T-\tau}(\omega')
\big\rVert
\le
\sup_{0\le r\le T-\tau}
\big\lVert
\omega(r)-\omega'(r)
\big\rVert.
\end{align*}
Hence $\mathsf e_{T-\tau}$ is continuous under the uniform topology on
$\Omega_{\tau,T}$ and is therefore Borel measurable. By
\eqref{eq:appendix_exact_reverse_marginals} and the definition of
$\Q_{\tau,T}^{\tilde s,\nu}$,
\begin{align*}
(\mathsf e_{T-\tau})\#
\mathbf P_{\tau,T}^{\rm ex}
=
\P_\tau,
\qquad
(\mathsf e_{T-\tau})\#
\mathbf Q_{\tau,T}^{\tilde s,\nu}
=
\Q_{\tau,T}^{\tilde s,\nu}.
\end{align*}
Therefore, Lemma~\ref{lem:KL_data_processing} gives
\begin{align*}
D_{\rm KL}
\big(
\P_\tau
\|
\Q_{\tau,T}^{\tilde s,\nu}
\big)
&=
D_{\rm KL}
\big(
(\mathsf e_{T-\tau})\#
\mathbf P_{\tau,T}^{\rm ex}
\|
(\mathsf e_{T-\tau})\#
\mathbf Q_{\tau,T}^{\tilde s,\nu}
\big)
\\
&\le
D_{\rm KL}
\big(
\mathbf P_{\tau,T}^{\rm ex}
\|
\mathbf Q_{\tau,T}^{\tilde s,\nu}
\big)
\le
D_{\rm KL}(\P_T\|\nu)
+
\int_\tau^T
\lVert
\tilde s_t-s_t
\rVert_{L^2(\P_t)}^2
\d t.
\end{align*}
 This proves
\eqref{eq:diffusion_terminal_KL_bound_general}.
\end{proof}

For the Ornstein-Uhlenbeck model, the terminal bound
\eqref{eq:diffusion_terminal_KL_bound_general} has the same form as the
likelihood-weighted KL bound of
\citet[Theorem~1]{song2021maximum}, which is stated under stronger
regularity conditions. Here, the localized argument of
\citet[Lemma~4.4(i) and Remark~4.5]{lacker2023hierarchies} requires only
the diffusion well-posedness and finite-energy conditions verified above,
rather than a Novikov-type exponential-integrability condition.

The next corollary translates Proposition~\ref{prop:diffusion_reverse_KL_bounds}
directly into the notation of Section~\ref{sec:diffusion}.

\begin{corollary}[KL-divergence bound for the fitted reverse diffusion]
\label{cor:diffusion_fitted_reverse_KL}
Suppose Assumptions~\ref{ass:diffusion_data_path} and
\ref{def:diffusion_fit_generator} hold. Assume that, conditionally
on $\cO$, the fitted score $\hat s_n$ is diffusion-admissible on
$[\tau_n,T_n]$ almost surely. Then the learned reverse equation
\eqref{eq:learned_reverse_diffusion} admits a unique non-explosive strong
solution and, realization-wise,
\begin{align*}
D_{\rm KL}(\P_{\tau_n}\|\Q_n)
&\le
D_{\rm KL}(\P_{T_n}\|\gamma_p)
+
\int_{\tau_n}^{T_n}
\lVert
\hat s_{n,t}-s_t
\rVert_{L^2(\P_t)}^2
\d t
\\
&=
D_{\rm KL}(\P_{T_n}\|\gamma_p)
+
\mathcal E_n^{\rm DSM}.
\end{align*}
\end{corollary}

\begin{proof}
Apply Propositions~\ref{prop:diffusion_admissible_regularity} and
\ref{prop:diffusion_reverse_KL_bounds} realization-wise with
\begin{align*}
\tau=\tau_n,
\qquad
T=T_n,
\qquad
\tilde s=\hat s_n,
\qquad
\nu=\gamma_p.
\end{align*}
The Gaussian-mixture bound in the proof of
Lemma~\ref{lem:diffusion_TV_consistency} shows that
$D_{\rm KL}(\P_{T_n}\|\gamma_p)<\infty$ for every $n$. The terminal law
$\Q_{\tau_n,T_n}^{\hat s_n,\gamma_p}$ is precisely $\Q_n$ by
\eqref{eq:learned_reverse_diffusion} and the definition of $\Q_n$.
\end{proof}

We conclude with a simple example, which verifies the output-level
score condition by a structured plug-in estimator. For symmetric matrices $A$ and $B$, we write $A\preceq B$ if
$B-A$ is positive semidefinite.

\begin{proposition}[Gaussian plug-in verification]
\label{prop:diffusion_Gaussian_plugin}
Suppose
\begin{align*}
\P
=
{\rm N}(m,\Sigma),
\qquad
\underline\lambda I_p
\preceq
\Sigma
\preceq
\overline\lambda I_p
\end{align*}
for some fixed constants $0< \underline\lambda \le \overline\lambda<\infty$. Let
$\hat m_n$ and $\hat\Sigma_n$ be $\cO$-measurable estimators satisfying
\begin{align*}
\lVert\hat m_n-m\rVert
+
\lVert\hat\Sigma_n-\Sigma\rVert_{\rm op}
=
o_{\P_\cO}(1),
\qquad
\frac{\underline\lambda}{2}I_p
\preceq
\hat\Sigma_n
\preceq
2\overline\lambda I_p
\end{align*}
almost surely. Define
\begin{align*}
\hat\Sigma_{n,t}
:=
e^{-2t}\hat\Sigma_n
+
(1-e^{-2t})I_p,
\qquad
\hat s_{n,t}(x)
:=
-\hat\Sigma_{n,t}^{-1}
\big(
x-e^{-t}\hat m_n
\big).
\yestag
\label{eq:diffusion_Gaussian_plugin_score}
\end{align*}
Then, conditionally on $\cO$, the fitted score is diffusion-admissible on every
$[\tau_n,T_n]$. Moreover, there exists a constant $C<\infty$, depending only
on $p$, $\underline\lambda$, and $\overline\lambda$, such that
\begin{align*}
\mathcal E_n^{\rm DSM}
&\le
C
\Big
\{
\lVert\hat m_n-m\rVert^2
+
\lVert\hat\Sigma_n-\Sigma\rVert_{\rm op}^2
\Big\}
\yestag
\label{eq:diffusion_Gaussian_score_energy_bound}
\end{align*}
for every $n$. Consequently,
$\mathcal E_n^{\rm DSM}=o_{\P_\cO}(1)$. If $\hat m_n$ and $\hat\Sigma_n$ are the sample
mean and a spectrally truncated sample covariance matrix in fixed dimension,
then $\mathcal E_n^{\rm DSM}=O_{\P_\cO}(n^{-1})$.
\end{proposition}

\begin{proof}
In this case, the OU marginal is
\begin{align*}
\P_t
=
{\rm N} \big(e^{-t}m, \Sigma_t\big),
\end{align*}
where we write $\Sigma_t
:=
e^{-2t}\Sigma
+
(1-e^{-2t})I_p$;
so its exact score is
\begin{align*}
s_t(x)
=
-\Sigma_t^{-1}
\big(
x-e^{-t}m
\big).
\end{align*}
The spectral assumptions imply that the eigenvalues of both $\Sigma_t$ and
$\hat\Sigma_{n,t}$ are bounded above and away from zero uniformly over
$t\ge0$ and $n$. In particular,
\begin{align*}
\sup_{n\ge1}
\sup_{t\ge0}
\lVert\hat\Sigma_{n,t}^{-1}\rVert_{\rm op}
<
\infty.
\end{align*}
Thus \eqref{eq:diffusion_Gaussian_plugin_score} is affine in $x$, with a
uniformly bounded spatial Lipschitz constant. Its value at zero is bounded by
$Ce^{-t}\lVert\hat m_n\rVert$, which is square integrable over every
$[\tau_n,T_n]$. Hence the fitted score is diffusion-admissible.

The resolvent identity gives
\begin{align*}
\hat\Sigma_{n,t}^{-1}
-
\Sigma_t^{-1}
=
\hat\Sigma_{n,t}^{-1}
\big(
\Sigma_t-
\hat\Sigma_{n,t}
\big)
\Sigma_t^{-1}
=
-e^{-2t}
\hat\Sigma_{n,t}^{-1}
\big(
\hat\Sigma_n-
\Sigma
\big)
\Sigma_t^{-1}.
\end{align*}
Consequently,
\begin{align*}
\lVert
\hat\Sigma_{n,t}^{-1}
-
\Sigma_t^{-1}
\rVert_{\rm op}
&\le
Ce^{-2t}
\lVert
\hat\Sigma_n-
\Sigma
\rVert_{\rm op}.
\yestag
\label{eq:diffusion_Gaussian_inverse_bound}
\end{align*}
For $X_t\sim\P_t$,
\begin{align*}
\hat s_{n,t}(X_t)-s_t(X_t)
&=
\big(
\Sigma_t^{-1}
-
\hat\Sigma_{n,t}^{-1}
\big)
\big(
X_t-e^{-t}m
\big)
+
e^{-t}
\hat\Sigma_{n,t}^{-1}
\big(
\hat m_n-m
\big).
\end{align*}
Since
$\E\big[
\lVert X_t-e^{-t}m\rVert^2
\big]
=
\tr(\Sigma_t)
\le C$,
we obtain from \eqref{eq:diffusion_Gaussian_inverse_bound}
\begin{align*}
\lVert
\hat s_{n,t}-s_t
\rVert_{L^2(\P_t)}^2
&\le
Ce^{-4t}
\lVert
\hat\Sigma_n-
\Sigma
\rVert_{\rm op}^2
+
Ce^{-2t}
\lVert
\hat m_n-m
\rVert^2.
\end{align*}
Integrating over $[\tau_n,T_n]$ and using
\begin{align*}
\int_0^\infty e^{-4t}\d t
<
\infty,
\qquad
\int_0^\infty e^{-2t}\d t
<
\infty
\end{align*}
gives \eqref{eq:diffusion_Gaussian_score_energy_bound}. The final assertion
follows from the standard fixed-dimensional rates for the sample mean and
sample covariance matrix.
\end{proof}

\subsection{Wasserstein generative adversarial networks}
\label{sec:WGAN-apdx}

We first explain the connection between the empirical adversarial
criterion in Assumption~\ref{def:WGAN_generator} and Wasserstein
minimum-distance fitting. For a fixed measurable generator
$G:\bR^q\to\bR^p$ such that $G\#\P_U$ has a finite first moment, the
Kantorovich-Rubinstein duality gives
\begin{align*}
W_1
\big(
\bP_n,
G\#\P_U
\big)
&=
\sup_{\varphi\in{\rm Lip}_1^0(\bR^p)}
\big\{
\bP_n[\varphi]
-
\P_U[\varphi\circ G]
\big\}.
\end{align*}
Thus, minimizing the right-hand side over $G$ seeks a generated
distribution that is close to the empirical data distribution in
$W_1$.

This ideal objective is not directly available. The full
unit-Lipschitz class is infinite-dimensional, and the expectation
$\P_U[\varphi\circ G]$ is typically unavailable in closed form when
$\varphi$ and $G$ are neural networks. Replacing the full Lipschitz class
by the neural critic class $\mathcal D_n$ and replacing $\P_U$ by the
empirical latent distribution $\bP_{U,n}$ gives
\begin{align*}
\inf_{G\in\cG_n}
\sup_{D\in\mathcal D_n}
\big\{
\bP_n[D]
-
\bP_{U,n}[D\circ G]
\big\}
&=
\inf_{G\in\cG_n}
\sup_{D\in\mathcal D_n}
\widehat{\mathcal L}_n(G,D),
\end{align*}
which is the empirical saddle-point problem used in
Assumption~\ref{def:WGAN_generator}. The difference between
$\bP_{U,n}$ and $\P_U$ is handled separately through the latent-sample
stability condition \eqref{eq:WGAN_latent_stability}.

The verification proceeds in two parts. First, we give norm-controlled
generator and critic classes that ensure the deterministic Lipschitz
and latent-stability conditions. Second, we combine critic
approximation on the empirical fitting set with approximate critic
maximization, generator minimization, and population approximation to
verify the remaining output-level training conditions.

\begin{definition}[Controlled feedforward WGAN classes]
\label{def:controlled_WGAN_classes}
Let $\sigma:\bR\to\bR$ be an activation function applied coordinatewise and satisfy $\sigma(0) = 0$ and $\lvert \sigma(s)-\sigma(t) \rvert \le \lvert s-t\rvert$ for every $s,t\in\bR$.

Fix positive integers $L_{\rm G}$ and $L_{\rm D}$ and constants
$B_{\rm G},M_{\rm G},M_{\rm D}<\infty$, none of which depends on $n$.
The hidden-layer widths below may depend on $n$.

A generator $G_\theta:\bR^q\to\bR^p$ is defined recursively by
\begin{align*}
h_0^{\rm G}(u)
&:=
u,
\\
h_\ell^{\rm G}(u)
&:=
\sigma\big(
W_\ell^{\rm G}h_{\ell-1}^{\rm G}(u)
+
b_\ell^{\rm G}
\big),
\quad
\ell\in\zahl{L_{\rm G}},
\\
G_\theta(u)
&:=
W_{L_{\rm G}+1}^{\rm G}
h_{L_{\rm G}}^{\rm G}(u)
+
b_{L_{\rm G}+1}^{\rm G}.
\end{align*}
All matrix and vector dimensions are compatible and
\begin{align*}
\max_{1\le\ell\le L_{\rm G}+1}
\big\lVert
W_\ell^{\rm G}
\big\rVert_{\rm op}
\le
B_{\rm G},
\quad
\max_{1\le\ell\le L_{\rm G}+1}
\big\lVert
b_\ell^{\rm G}
\big\rVert
\le
M_{\rm G}.
\end{align*}
Let $\cG_n$ be any deterministic class of such generator networks.

For the critic $D_\eta$, define
\begin{align*}
h_0^{\rm D}(x)
&:=
x,
\\
h_\ell^{\rm D}(x)
&:=
\sigma\big(
W_\ell^{\rm D}h_{\ell-1}^{\rm D}(x)
+ b_\ell^{\rm D}
\big),
\quad
\ell\in\zahl{L_{\rm D}},
\\
\bar D_\eta(x)
&:=
a^\top h_{L_{\rm D}}^{\rm D}(x),
\\
D_\eta(x)
&:=
\bar D_\eta(x)
-
\bar D_\eta(0),
\end{align*}
where
\begin{align*}
\max_{1\le\ell\le L_{\rm D}}
\big\lVert
W_\ell^{\rm D}
\big\rVert_{\rm op}
\le 1,
\qquad
\lVert a\rVert
\le 1,
\qquad
\max_{1\le\ell\le L_{\rm D}}
\big\lVert b_\ell^{\rm D} \big\rVert
\le M_{\rm D}.
\end{align*}
Let $\mathcal D_n$ consist of all such anchored critics, their
negatives, and the zero function.
\end{definition}

The ReLU and hyperbolic tangent activations are simple examples of
activations satisfying the Lipschitz and anchoring conditions in
Definition~\ref{def:controlled_WGAN_classes}. The norm constraints above
are exact deterministic constraints defining the theoretical model
class; practical normalization or regularization schemes need not impose
them exactly.

\begin{proposition}[Regularity of the controlled WGAN classes]
\label{prop:controlled_WGAN_regularity}
Under Definition~\ref{def:controlled_WGAN_classes}, define $L_{\rm G}^\star
:=
B_{\rm G}^{L_{\rm G}+1}$ 
and 
$M_{\rm G}^\star
:=
M_{\rm G}
\sum_{j=0}^{L_{\rm G}}
B_{\rm G}^{j}$.
Then every $G\in\cG_n$ satisfies
\begin{align*}
{\rm Lip}(G)
&\le
L_{\rm G}^\star,
\yestag
\label{eq:controlled_WGAN_generator_Lip}
\\
\lVert G(u)\rVert
&\le
M_{\rm G}^\star
+
L_{\rm G}^\star
\lVert u\rVert,
\quad
u\in\bR^q.
\yestag
\label{eq:controlled_WGAN_generator_growth}
\end{align*}
Moreover,
\begin{align*}
\mathcal D_n
\subseteq
{\rm Lip}_1^0(\bR^p),
\qquad
0\in\mathcal D_n,
\qquad
D\in\mathcal D_n
\text{ implies }
-D\in\mathcal D_n.
\yestag
\label{eq:controlled_WGAN_critic_properties}
\end{align*}
If $\P_U[\lVert U\rVert]<\infty$ and
$\hat G_n\in\cG_n$, then $\hat G_n\#\P_U$ has a finite first moment and
\begin{align*}
{\rm Lip}(\hat G_n)
W_1
\big(
\bP_{U,n},
\P_U
\big)
\longrightarrow
0
\end{align*}
almost surely. Consequently,
Assumption~\ref{ass:WGAN_model}(i) and the latent-sample stability
condition in Assumption~\ref{ass:WGAN_model}(iv) hold.
\end{proposition}

\begin{proof}
Since $\sigma$ is applied coordinatewise and is $1$-Lipschitz,
\begin{align*}
\big\lVert
\sigma(x)-\sigma(y)
\big\rVert
\le
\lVert x-y\rVert
\end{align*}
for vectors $x$ and $y$ of the same dimension. Therefore, for every $1\le\ell\le L_{\rm G}$,
\begin{align*}
\big\lVert
h_\ell^{\rm G}(u)
-
h_\ell^{\rm G}(v)
\big\rVert
&\le
\big\lVert
W_\ell^{\rm G}
\big\rVert_{\rm op}
\big\lVert
h_{\ell-1}^{\rm G}(u)
-
h_{\ell-1}^{\rm G}(v)
\big\rVert.
\end{align*}
Iterating this inequality over the hidden layers and then applying the
operator-norm bound to the output layer gives
\begin{align*}
\lVert
G(u)-G(v)
\rVert
&\le
\big\lVert
W_{L_{\rm G}+1}^{\rm G}
\big\rVert_{\rm op}
\big\lVert
h_{L_{\rm G}}^{\rm G}(u)
-
h_{L_{\rm G}}^{\rm G}(v)
\big\rVert
\\
&\le
\big(
\prod_{\ell=1}^{L_{\rm G}+1}
\big\lVert
W_\ell^{\rm G}
\big\rVert_{\rm op}
\big)
\lVert u-v\rVert
\le
B_{\rm G}^{L_{\rm G}+1}
\lVert u-v\rVert,
\end{align*}
which proves \eqref{eq:controlled_WGAN_generator_Lip}.

Because $\sigma(0)=0$,
\begin{align*}
\big\lVert
h_\ell^{\rm G}(0)
\big\rVert
&\le
B_{\rm G}
\big\lVert
h_{\ell-1}^{\rm G}(0)
\big\rVert
+
M_{\rm G}.
\end{align*}
Starting from $h_0^{\rm G}(0)=0$ and iterating this bound yields
\begin{align*}
\lVert G(0)\rVert
&\le
M_{\rm G}
\sum_{j=0}^{L_{\rm G}}
B_{\rm G}^{j}
=
M_{\rm G}^\star.
\end{align*}
Combining this with
\eqref{eq:controlled_WGAN_generator_Lip} gives
\begin{align*}
\lVert G(u)\rVert
\le
\lVert G(u)-G(0)\rVert
+
\lVert G(0)\rVert
\le
L_{\rm G}^\star
\lVert u\rVert
+
M_{\rm G}^\star,
\end{align*}
which proves \eqref{eq:controlled_WGAN_generator_growth}.

For every unanchored critic $\bar D_\eta$,
\begin{align*}
\big\lvert
\bar D_\eta(x)-\bar D_\eta(y)
\big\rvert
\le
\lVert a\rVert
\big(
\prod_{\ell=1}^{L_{\rm D}}
\big\lVert
W_\ell^{\rm D}
\big\rVert_{\rm op}
\big)
\lVert x-y\rVert
\le
\lVert x-y\rVert.
\end{align*}
Subtracting $\bar D_\eta(0)$ does not change the Lipschitz constant and
gives $D_\eta(0)=0$. The remaining properties in
\eqref{eq:controlled_WGAN_critic_properties} follow directly from the
definition of $\mathcal D_n$.

If $G\in\cG_n$, then
\begin{align*}
\P_U
\big[
\lVert G(U)\rVert
\big]
&\le
M_{\rm G}^\star
+
L_{\rm G}^\star
\P_U
\big[
\lVert U\rVert
\big]
<
\infty.
\end{align*}
Thus $G\#\P_U$ has a finite first moment. Finally,
Lemma~\ref{lem:empirical_W1_consistency} gives
\begin{align*}
W_1
\big(
\bP_{U,n},
\P_U
\big)
\longrightarrow
0
\end{align*}
almost surely. Since
${\rm Lip}(\hat G_n)\le L_{\rm G}^\star$, the latent-sample stability
conclusion follows.
\end{proof}

The fixed norm bounds in
Definition~\ref{def:controlled_WGAN_classes} are only one convenient
choice. More generally, if a generator class has a deterministic
Lipschitz bound $L_n^{\rm G}$, then it suffices to require
\begin{align*}
L_n^{\rm G}
W_1
\big(
\bP_{U,n},
\P_U
\big)
=
o_{\P_\cO}(1).
\end{align*}
This permits the generator complexity and Lipschitz constant to increase
with $n$.

We next isolate a direct sufficient condition for the critic-class part
of Assumption~\ref{ass:WGAN_model}(ii). For $G\in\cG_n$, define the
restricted empirical critic value
\begin{align*}
\widehat{\mathcal V}_n(G)
:=
\sup_{D\in\mathcal D_n}
\widehat{\mathcal L}_n(G,D).
\yestag
\label{eq:WGAN_restricted_empirical_value}
\end{align*}
For the fitted generator, let
\begin{align*}
K_n^{\rm fit}
:=
\big\{
Z_i,
\hat G_n(U_i):
i\in\zahl{n}
\big\}
\end{align*}
be the finite set on which the empirical adversarial criterion is
evaluated, and define
\begin{align*}
\eta_n^{\rm crit}
:=
\sup_{\varphi\in{\rm Lip}_1^0(\bR^p)}
\inf_{D\in\mathcal D_n}
\max_{x\in K_n^{\rm fit}}
\big\lvert
\varphi(x)-D(x)
\big\rvert.
\yestag
\label{eq:WGAN_training_support_approximation}
\end{align*}

\begin{lemma}[Critic approximation on the empirical fitting set]
\label{lem:WGAN_training_support_critic}
Suppose
$\mathcal D_n\subseteq{\rm Lip}_1^0(\bR^p)$. Then, realization-wise,
\begin{align*}
0
\le
W_1
\big(
\bP_n,
\hat G_n\#\bP_{U,n}
\big)
-
\widehat{\mathcal V}_n(\hat G_n)
\le
2\eta_n^{\rm crit}.
\yestag
\label{eq:WGAN_training_support_gap}
\end{align*}
Consequently, if
$\eta_n^{\rm crit}=o_{\P_\cO}(1)$, then
$W_1\big(\bP_n, \hat G_n\#\bP_{U,n} \big)
-
\widehat{\mathcal V}_n(\hat G_n)
= o_{\P_\cO}(1)$.
\end{lemma}

\begin{proof}
The first inequality follows from
$\mathcal D_n\subseteq{\rm Lip}_1^0(\bR^p)$ and the
Kantorovich-Rubinstein duality.

Fix
$\varphi\in{\rm Lip}_1^0(\bR^p)$ and $\epsilon>0$. By
\eqref{eq:WGAN_training_support_approximation}, there exists
$D_{\varphi,\epsilon}\in\mathcal D_n$ such that
\begin{align*}
\max_{x\in K_n^{\rm fit}}
\big\lvert
\varphi(x)-D_{\varphi,\epsilon}(x)
\big\rvert
\le
\eta_n^{\rm crit}
+
\epsilon.
\end{align*}
It follows that
\begin{align*}
\widehat{\mathcal L}_n
\big(
\hat G_n,\varphi
\big)
&\le
\widehat{\mathcal L}_n
\big(
\hat G_n,D_{\varphi,\epsilon}
\big)
+
\bP_n
\big[
\lvert
\varphi-D_{\varphi,\epsilon}
\rvert
\big]
+
\bP_{U,n}
\big[\lvert
\varphi\circ\hat G_n
-
D_{\varphi,\epsilon}\circ\hat G_n
\rvert\big]
\\
&\le
\widehat{\mathcal V}_n(\hat G_n)
+
2\eta_n^{\rm crit}
+
2\epsilon.
\end{align*}
Taking the supremum over
$\varphi\in{\rm Lip}_1^0(\bR^p)$, applying
Lemma~\ref{lem:Kantorovich_Rubinstein}, and then letting
$\epsilon\downarrow0$ proves
\eqref{eq:WGAN_training_support_gap}.
\end{proof}

The quantity $\eta_n^{\rm crit}$ measures the critic's ability to
approximate unit-Lipschitz test functions on the empirical fitting set
associated with the fitted generator. It does not require uniform approximation over every $G\in\cG_n$, nor does it require either empirical distribution to have compact support in a fixed nonrandom set.

We now combine critic approximation, numerical critic maximization, and
empirical generator minimization into sufficient conditions for the
output-level requirements in
Assumption~\ref{ass:WGAN_model}(ii)-(iii).

\begin{proposition}[Primitive sufficient conditions for a well-trained fitted pair]
\label{prop:WGAN_primitive_verification}
Suppose Assumption~\ref{ass:WGAN_data} holds,
$\P_U[\lVert U\rVert]<\infty$, and
$0 \in \mathcal D_n
\subseteq
{\rm Lip}_1^0(\bR^p)$.
Suppose the fitted pair
$(\hat G_n,\hat D_n)\in\cG_n\times\mathcal D_n$ satisfies the following
conditions.

\begin{enumerate}[label=(\roman*),itemsep=-.5ex]

\item The critic approximation error on the empirical fitting set satisfies
$\eta_n^{\rm crit}=o_{\P_\cO}(1)$.

\item The fitted critic approximately maximizes the restricted empirical
criterion:
\begin{align*}
0
\le
\widehat{\mathcal V}_n(\hat G_n)
-
\widehat{\mathcal L}_n
\big(
\hat G_n,\hat D_n
\big)
=
o_{\P_\cO}(1).
\yestag
\label{eq:WGAN_critic_optimization_apdx}
\end{align*}

\item The fitted generator approximately minimizes the restricted
empirical critic value: for some nonnegative
$r_n=o_{\P_\cO}(1)$,
\begin{align*}
\widehat{\mathcal V}_n(\hat G_n)
\le
\inf_{G\in\cG_n}
\widehat{\mathcal V}_n(G)
+
r_n.
\yestag
\label{eq:WGAN_generator_optimization_apdx}
\end{align*}

\item There exists a deterministic sequence
$G_n^\circ\in\cG_n$ such that
\begin{align*}
W_1
\big(
G_n^\circ\#\P_U,
\P
\big)
\to
0,
\qquad
{\rm Lip}(G_n^\circ)
W_1
\big(
\bP_{U,n},
\P_U
\big)
=
o_{\P_\cO}(1).
\yestag
\label{eq:WGAN_comparator_conditions}
\end{align*}

\end{enumerate}

Then the critic-adequacy condition
\eqref{eq:WGAN_critic_adequacy} and the fitted empirical-value condition
\eqref{eq:WGAN_empirical_value} hold.
\end{proposition}

\begin{proof}
The triangle inequality and the contraction property under
$G_n^\circ$ give
\begin{align*}
&
W_1
\big(
\bP_n,
G_n^\circ\#\bP_{U,n}
\big)
\\
&~~\le
W_1
\big(
\bP_n,
\P
\big)
+
W_1
\big(
\P,
G_n^\circ\#\P_U
\big)
+
W_1
\big(
G_n^\circ\#\P_U,
G_n^\circ\#\bP_{U,n}
\big)
\\
&~~\le
W_1
\big(
\bP_n,
\P
\big)
+
W_1
\big(
\P,
G_n^\circ\#\P_U
\big)
+
{\rm Lip}(G_n^\circ)
W_1
\big(
\P_U,
\bP_{U,n}
\big)
=
o_{\P_\cO}(1),
\end{align*}
where the first term converges to zero almost surely by
Lemma~\ref{lem:empirical_W1_consistency} and the remaining terms vanish
by \eqref{eq:WGAN_comparator_conditions}.

Since
$\mathcal D_n\subseteq{\rm Lip}_1^0(\bR^p)$,
\begin{align*}
\widehat{\mathcal V}_n(G_n^\circ)
&\le
W_1
\big(
\bP_n,
G_n^\circ\#\bP_{U,n}
\big)
=
o_{\P_\cO}(1).
\end{align*}
Consequently,
\eqref{eq:WGAN_generator_optimization_apdx} yields
\begin{align*}
0
\le
\widehat{\mathcal V}_n(\hat G_n)
\le
\widehat{\mathcal V}_n(G_n^\circ)
+
r_n
=
o_{\P_\cO}(1),
\yestag
\label{eq:WGAN_restricted_value_small}
\end{align*}
where the first inequality follows from $0\in\mathcal D_n$.

By
\eqref{eq:WGAN_critic_optimization_apdx} and
\eqref{eq:WGAN_restricted_value_small},
\begin{align*}
\big\lvert
\widehat{\mathcal L}_n
\big(
\hat G_n,\hat D_n
\big)
\big\rvert
\le
\widehat{\mathcal V}_n(\hat G_n)
+
\big\lvert
\widehat{\mathcal V}_n(\hat G_n)
-
\widehat{\mathcal L}_n
\big(
\hat G_n,\hat D_n
\big)
\big\rvert
=
o_{\P_\cO}(1),
\end{align*}
which proves
\eqref{eq:WGAN_empirical_value}.

Finally, Lemma~\ref{lem:WGAN_training_support_critic} and
\eqref{eq:WGAN_critic_optimization_apdx} give
\begin{align*}
&
\sup_{\varphi\in{\rm Lip}_1^0(\bR^p)}
\widehat{\mathcal L}_n
\big(
\hat G_n,\varphi
\big)
-
\widehat{\mathcal L}_n
\big(
\hat G_n,\hat D_n
\big)
\\
&\quad=
\big\{
W_1
\big(
\bP_n,
\hat G_n\#\bP_{U,n}
\big)
-
\widehat{\mathcal V}_n(\hat G_n)
\big\}
+
\big\{
\widehat{\mathcal V}_n(\hat G_n)
-
\widehat{\mathcal L}_n
\big(
\hat G_n,\hat D_n
\big)
\big\}
\\
&\quad\le
2\eta_n^{\rm crit}
+
\big\{
\widehat{\mathcal V}_n(\hat G_n)
-
\widehat{\mathcal L}_n
\big(
\hat G_n,\hat D_n
\big)
\big\}
=
o_{\P_\cO}(1),
\end{align*}
which proves
\eqref{eq:WGAN_critic_adequacy}.
\end{proof}

Combining the controlled architecture with the preceding proposition
gives the following concise verification of the full WGAN model
assumption.

\begin{corollary}[Verification for the controlled WGAN design]
\label{cor:controlled_WGAN_verification}
Suppose Assumption~\ref{ass:WGAN_data} holds and
$\P_U[\lVert U\rVert]<\infty$. Let $\cG_n$ and $\mathcal D_n$ be the
controlled feedforward classes in
Definition~\ref{def:controlled_WGAN_classes}, and suppose the fitted
pair belongs to
$\cG_n\times\mathcal D_n$.
Assume
\begin{align*}
\inf_{G\in\cG_n}
W_1
\big(
G\#\P_U,
\P
\big)
\longrightarrow
0,
\yestag
\label{eq:controlled_WGAN_population_approximation}
\end{align*}
and suppose
\begin{align*}
\eta_n^{\rm crit}
=
o_{\P_\cO}(1),
\qquad
\widehat{\mathcal V}_n(\hat G_n)
-
\widehat{\mathcal L}_n
\big(
\hat G_n,\hat D_n
\big)
=
o_{\P_\cO}(1),
\qquad
\widehat{\mathcal V}_n(\hat G_n)
\le
\inf_{G\in\cG_n}
\widehat{\mathcal V}_n(G)
+
o_{\P_\cO}(1).
\end{align*}
Then Assumption~\ref{ass:WGAN_model} holds.
\end{corollary}

\begin{proof}
Proposition~\ref{prop:controlled_WGAN_regularity} verifies
Assumption~\ref{ass:WGAN_model}(i) and (iv).

By
\eqref{eq:controlled_WGAN_population_approximation}, one can choose a
deterministic $G_n^\circ\in\cG_n$ such that
\begin{align*}
W_1
\big(
G_n^\circ\#\P_U,
\P
\big)
\le
\inf_{G\in\cG_n}
W_1
\big(
G\#\P_U,
\P
\big)
+
\frac{1}{n}.
\end{align*}
Moreover,
Proposition~\ref{prop:controlled_WGAN_regularity} gives
${\rm Lip}(G_n^\circ)
\le
L_{\rm G}^\star$,
and hence
\begin{align*}
{\rm Lip}(G_n^\circ)
W_1
\big(
\bP_{U,n},
\P_U
\big)
\longrightarrow
0
\end{align*}
almost surely. Thus
\eqref{eq:WGAN_comparator_conditions} holds. The remaining conditions of
Proposition~\ref{prop:WGAN_primitive_verification} are precisely the
three displayed training conditions in the statement, so
Assumption~\ref{ass:WGAN_model}(ii)-(iii) follows.
\end{proof}

\begin{remark}[Scope of the verification]
\label{rem:WGAN_verification_scope}
The norm-controlled critic architecture verifies the exact admissibility
condition
$\mathcal D_n\subseteq{\rm Lip}_1^0(\bR^p)$, while
$\eta_n^{\rm crit}=o_{\P_\cO}(1)$ is the remaining critic-richness
condition. Operator-norm control alone does not imply that a critic class
can approximate every unit-Lipschitz function well enough to distinguish
the fitted empirical distributions. This approximation property may be
verified using a sufficiently rich growing critic sieve or an
architecture-specific Lipschitz approximation result.

We leave this part at the level of
\eqref{eq:WGAN_training_support_approximation}, rather than imposing a
particular Lipschitz-universal architecture, because the bootstrap-process
result only uses the output-level conditions in
Assumption~\ref{ass:WGAN_model}. Similarly, the approximate maximization
and minimization conditions concern the numerical training procedure and
are not consequences of the architectural norm bounds alone.

The controlled feedforward design parallels the fixed-depth,
operator-norm-controlled generator construction used by
\citet{tran2026generative}. Their compact-support argument controls the
empirical-to-population discrepancy through bounded-Lipschitz uniform
laws. Here the target and latent distributions may be unbounded: finite
first moments and the contraction argument in
Lemma~\ref{lem:WGAN_training_implies_W1} replace the compact-support
requirement.
\end{remark}

\section{Auxiliary lemmas}

\begin{lemma}[Theorem 2.3.2 of \cite{durrett2019probability}]
\label{lemma:subseq_prob_cvg}

Let $V_1,V_2,...$ be a sequence of $\bR^r$-valued random variables. The sequence converges in probability to some random variable $V$ if and only if for each subsequence $n_k$ of $n$, there is a further subsequence $n_{k_\ell}$ that converges almost surely to $V$.

\end{lemma}

\begin{lemma}[Lévy's inequality, Proposition A.1.2 in \cite{MR1385671}]
\label{lemma:LevyIneq}

Let $X_1, \ldots, X_n$ be independent symmetric stochastic processes indexed by $T$; for any $k \in \bN$, we write $S_k := \sum_{i=1}^k X_i$. then for any $\lambda>0$ we have 
\begin{align*}
    \P\Big(\max_{1\le k \le n} \sup_{t \in T} \lvert S_k(t) \rvert > \lambda \Big) 
    \le 
    2\P\Big(\sup_{t \in T} \lvert S_n(t) \rvert > \lambda \Big).
\end{align*}
\end{lemma}

\begin{lemma}[Contraction principle, Proposition A.1.10 in \cite{MR1385671}]
\label{lemma:ContracIneq}
Let $X_1, \ldots, X_n$ be $T$-indexed stochastic processes, $\gamma_1, \ldots, \gamma_n$ be real valued random variables such that $0\le \gamma_i \le 1$. Let $\xi_1, \ldots, \xi_n$ be independent zero-mean real random variables independent of $(X_1, \ldots, X_n, \gamma_1, \ldots, \gamma_n)$, then
\begin{align*}
    \E\Big[ \sup_{t \in T} \Big\lvert \sum_{i=1}^n \xi_i \gamma_i X_i(t) \Big\rvert \Big]
    \le 
    \E\Big[ \sup_{t \in T} \Big\lvert \sum_{i=1}^n \xi_i X_i(t) \Big\rvert \Big].
\end{align*}  
\end{lemma}

We record a comparison inequality for centered empirical processes.

\begin{lemma}[Centered empirical-process comparison]
\label{lem:centered_empirical_process_comparison}
Let $\mu$ and $\nu$ be probability measures satisfying $\mu\le D\nu$ for
some $D<\infty$, and let $\cA$ be a function class with an envelope bounded
by $A_0<\infty$. Then, for every $m\ge1$,
\begin{align*}
\E_\mu\big[
\big\lVert\bG_m^\mu\big\rVert_{\cA}
\big]
\le
K_D
\E_\nu\big[
\big\lVert\bG_m^\nu\big\rVert_{\cA}
\big]
+
4A_0\sqrt m
\exp\Big(
-\frac m4
\Big),
\end{align*}
with the constant $K_D=8\lceil2D\rceil$.
\end{lemma}
\begin{proof}
For any probability measure $\lambda$, let
$X_1,\ldots,X_m,X_1',\ldots,X_m'$ be two independent i.i.d. samples from
$\lambda$, and write
\begin{align*}
\gamma_m(\lambda,\cA)
:=
\E_\lambda\big[
\big\lVert\bG_m^\lambda\big\rVert_{\cA}
\big],
\qquad
\Delta_m(\lambda,\cA)
:=
\E\Big[
\Big\lVert
\frac{1}{\sqrt m}
\sum_{i=1}^m
\{ a(X_i)-a(X_i')\}
\Big\rVert_{\cA}
\Big].
\end{align*}
Standard result gives $\gamma_m(\lambda,\cA)
\le
\Delta_m(\lambda,\cA)
\le
2\gamma_m(\lambda,\cA)$,
where the first inequality is obtained by conditioning on $X_1,\ldots,X_m$ and applying Jensen's inequality, and the second inequality is obtained from applying triangle inequality and using the identical distributions of the two samples.

Choose a version $r:=\d\mu/\d\nu$ satisfying $0\le r\le D$. Let $Y_1,Y_2,\ldots$ be i.i.d. from $\nu$, let
$V_1,V_2,\ldots$ be i.i.d. ${\rm Unif}[0,1]$, and let
$\xi_1,\xi_2,\ldots$ be independent Rademacher variables. Assume these three
sequences are mutually independent, and define
\begin{align*}
J_i
:=
\ind\Big(
V_i\le\frac{r(Y_i)}{D}
\Big),
\qquad
T_j
:=
\inf\Big\{ 
t\ge1: \sum_{i=1}^tJ_i=j 
\Big\}.
\end{align*}
For every Borel set $A$, one then has
\begin{align*}
\P\big(Y_i\in A, J_i=1 \big)
=
\int_A
\frac{r(y)}{D}
\d\nu(y)
=
\frac{\mu(A)}{D},
\qquad
\P(J_i=1)=\frac{1}{D}.
\end{align*}
Fix
$0=t_0<t_1<\cdots<t_\ell$ and Borel sets
$A_1,\ldots,A_\ell$. Independence across trials gives
\begin{align*}
&\P\big(
T_1=t_1,\ldots,T_\ell=t_\ell,
Y_{T_1}\in A_1,\ldots,Y_{T_\ell}\in A_\ell
\big)
\\
&=
\prod_{j=1}^\ell
\big\{
\P(J_i=0)^{t_j-t_{j-1}-1}
\P\big(Y_i\in A_j, J_i=1 \big)
\big\}
=
\prod_{j=1}^\ell
\Big\{
\big(1-\frac{1}{D}\big)^{t_j-t_{j-1}-1}
\frac{\mu(A_j)}{D}
\Big\}.
\end{align*}
Summing over $0<t_1<\cdots<t_\ell$ yields
\begin{align*}
\P\big(
Y_{T_1}\in A_1,\ldots,Y_{T_\ell}\in A_\ell
\big)
&=
\sum_{0<t_1<\cdots<t_\ell}
\P\big(
T_1=t_1,\ldots,T_\ell = t_\ell,
Y_{T_1}\in A_1,\ldots,Y_{T_\ell}\in A_\ell
\big)
\\
&=
\prod_{j=1}^\ell\mu(A_j)
\sum_{0<t_1<\cdots<t_\ell}
\prod_{j=1}^\ell
\Big\{
\big(1-\frac{1}{D}\big)^{t_j-t_{j-1}-1}
\frac{1}{D}
\Big\}
\\
&=
\prod_{j=1}^\ell\mu(A_j)
\sum_{0<t_1<\cdots<t_\ell}
\P\big( T_1=t_1,\ldots,T_\ell = t_\ell \big)
\\
&=
\prod_{j=1}^\ell\mu(A_j)
\P(T_\ell<\infty)
=
\prod_{j=1}^\ell\mu(A_j),
\end{align*}
where the last equality is given by the second Borel-Cantelli Lemma applied to independent Bernoulli trials.

For $a\in\cA$, write $b_a:=a-\nu[a]$, and let $X_1,\ldots,X_m,X_1',\ldots,X_m'$ denote the two independent $\mu$-samples
appearing in $\Delta_m(\mu,\cA)$. Since each paired difference is symmetric, the rejection construction gives
\begin{align*}
\Delta_m(\mu,\cA)
&=
\E_{\mu,\xi}
\Big[ \Big\lVert
\frac{1}{\sqrt m}
\sum_{j=1}^m
\xi_j
\{a(X_j)-a(X_j')\}
\Big\rVert_{\cA} \Big]
\\
&\le
2
\E_{\mu,\xi}\Big[
\Big\lVert
\frac{1}{\sqrt m}
\sum_{j=1}^m
\xi_jb_a(X_j)
\Big\rVert_{\cA}
\Big]
\\
&=
2\E_{\nu,V,\xi}\Big[
\Big\lVert
\frac{1}{\sqrt m}
\sum_{j=1}^m
\xi_{T_j}b_a(Y_{T_j})
\Big\rVert_{\cA}
\Big]
=
2\E_{\nu,V,\xi}\Big[
\Big\lVert
\frac{1}{\sqrt m}
\sum_{i=1}^{T_m}
\xi_iJ_i b_a(Y_i)
\Big\rVert_{\cA}
\Big]
\end{align*}
Because $\lvert b_a\rvert\le2A_0$ uniformly over $a\in\cA$, the stopped sum is bounded by $2A_0\sqrt m$ on
$\{T_m>M\}$. On $\{T_m\le M\}$, it is bounded by the largest partial sum up
to time $M$. Therefore,
\begin{align*}
\E_{\nu,V,\xi}\Big[
\Big\lVert
\frac{1}{\sqrt m}
\sum_{i=1}^{T_m}
\xi_iJ_i b_a(Y_i)
\Big\rVert_{\cA}
\Big]
\le
\E_{\nu,V,\xi}\Big[
\max_{1\le \ell\le M}
\Big\lVert
\frac{1}{\sqrt m}
\sum_{i=1}^{\ell}
\xi_iJ_i b_a(Y_i)
\Big\rVert_{\cA}
\Big]
+
2A_0\sqrt m\P(T_m>M).
\end{align*}
Take $k=\lceil2D\rceil$ and $M := k m$. For the second term here, Chernoff bound gives
\begin{align*}
\P(T_m>M)
=
\P\Big(
\sum_{i=1}^M J_i<m
\Big)
=
\P\big(
{\rm Bin}(M,{1}/{D})<m
\big)
\le
\exp\Big(
\frac{-m}{4}
\Big).
\end{align*}
For the first term, we have 
\begin{align*}
\E_{\nu,V,\xi}\Big[
\max_{1\le \ell\le M}
\Big\lVert
\frac{1}{\sqrt m}
\sum_{i=1}^{\ell}
\xi_iJ_i b_a(Y_i)
\Big\rVert_{\cA}
\Big]
&\le
2
\E_{\nu,V,\xi}\Big[
\Big\lVert
\frac{1}{\sqrt m}
\sum_{i=1}^{M}
\xi_iJ_i b_a(Y_i)
\Big\rVert_{\cA}
\Big]
\\
&\le
2\sqrt{k}
\E_{\nu,\xi}\Big[
\Big\lVert
\frac{1}{\sqrt M}
\sum_{i=1}^{M}
\xi_i
 b_a(Y_i)
\Big\rVert_{\cA}
\Big];
\end{align*}
where the first inequality is a result of applying Lemma~\ref{lemma:LevyIneq} and integrating both sides, while the second inequality is implied by Lemma~\ref{lemma:ContracIneq}. We then get
\begin{align*}
\E_{\nu,V,\xi}\Big[
\Big\lVert
\frac{1}{\sqrt m}
\sum_{i=1}^{T_m}
\xi_iJ_i b_a(Y_i)
\Big\rVert_{\cA}
\Big]
\le
2\sqrt{k}
\E_{\nu,\xi}\Big[
\Big\lVert
\frac{1}{\sqrt M}
\sum_{i=1}^{M}
\xi_i
\{a(Y_i)-\nu[a]\}
\Big\rVert_{\cA}
\Big]
+
2A_0\sqrt m
\exp\Big(\frac{-m}{4}\Big).
\end{align*}
Let $Y_1',\ldots,Y_M'$ be an independent i.i.d. sample from $\nu$.
Conditional Jensen's inequality and the triangle inequality gives
\begin{align*}
\E_{\nu,\xi}\Big[
\Big\lVert
\frac{1}{\sqrt M}
\sum_{i=1}^{M}
\xi_i
\{ a(Y_i)-\nu[a]\}
\Big\rVert_{\cA}
\Big]
\le
\E\Big[
\Big\lVert
\frac{1}{\sqrt M}
\sum_{i=1}^{M}
\xi_i
\{a(Y_i)-a(Y_i')\}
\Big\rVert_{\cA}
\Big]
\le
\sqrt k \Delta_m(\nu,\cA).
\end{align*}
Combining the preceeding bounds we obtain
\begin{align*}
\gamma_m(\mu,\cA)
&\le
4k\Delta_m(\nu,\cA)
+
4A_0\sqrt m
\exp\Big(
-\frac m4
\Big)
\\
&\le
8k\gamma_m(\nu,\cA)
+
4A_0\sqrt m
\exp\Big(
-\frac m4
\Big).
\end{align*}
Since $k=\lceil2D\rceil$, this is the claimed bound.
\end{proof}

For a probability measure $\mu$, define its local mean modulus by
\begin{align*}
\omega_m(\mu,\delta)
:=
\E_\mu\big[
\big\lVert\bG_m^\mu\big\rVert_{\cF_\delta}\big].
\end{align*}
For $\mu=\Q_n$, this is a random quantity through its dependence on $\cO$.
The next lemma converts the probability formulation in \eqref{eq:CAE} into
the local mean formulation needed for the previous comparison argument.

\begin{lemma}[Conditional equicontinuity and local mean moduli]
\label{lem:CAE_local_mean_equivalence}
The conditional asymptotic equicontinuity condition \eqref{eq:CAE} holds if
and only if, for every $a>0$,
\begin{align*}
\lim_{\delta\downarrow0}
\limsup_{n\to\infty}
\P_\cO\big(
\omega_n(\Q_n,\delta)>a
\big)
=0.
\yestag
\label{eq:conditional_local_mean_equicontinuity}
\end{align*}
\end{lemma}

\begin{proof}
Write
\begin{align*}
Z_{n,\delta}
:=
\sup_{\rho_\P(f,g)<\delta}
\big\lvert
\tilde\bG_n[f]-\tilde\bG_n[g]
\big\rvert.
\end{align*}
Conditionally on $\cO$, its mean is
$\omega_n(\Q_n,\delta)$.  If
\eqref{eq:conditional_local_mean_equicontinuity} holds, conditional Markov's
inequality gives, for every $\epsilon,\eta>0$,
\begin{align*}
\P_\cO\big(
\P_{\tilde U}\big(
Z_{n,\delta}>\epsilon
\mid \cO
\big)>\eta
\big)
\le
\P_\cO\big(
\omega_n(\Q_n,\delta)>\epsilon\eta
\big),
\end{align*}
which implies \eqref{eq:CAE}.

Conversely, replacing one synthetic observation changes $Z_{n,\delta}$ by
at most $4B/\sqrt n$.  McDiarmid's lower-tail inequality therefore gives,
conditionally on $\cO$,
\begin{align*}
\P_{\tilde U}\big(
Z_{n,\delta}
\le
\omega_n(\Q_n,\delta)-t
\mid \cO
\big)
\le
\exp\Big(
-\frac{t^2}{8B^2}
\Big).
\yestag
\label{eq:conditional_modulus_lower_tail}
\end{align*}
Fix $a>0$ and let $\kappa_a := 1-\exp(-{a^2}/{32B^2})>0$. On the event $\{\omega_n(\Q_n,\delta)>a\}$,
\eqref{eq:conditional_modulus_lower_tail} implies
\begin{align*}
\P_{\tilde U}\Big(
Z_{n,\delta}>\frac a2
\mid \cO
\Big)
\ge
\kappa_a.
\end{align*}
Applying \eqref{eq:CAE} with $\epsilon=a/2$ and
$\eta=\kappa_a/2$ proves
\eqref{eq:conditional_local_mean_equicontinuity}.
\end{proof}

\begin{lemma}[$W_q$-distance convergence implies weak convergence]
\label{lemma:wq-implies-weak}
Let $(\mathcal X,d)$ be a separable metric space, and let $\Q$ be a tight Borel probability measure on $\mathcal X$. For some $q\ge 1$, let $\Q_n$ be a sequence of random Borel probability measures on $\mathcal X$ such that $W_q(\Q_n,\Q)$ is well-defined and
\begin{align*}
    W_q(\Q_n,\Q)=o_{\P}(1).
\end{align*}
Then, $\Q_n$ converges weakly to $\Q$ in probability. Equivalently,
\begin{align*}
    \Q_n h-\Q h=o_{\P}(1)
\end{align*}
for every bounded continuous function $h:\mathcal X\to\bR$. In particular, the conclusion holds when $\mathcal X=\bR^m$ is equipped with the Euclidean metric.
\end{lemma}
\begin{proof}
Fix a bounded continuous function $h:\mathcal X\to\bR$ and write $M:=\lVert h\rVert_\infty$.
If $M=0$, the conclusion is immediate. Hence suppose $M>0$, we first prove the following deterministic implication: for every $\epsilon>0$, there exists $\eta>0$ such that $\lvert \nu h-\Q h\rvert<\epsilon$ for every Borel probability measure $\nu$ with $W_q(\nu,\Q)<\eta$.

Fix $\epsilon>0$. Since $\Q$ is tight, there exists a compact set $K\subset\mathcal X$ such that $\Q(K^c)<{\epsilon}/{8M}$.
By continuity of $h$ and compactness of $K$, there exists $\delta>0$ such that, for every $x\in K$ and every $y\in\mathcal X$ with $d(x,y)<\delta$,
\begin{align*}
    \lvert h(x)-h(y)\rvert<\frac{\epsilon}{2}.
\end{align*}
Indeed, for each $x\in K$, one can choose $r_x>0$ such that $d(x,y)<r_x$ implies $\lvert h(x)-h(y)\rvert<{\epsilon}/{4}$. Take $x_1,\ldots,x_N\in K$ such that
\begin{align*}
    K\subset \cup_{j=1}^N \{x\in\mathcal X:d(x,x_j)<r_{x_j}/2\},
\end{align*}
and set $\delta:=\min_{1\le j\le N}r_{x_j}/2$. Then if $x\in K$ and $d(x,y)<\delta$, choosing $j$ with $d(x,x_j)<r_{x_j}/2$ gives
\begin{align*}
    d(y,x_j)
    \le d(y,x)+d(x,x_j)
    < \delta+\frac{r_{x_j}}{2}
    \le r_{x_j},
\end{align*}
and therefore
\begin{align*}
    \lvert h(x)-h(y)\rvert
    \le \lvert h(x)-h(x_j)\rvert+\lvert h(y)-h(x_j)\rvert
    < \frac{\epsilon}{2}.
\end{align*}

Now set
\begin{align*}
    \eta:=\delta \Big(\frac{\epsilon}{8M}\Big)^{1/q}.
\end{align*}
Let $\nu$ be any Borel probability measure satisfying $W_q(\nu,\Q)<\eta$. By the definition of $W_q$, there exists a coupling $\pi$ of $\P$ and $\nu$ such that
\begin{align*}
    \Big\{\int_{\mathcal X\times\mathcal X} d(x,y)^q\d\pi(x,y)\Big\}^{1/q}<\eta.
\end{align*}
Let $(X,Y)$ have law $\pi$, so that $X\sim\Q$ and $Y\sim\nu$. Then
\begin{align*}
    \lvert \nu h-\Q h\rvert
    &=\Big\lvert \int_{\mathcal X\times\mathcal X} \{h(y)-h(x)\}\d\pi(x,y)\Big\rvert \\
    &\le \int_{\mathcal X\times\mathcal X} \lvert h(y)-h(x)\rvert\d\pi(x,y) \\
    &\le \frac{\epsilon}{2}
        +2M\pi\big\{x\notin K\text{ or }d(x,y)\ge\delta\big\} \\
    &\le \frac{\epsilon}{2}
        +2M\Q(K^c)
        +2M\pi\big\{d(x,y)\ge\delta\big\}.
\end{align*}
By Markov's inequality,
\begin{align*}
    \pi\big\{d(x,y)\ge\delta\big\}
    \le
    \frac{1}{\delta^q}
    \int_{\mathcal X\times\mathcal X}d(x,y)^q\d\pi(x,y)
    <
    \left(\frac{\eta}{\delta}\right)^q
    =
    \frac{\epsilon}{8M}.
\end{align*}
Consequently,
\begin{align*}
    \lvert \nu h-\Q h\rvert
    <
    \frac{\epsilon}{2}
    +2M\cdot\frac{\epsilon}{8M}
    +2M\cdot\frac{\epsilon}{8M}
    =
    \epsilon.
\end{align*}
This proves the deterministic implication.

Applying this implication with $\nu=\Q_n$, we obtain
\begin{align*}
    \P\big(\lvert \Q_n h-\Q h\rvert>\epsilon\big)
    \le
    \P\big(W_q(\Q_n,\Q)\ge\eta\big)
    \to 0.
\end{align*}
Since $\epsilon>0$ is arbitrary, one obtain the claimed $\Q_n h-\Q h=o_{\P}(1)$.

Finally, if $\mathcal X=\bR^m$ with the Euclidean metric, then every Borel probability measure on $\bR^m$ is tight. Indeed, for $K_R:=\{x\in\bR^m:\lVert x\rVert\le R\}$, we have $K_R\uparrow\bR^m$, and hence $\Q(K_R^c)\downarrow0$. Therefore the preceding argument applies directly on $\bR^m$.
\end{proof}

\begin{lemma}[Portmanteau theorem for a.s. bounded continuous functions]
\label{lemma:portmanteau_as_continuous}
Let $(\mathcal X,d)$ be a metric space, and let $\mu_n$ and $\mu$ be Borel probability measures on $\mathcal X$. Then
\begin{align*}
    \mu_n\rightsquigarrow \mu
\end{align*}
if and only if, for every bounded measurable function $h:\mathcal X\to\bR$ such that $\mu(D_h)=0$,
\begin{align*}
    \int h\d\mu_n\to \int h\d\mu.
\end{align*}
\end{lemma}

\begin{proof}
The reverse direction is immediate, since every bounded continuous function $h$ satisfies $D_h=\varnothing$. It remains to prove the forward direction. Suppose $\mu_n\rightsquigarrow\mu$, meaning that
\begin{align*}
    \int \varphi\d\mu_n\to \int \varphi\d\mu
\end{align*}
for every bounded continuous function $\varphi:\mathcal X\to\bR$. Let $h:\mathcal X\to\bR$ be bounded measurable and satisfy $\mu(D_h)=0$. We show that
\begin{align*}
    \int h\d\mu_n\to \int h\d\mu.
\end{align*}
After replacing $h$ by $h+\lVert h\rVert_\infty$, it suffices to consider the case $0\le h\le M$ for some $M<\infty$. For $t\in[0,M]$, define $A_t:=\{x\in\mathcal X:h(x)>t\}$. We claim that
\begin{align*}
    \partial A_t\subset D_h\cup\{x:h(x)=t\}.
\end{align*}
Indeed, if $x\notin D_h$ and $h(x)\ne t$, then $h$ is continuous at $x$ and either $h(x)>t$ or $h(x)<t$. In both cases, a small neighborhood of $x$ is either contained in $A_t$ or contained in $A_t^c$, so $x\notin\partial A_t$. Since $\mu(D_h)=0$, we have
\begin{align*}
    \mu(\partial A_t)
    \le \mu(D_h)+\mu\{x:h(x)=t\}
    =
    \mu\{x:h(x)=t\}.
\end{align*}
We now note that the set of $t$ satisfying $\mu\{x:h(x)=t\}>0$ is at most countable: for each integer $m\ge1$, define $S_m:=\{t\in\bR:\mu\{x:h(x)=t\}>{1}/{m}\}$.
The sets $\{x:h(x)=t\}$ are pairwise disjoint as $t$ varies. Hence $S_m$ contains at most $m$ elements; otherwise, if $t_1,\ldots,t_{m+1}\in S_m$ are distinct, then
\begin{align*}
1 \ge
\mu\big(\cup_{j=1}^{m+1}\{x:h(x)=t_j\}\big)
= \sum_{j=1}^{m+1}\mu\{x:h(x)=t_j\}
> \frac{m+1}{m} > 1,
\end{align*}
a contradiction. Therefore each $S_m$ is finite. Since
\begin{align*}
\left\{t\in\bR:\mu\{x:h(x)=t\}>0\right\}
= \cup_{m=1}^{\infty}S_m,
\end{align*}
this set is at most countable. Consequently, $\mu(\partial A_t)=0$ for Lebesgue-a.e. $t\in[0,M]$.

By the usual portmanteau theorem in the continuity-set form, which follows from the bounded-continuous-function characterization of weak convergence, for every such $t$ we have
\begin{align*}
    \mu_n(A_t)\to\mu(A_t).
\end{align*}
Since $0\le \mu_n(A_t)\le1$, dominated convergence gives
\begin{align*}
    \int h\d\mu_n
    =
    \int_0^M \mu_n(h>t)\d t
    =
    \int_0^M \mu_n(A_t)\d t
    \to 
    \int h\d\mu
    =
    \int_0^M \mu(A_t)\d t.
\end{align*}
This proves the claim.
\end{proof}

\begin{lemma}[$L^1$-$L^r$ convolution inequality, Theorem~1.2.10 in \cite{grafakos2008classical}]
\label{lem:L1_Lr_convolution}
Let $r\in[1,\infty]$, let $f\in L^r(\lambda)$, and let
$g\in L^1(\lambda)$, where $\lambda$ is Lebesgue measure on $\bR^p$.
Define
\begin{align*}
(g*f)(x)
:=
\int_{\bR^p}
g(y)f(x-y)
\d y.
\end{align*}
Then $g*f$ exists Lebesgue-almost everywhere, belongs to
$L^r(\lambda)$, and satisfies
\begin{align*}
\lVert
g*f
\rVert_{L^r(\lambda)}
\le
\lVert
g
\rVert_{L^1(\lambda)}
\lVert
f
\rVert_{L^r(\lambda)}.
\end{align*}
\end{lemma}

\begin{lemma}[KL divergence contraction under measurable maps]
\label{lem:KL_data_processing}
Let $\cX$ and $\cY$ be Polish spaces, let
$\Phi:\cX\to\cY$ be Borel measurable, and let $\mu$ and
$\nu$ be Borel probability measures on $\cX$. Then
\begin{align*}
D_{\rm KL}
\big(
\Phi\#\mu
\|
\Phi\#\nu
\big)
\le
D_{\rm KL}
\big(
\mu
\|
\nu
\big).
\end{align*}
\end{lemma}

\begin{proof}
By Lemma A.1 in \cite{fischer2014form},
\begin{align*}
D_{\rm KL}
(\Phi\#\mu
\|
\Phi\#\nu)
=
\inf
\big\{
D_{\rm KL}
(\gamma \| \nu)
:
\gamma
\text{ is a Borel probability measure on }\cX
\text{ s.t. }
\Phi\#\gamma=\Phi\#\mu
\big\}.
\end{align*}
Since $\mu$ is a Borel probability measure on $\cX$ satisfying
$\Phi\#\mu=\Phi\#\mu$, it is one of the measures over which the infimum
is taken and we have the desired result. 
\end{proof}

\begin{lemma}[Kantorovich-Rubinstein duality]
\label{lem:Kantorovich_Rubinstein}
Let $\mu$ and $\nu$ be Borel probability measures on $\bR^p$ with finite
first moments, and fix $z_0\in\bR^p$. Define
\begin{align*}
{\rm Lip}_1^{z_0}(\bR^p)
:=
\big\{
\varphi:\bR^p\to\bR:
\varphi(z_0)=0,\
{\rm Lip}(\varphi)\le1
\big\}.
\end{align*}
Then
\begin{align*}
W_1(\mu,\nu)
=
\sup_{{\rm Lip}(\varphi)\le1}
(\mu-\nu)[\varphi] 
=
\sup_{\varphi\in{\rm Lip}_1^{z_0}(\bR^p)}
(\mu-\nu)[\varphi]
=
\sup_{\varphi\in{\rm Lip}_1^{z_0}(\bR^p)}
\big\lvert
(\mu-\nu)[\varphi]
\big\rvert.
\yestag
\label{eq:KR_anchored_absolute}
\end{align*}
Consequently, for every real-valued $M$-Lipschitz function $\psi$, where
$M\ge0$, $\lvert (\mu-\nu)[\psi] \rvert
\le M W_1(\mu,\nu)$.
\end{lemma}

\begin{proof}
The first equality in \eqref{eq:KR_anchored_absolute} is the
Kantorovich-Rubinstein duality theorem; see Theorem~1.14 of \cite{villani2003topics}. We verify the conversion from its
usual signed, unanchored form to the form used in this paper.

For any $1$-Lipschitz function $\varphi$, define $\varphi^0(z) := \varphi(z)-\varphi(z_0)$.
Then $\varphi^0\in{\rm Lip}_1^{z_0}(\bR^p)$ and, because $\mu$ and $\nu$
are probability measures, $(\mu-\nu)[\varphi^0] = (\mu-\nu)[\varphi]$.
Thus anchoring does not change the supremum. 

Moreover, every anchored
$1$-Lipschitz function is integrable under both measures, since $\lvert\varphi(z)\rvert
\le
\lVert z-z_0\rVert$.
The anchored class is symmetric: if
$\varphi\in{\rm Lip}_1^{z_0}(\bR^p)$, then
$-\varphi\in{\rm Lip}_1^{z_0}(\bR^p)$. Hence $\lvert (\mu-\nu)[\varphi] \rvert
=
\max \{(\mu-\nu)[\varphi],
(\mu-\nu)[-\varphi]\}$ for every such $\varphi$.
Taking suprema proves the final equality in
\eqref{eq:KR_anchored_absolute}.

If $M=0$, then $\psi$ is constant and
$(\mu-\nu)[\psi]=0$. If $M>0$, apply
\eqref{eq:KR_anchored_absolute} to $\psi/M$, after subtracting its value at
$z_0$. This gives the second statement.
\end{proof}

\begin{lemma}[Almost-sure Wasserstein consistency of the empirical measure]
\label{lem:empirical_W1_consistency}
Let $Z_1,Z_2,\ldots$ be i.i.d. from a Borel probability measure $\P$ on
$\bR^p$, and suppose $\P[\lVert Z\rVert]<\infty$.
Then $W_1(\bP_n,\P) \to0$ almost surely.
\end{lemma}

\begin{proof}
Empirical measures of i.i.d. observations in a separable metric space converge weakly to their common distribution almost
surely; see \citet[Theorem~3]{varadarajan1958convergence}. Hence $\bP_n \rightsquigarrow \P$ almost surely.
The finite-first-moment assumption and the strong law of large numbers give
\begin{align*}
\bP_n\big[
\lVert z\rVert
\big]
=
\frac1n
\sum_{i=1}^n
\lVert Z_i\rVert
\to
\P\big[
\lVert z\rVert
\big]
\quad\text{almost surely}.
\end{align*}
On $\bR^p$, weak convergence together with convergence of first moments is
equivalent to convergence in $W_1$; see
\cite[Definition~6.8 and Theorem~6.9]{villani2009optimal}. Combining these
two facts proves the claim.
\end{proof}

\begin{lemma}[Lower-dimensional locally Lipschitz images]
\label{lem:lower_dimensional_Lipschitz_image}
Let $q<p$, and let $G:\bR^q\to\bR^p$ be locally Lipschitz. Then
$G(\bR^q)$ is a Borel subset of $\bR^p$ and $\lambda_p\big(G(\bR^q)\big)=0$.
Consequently, for every probability measure $\P_U$ on $\bR^q$, the
pushforward law $\Q_G=G\#\P_U$ is supported on a
$\lambda_p$-null set.
\end{lemma}

\begin{proof}
For $m\ge1$, let
\begin{align*}
C_m
:=
\big\{
u\in\bR^q:
\lVert u\rVert\le m
\big\}.
\end{align*}
We first show that $\lambda_p(G(C_m))=0$.

For every $u\in C_m$, local Lipschitz continuity gives a radius $r_u>0$
and a constant $L_u<\infty$ such that $G$ is $L_u$-Lipschitz on
$B(u,r_u)$. By compactness of $C_m$, there exist
$u_1,\ldots,u_N\in C_m$ such that $C_m \subset \cup_{j=1}^N B(u_j,r_{u_j}/2)$.
Set $A_j := C_m \cap B(u_j,r_{u_j}/2)$,
then the restriction of $G$ to $A_j$ is $L_{u_j}$-Lipschitz.

We use an covering argument. Let $A\subset\bR^q$ be bounded,
and suppose the restriction of $G$ to $A$ is $L$-Lipschitz. For every
sufficiently small $\delta>0$, a regular grid covers $A$ by at most
$C\delta^{-q}$ cubes of side length $\delta$, where $C<\infty$ depends
only on a bounded cube containing $A$ and on $q$. The intersection of $A$
with each grid cube has diameter at most $\sqrt q\delta$, so its image is
contained in a $p$-dimensional ball of radius
$L\sqrt q\delta$. If $v_p$ denotes the volume of the unit ball in
$\bR^p$, the $p$-dimensional Lebesgue outer measure satisfies
\begin{align*}
\lambda_p^*\big(
G(A)
\big)
\le
C\delta^{-q}
v_p
\big(
L\sqrt q\delta
\big)^p
=
Cv_pL^pq^{p/2}\delta^{p-q}.
\end{align*}
Since $p-q>0$, letting $\delta\downarrow0$ gives $\lambda_p^*(G(A))=0$.
Applying this argument to $A_1,\ldots,A_N$ and using finite subadditivity
gives $\lambda_p(G(C_m))=0$.

Local Lipschitz continuity implies continuity, so $G(C_m)$ is compact. In
particular, it is Borel. Moreover, $G(\bR^q) = \cup_{m=1}^\infty
G(C_m)$.
Thus $G(\bR^q)$ is an $F_\sigma$ set and hence is Borel. Countable
subadditivity gives $\lambda_p(G(\bR^q))=0$.
Finally,
\begin{align*}
\Q_G\big(
G(\bR^q)
\big)
=
\P_U\big(
G^{-1}\big(
G(\bR^q)
\big)
\big)
=
\P_U(\bR^q)
=
1.
\end{align*}
This proves the final assertion.
\end{proof}

{
\bibliographystyle{apalike}
\bibliography{AMS}

@article{lin2026consistency,
	Author = {Lin, Ziming and Han, Fang},
	Journal = {Biometrika},
	Number = {1},
	Pages = {asag005},
	Publisher = {Oxford University Press},
	Title = {On the consistency of bootstrap for matching estimators},
	Volume = {113},
	Year = {2026}}

@article{baptista2025approximation,
	Author = {Baptista, Ricardo and Hosseini, Bamdad and Kovachki, Nikola and Marzouk, Youssef and Sagiv, Amir},
	Journal = {Mathematics of Computation},
	Number = {354},
	Pages = {1863--1909},
	Title = {An approximation theory framework for measure-transport sampling algorithms},
	Volume = {94},
	Year = {2025}}

@article{zech2022sparse,
	Author = {Zech, Jakob and Marzouk, Youssef},
	Journal = {Constructive Approximation},
	Number = {3},
	Pages = {919--986},
	Publisher = {Springer},
	Title = {Sparse approximation of triangular transports, part {I}: The finite-dimensional case},
	Volume = {55},
	Year = {2022}}

@article{wang2022minimax,
	Author = {Wang, Sven and Marzouk, Youssef},
	Journal = {arXiv preprint arXiv:2207.10231},
	Title = {On minimax density estimation via measure transport},
	Year = {2022}}

@article{gao2024continuous,
	Author = {Gao, Yuan and Huang, Jian and Jiao, Yuling and Zheng, Shurong},
	Journal = {arXiv preprint arXiv:2404.00551},
	Title = {Convergence of continuous normalizing flows for learning probability distributions},
	Year = {2024}}

@article{benton2024flow,
	Author = {Joe Benton and George Deligiannidis and Arnaud Doucet},
	Issn = {2835-8856},
	Journal = {Transactions on Machine Learning Research},
	Title = {Error Bounds for Flow Matching Methods},
	Url = {https://openreview.net/forum?id=uqQPyWFDhY},
	Year = {2024}}

@article{kumar2026manifold,
	Author = {Kumar, Shivam and Wang, Yixin and Lin, Lizhen},
	Journal = {arXiv preprint arXiv:2602.22486},
	Title = {Flow matching is adaptive to manifold structures},
	Year = {2026}}

@article{kunkel2026lipschitz,
	Author = {Kunkel, Lea},
	Journal = {Journal of Multivariate Analysis},
	Pages = {105657},
	Publisher = {Elsevier},
	Title = {Distribution estimation via Flow Matching with {Lipschitz} guarantees},
	Volume = {215},
	Year = {2026}}

@article{kunkel2025minimax,
	Author = {Kunkel, Lea and Trabs, Mathias},
	Journal = {arXiv preprint arXiv:2504.13336},
	Title = {On the minimax optimality of flow matching through the connection to kernel density estimation},
	Year = {2025}}

@inproceedings{zhou2025flowerror,
	Author = {Zhou, Zhengyu and Liu, Weiwei},
	Booktitle = {International Conference on Machine Learning},
	Title = {An error analysis of flow matching for deep generative modeling},
	Year = {2025}}

@article{liang2021how,
	Author = {Liang, Tengyuan},
	Journal = {Journal of Machine Learning Research},
	Number = {228},
	Pages = {1--41},
	Title = {How well generative adversarial networks learn distributions},
	Volume = {22},
	Year = {2021}}

@inproceedings{tanielian2021approximating,
	Author = {Tanielian, Ugo and Biau, Gerard},
	Booktitle = {International Conference on Artificial Intelligence and Statistics},
	Title = {Approximating {Lipschitz} continuous functions with {GroupSort} neural networks},
	Year = {2021}}

@article{stephanovitch2025generalization,
	Author = {St{\'e}phanovitch, Arthur and Aamari, Eddie and Levrard, Cl{\'e}ment},
	Journal = {arXiv preprint arXiv:2507.04794},
	Title = {Generalization bounds for score-based generative models: a synthetic proof},
	Year = {2025}}

@inproceedings{zhang2024minimax,
	Author = {Kaihong Zhang and Heqi Yin and Feng Liang and Jingbo Liu},
	Booktitle = {International Conference on Machine Learning},
	Title = {Minimax Optimality of Score-based Diffusion Models: Beyond the Density Lower Bound Assumptions},
	Url = {https://openreview.net/forum?id=wTd7dogTsB},
	Year = {2024}}

@inproceedings{benton2024nearly,
	Author = {Benton, Joe and De Bortoli, Valentin and Doucet, Arnaud and Deligiannidis, George},
	Booktitle = {International Conference on Learning Representations},
	Title = {Nearly $ d $-linear convergence bounds for diffusion models via stochastic localization},
	Year = {2024}}

@inproceedings{chen2023improved,
	Author = {Chen, Hongrui and Lee, Holden and Lu, Jianfeng},
	Booktitle = {International Conference on Machine Learning},
	Title = {Improved analysis of score-based generative modeling: User-friendly bounds under minimal smoothness assumptions},
	Year = {2023}}

@inproceedings{lee2023convergence,
	Author = {Lee, Holden and Lu, Jianfeng and Tan, Yixin},
	Booktitle = {International Conference on Algorithmic Learning Theory},
	Title = {Convergence of score-based generative modeling for general data distributions},
	Year = {2023}}

@inproceedings{chen2023sampling,
	Author = {Sitan Chen and Sinho Chewi and Jerry Li and Yuanzhi Li and Adil Salim and Anru Zhang},
	Booktitle = {International Conference on Learning Representations},
	Title = {Sampling is as easy as learning the score: theory for diffusion models with minimal data assumptions},
	Url = {https://openreview.net/forum?id=zyLVMgsZ0U_},
	Year = {2023}}

@inproceedings{fukumizu2025flow,
	Author = {Fukumizu, Kenji and Suzuki, Taiji and Isobe, Noboru and Oko, Kazusato and Koyama, Masanori},
	Booktitle = {International Conference on Learning Representations},
	Title = {Flow matching achieves almost minimax optimal convergence},
	Year = {2025}}

@book{sontag1998mathematical,
	Author = {Sontag, Eduardo D.},
	Edition = {Second},
	Publisher = {Springer-Verlag New York},
	Title = {Mathematical Control Theory},
	Year = {1998}}

@article{albergo2025stochastic,
	Author = {Albergo, Michael and Boffi, Nicholas M and Vanden-Eijnden, Eric},
	Journal = {Journal of Machine Learning Research},
	Number = {209},
	Pages = {1--80},
	Title = {Stochastic interpolants: A unifying framework for flows and diffusions},
	Volume = {26},
	Year = {2025}}

@conference{albergo2022building,
	Author = {Albergo, Michael S and Vanden-Eijnden, Eric},
	Booktitle = {International Conference on Learning Representations},
	Title = {Building normalizing flows with stochastic interpolants},
	Year = {2023}}

@book{pachpatte1998inequalities,
	Author = {Pachpatte, B. G.},
	Publisher = {Academic Press},
	Title = {Inequalities for Differential and Integral Equations},
	Year = {1998}}

@article{dumbgen1993nondifferentiable,
	Author = {D{\"u}mbgen, Lutz},
	Journal = {Probability Theory and Related Fields},
	Number = {1},
	Pages = {125--140},
	Publisher = {Springer},
	Title = {On nondifferentiable functions and the bootstrap},
	Volume = {95},
	Year = {1993}}

@article{fang2019inference,
	Author = {Fang, Zheng and Santos, Andres},
	Journal = {The Review of Economic Studies},
	Number = {1},
	Pages = {377--412},
	Publisher = {Oxford University Press},
	Title = {Inference on directionally differentiable functions},
	Volume = {86},
	Year = {2019}}

@article{el2018can,
	Author = {El Karoui, Noureddine and Purdom, Elizabeth},
	Journal = {Journal of Machine Learning Research},
	Number = {5},
	Pages = {1--66},
	Title = {Can we trust the bootstrap in high-dimensions? The case of linear models},
	Volume = {19},
	Year = {2018}}

@article{lin2026bootstrap,
	Author = {Lin, Ziming and Han, Fang},
	Journal = {arXiv preprint arXiv:2604.17239},
	Title = {Bootstrap consistency for general double/debiased machine learning estimators},
	Year = {2026}}

@article{varadarajan1958convergence,
	Author = {Varadarajan, Veeravalli S},
	Journal = {Sankhy{\=a}: The Indian Journal of Statistics},
	Number = {1/2},
	Pages = {23--26},
	Publisher = {JSTOR},
	Title = {On the convergence of sample probability distributions},
	Volume = {19},
	Year = {1958}}

@article{chernozhukov2016empirical,
	Author = {Chernozhukov, Victor and Chetverikov, Denis and Kato, Kengo},
	Journal = {Stochastic Processes and their Applications},
	Number = {12},
	Pages = {3632--3651},
	Publisher = {Elsevier},
	Title = {Empirical and multiplier bootstraps for suprema of empirical processes of increasing complexity, and related {Gaussian} couplings},
	Volume = {126},
	Year = {2016}}

@techreport{praestgaard1990bootstrap,
	Author = {Praestgaard, J},
	Institution = {Department of Statistics, University of Washington},
	Journal = {Technical Rep},
	Number = {195},
	Publisher = {Citeseer},
	Title = {Bootstrap with general weights and multiplier central limit theorems},
	Year = {1990}}

@article{hall1989smoothing,
	Author = {Hall, Peter and DiCiccio, Thomas J and Romano, Joseph P},
	Journal = {The Annals of Statistics},
	Number = {2},
	Pages = {692--704},
	Publisher = {JSTOR},
	Title = {On smoothing and the bootstrap},
	Volume = {17},
	Year = {1989}}

@article{silverman1987bootstrap,
	Author = {Silverman, B. W. and Young, G. A.},
	Journal = {Biometrika},
	Number = {3},
	Pages = {469--479},
	Publisher = {Oxford University Press},
	Title = {The bootstrap: to smooth or not to smooth?},
	Volume = {74},
	Year = {1987}}

@book{grafakos2008classical,
	Author = {Grafakos, Loukas},
	Edition = {Second},
	Publisher = {Springer},
	Title = {Classical {Fourier} Analysis},
	Year = {2008}}

@article{fischer2014form,
	Author = {Fischer, Markus},
	Journal = {Bernoulli},
	Number = {4},
	Pages = {1765--1801},
	Publisher = {JSTOR},
	Title = {On the form of the large deviation rate function for the empirical measures of weakly interacting systems},
	Volume = {20},
	Year = {2014}}

@article{lacker2023hierarchies,
	Author = {Lacker, Daniel},
	Journal = {Probability and Mathematical Physics},
	Number = {2},
	Pages = {377--432},
	Publisher = {Mathematical Sciences Publishers},
	Title = {Hierarchies, entropy, and quantitative propagation of chaos for mean field diffusions},
	Volume = {4},
	Year = {2023}}

@article{haussmann1986time,
	Author = {Haussmann, Ulrich G and Pardoux, Etienne},
	Journal = {The Annals of Probability},
	Number = {4},
	Pages = {1188--1205},
	Publisher = {JSTOR},
	Title = {Time reversal of diffusions},
	Volume = {14},
	Year = {1986}}

@inproceedings{fu2026approximation,
	Author = {Fu, Guoji and Lee, Wee Sun},
	Booktitle = {Advances in Neural Information Processing Systems 38},
	Title = {Approximation and generalization abilities of score-based neural network generative models for sub-{Gaussian} distributions},
	Year = {2025}}

@inproceedings{han2024neural,
	Author = {Yinbin Han and Meisam Razaviyayn and Renyuan Xu},
	Booktitle = {International Conference on Learning Representations},
	Title = {Neural Network-Based Score Estimation in Diffusion Models: Optimization and Generalization},
	Year = {2024}}

@article{hyvarinen2005estimation,
	Author = {Aapo Hyv{{\"a}}rinen},
	Journal = {Journal of Machine Learning Research},
	Number = {24},
	Pages = {695--709},
	Title = {Estimation of Non-Normalized Statistical Models by Score Matching},
	Url = {http://jmlr.org/papers/v6/hyvarinen05a.html},
	Volume = {6},
	Year = {2005}}

@article{vincent2011connection,
	Author = {Vincent, Pascal},
	Journal = {Neural Computation},
	Number = {7},
	Pages = {1661--1674},
	Publisher = {MIT Press},
	Title = {A connection between score matching and denoising autoencoders},
	Volume = {23},
	Year = {2011}}

@inproceedings{oko2023diffusion,
	Author = {Oko, Kazusato and Akiyama, Shunta and Suzuki, Taiji},
	Booktitle = {International Conference on Machine Learning},
	Title = {Diffusion models are minimax optimal distribution estimators},
	Year = {2023}}

@inproceedings{chen2023score,
	Author = {Chen, Minshuo and Huang, Kaixuan and Zhao, Tuo and Wang, Mengdi},
	Booktitle = {International Conference on Machine Learning},
	Title = {Score approximation, estimation and distribution recovery of diffusion models on low-dimensional data},
	Year = {2023}}

@misc{liang2026diffusion,
	Archiveprefix = {arXiv},
	Author = {Ce Liang and Wei Ma},
	Eprint = {2607.24324},
	Howpublished = {arXiv preprint arXiv:2607.24324},
	Primaryclass = {stat.ME},
	Title = {Diffusion Bootstrap for High-Dimensional Linear Models},
	Url = {https://arxiv.org/abs/2607.24324},
	Year = {2026}}

@inproceedings{song2021maximum,
	Author = {Song, Yang and Durkan, Conor and Murray, Iain and Ermon, Stefano},
	Booktitle = {Advances in Neural Information Processing Systems 34},
	Title = {Maximum Likelihood Training of Score-Based Diffusion Models},
	Year = {2021}}

@inproceedings{song2021scorebased,
	Author = {Yang Song and Jascha Sohl-Dickstein and Diederik P Kingma and Abhishek Kumar and Stefano Ermon and Ben Poole},
	Booktitle = {International Conference on Learning Representations},
	Title = {Score-Based Generative Modeling through Stochastic Differential Equations},
	Url = {https://openreview.net/forum?id=PxTIG12RRHS},
	Year = {2021}}

@book{Dudley_2002,
	Author = {Dudley, R. M.},
	Collection = {Cambridge Studies in Advanced Mathematics},
	Edition = {Second},
	Place = {Cambridge},
	Publisher = {Cambridge University Press},
	Series = {Cambridge Studies in Advanced Mathematics},
	Title = {Real Analysis and Probability},
	Year = {2002}}

@article{stephanovitch2024optimal,
	Author = {Stephanovitch, Arthur and Tanielian, Ugo and Cadre, Benoit and Klutchnikoff, Nicolas and Biau, Gerard},
	Journal = {Bernoulli},
	Number = {4},
	Pages = {2955--2978},
	Publisher = {Bernoulli Society for Mathematical Statistics and Probability},
	Title = {Optimal 1-{Wasserstein} distance for {WGANs}},
	Volume = {30},
	Year = {2024}}

@book{villani2009optimal,
	Author = {Villani, C{\'e}dric},
	Publisher = {Springer},
	Title = {Optimal Transport: Old and New},
	Volume = {338},
	Year = {2009}}

@book{villani2003topics,
	Author = {Villani, C{\'e}dric},
	Publisher = {American Mathematical Society},
	Title = {Topics in Optimal Transportation},
	Volume = {58},
	Year = {2003}}

@inproceedings{gulrajani2017improved,
	Author = {Gulrajani, Ishaan and Ahmed, Faruk and Arjovsky, Martin and Dumoulin, Vincent and Courville, Aaron C},
	Booktitle = {Advances in Neural Information Processing Systems 30},
	Title = {Improved training of {Wasserstein} {GANs}},
	Year = {2017}}

@article{gao2023approximating,
	Author = {Gao, Yihang and Ng, Michael K and Zhou, Mingjie},
	Journal = {SIAM Journal on Mathematics of Data Science},
	Number = {4},
	Pages = {949--976},
	Publisher = {SIAM},
	Title = {Approximating probability distributions by using {Wasserstein} generative adversarial networks},
	Volume = {5},
	Year = {2023}}

@article{stephanovitch2024wasserstein,
	Author = {Stephanovitch, Arthur and Aamari, Eddie and Levrard, Clement},
	Journal = {The Annals of Statistics},
	Number = {5},
	Pages = {2167--2193},
	Publisher = {Institute of Mathematical Statistics},
	Title = {Wasserstein generative adversarial networks are minimax optimal distribution estimators},
	Volume = {52},
	Year = {2024}}

@inproceedings{schreuder2021statistical,
	Author = {Schreuder, Nicolas and Brunel, Victor-Emmanuel and Dalalyan, Arnak},
	Booktitle = {Algorithmic Learning Theory},
	Title = {Statistical guarantees for generative models without domination},
	Year = {2021}}

@article{biau2021some,
	Author = {Biau, G{\'e}rard and Sangnier, Maxime and Tanielian, Ugo},
	Journal = {Journal of Machine Learning Research},
	Number = {119},
	Pages = {1--45},
	Title = {Some theoretical insights into {Wasserstein} {GANs}},
	Volume = {22},
	Year = {2021}}

@inproceedings{arora2017generalization,
	Author = {Arora, Sanjeev and Ge, Rong and Liang, Yingyu and Ma, Tengyu and Zhang, Yi},
	Booktitle = {International Conference on Machine Learning},
	Title = {Generalization and equilibrium in generative adversarial nets ({GANs})},
	Year = {2017}}

@inproceedings{arjovsky2017towards,
	Author = {Arjovsky, Martin and Bottou, Leon},
	Booktitle = {International Conference on Learning Representations},
	Title = {Towards Principled Methods for Training Generative Adversarial Networks},
	Year = {2017}}

@article{gine1990bootstrapping,
	Author = {Gin{\'e}, Evarist and Zinn, Joel},
	Journal = {The Annals of Probability},
	Number = {2},
	Pages = {851--869},
	Publisher = {JSTOR},
	Title = {Bootstrapping general empirical measures},
	Volume = {18},
	Year = {1990}}

@article{huang2022error,
	Author = {Huang, Jian and Jiao, Yuling and Li, Zhen and Liu, Shiao and Wang, Yang and Yang, Yunfei},
	Journal = {Journal of Machine Learning Research},
	Number = {116},
	Pages = {1--43},
	Title = {An error analysis of generative adversarial networks for learning distributions},
	Volume = {23},
	Year = {2022}}

@inproceedings{arjovsky2017wasserstein,
	Author = {Arjovsky, Martin and Chintala, Soumith and Bottou, L{\'e}on},
	Booktitle = {International Conference on Machine Learning},
	Title = {Wasserstein generative adversarial networks},
	Year = {2017}}

@article{cybenko1989approximation,
	Author = {Cybenko, George},
	Journal = {Mathematics of Control, Signals and Systems},
	Number = {4},
	Pages = {303--314},
	Publisher = {Springer},
	Title = {Approximation by superpositions of a sigmoidal function},
	Volume = {2},
	Year = {1989}}

@inproceedings{kidger2020universal,
	Author = {Kidger, Patrick and Lyons, Terry},
	Booktitle = {Conference on Learning Theory},
	Title = {Universal approximation with deep narrow networks},
	Year = {2020}}

@inproceedings{kingma2018glow,
	Author = {Kingma, Durk P and Dhariwal, Prafulla},
	Booktitle = {Advances in Neural Information Processing Systems 31},
	Journal = {Advances in neural information processing systems},
	Title = {Glow: Generative flow with invertible 1x1 convolutions},
	Year = {2018}}

@inproceedings{huang2018neural,
	Author = {Huang, Chin-Wei and Krueger, David and Lacoste, Alexandre and Courville, Aaron},
	Booktitle = {International Conference on Machine Learning},
	Title = {Neural autoregressive flows},
	Year = {2018}}

@inproceedings{papamakarios2017masked,
	Author = {Papamakarios, George and Pavlakou, Theo and Murray, Iain},
	Booktitle = {Advances in Neural Information Processing Systems 30},
	Title = {Masked autoregressive flow for density estimation},
	Year = {2017}}

@inproceedings{dinh2017density,
	Author = {Dinh, Laurent and Sohl-Dickstein, Jascha and Bengio, Samy},
	Booktitle = {International Conference on Learning Representations},
	Title = {Density estimation using {Real NVP}},
	Year = {2017}}

@article{dinh2014nice,
	Author = {Dinh, Laurent and Krueger, David and Bengio, Yoshua},
	Journal = {arXiv preprint arXiv:1410.8516},
	Title = {{NICE}: Non-linear independent components estimation},
	Year = {2014}}

@article{kobyzev2020normalizing,
	Author = {Kobyzev, Ivan and Prince, Simon J. D. and Brubaker, Marcus A},
	Journal = {IEEE Transactions on Pattern Analysis and Machine Intelligence},
	Number = {11},
	Pages = {3964--3979},
	Publisher = {IEEE},
	Title = {Normalizing flows: An introduction and review of current methods},
	Volume = {43},
	Year = {2021}}

@inproceedings{miyato2018spectral,
	Author = {Miyato, Takeru and Kataoka, Toshiki and Koyama, Masanori and Yoshida, Yuichi},
	Booktitle = {International Conference on Learning Representations},
	Journal = {arXiv preprint arXiv:1802.05957},
	Title = {Spectral normalization for generative adversarial networks},
	Year = {2018}}

@inproceedings{grathwohl2018ffjord,
	Author = {Grathwohl, Will and Chen, Ricky T. Q. and Bettencourt, Jesse and Sutskever, Ilya and Duvenaud, David},
	Booktitle = {International Conference on Learning Representations},
	Title = {{FFJORD}: Free-form continuous dynamics for scalable reversible generative models},
	Year = {2019}}

@inproceedings{chen2018neural,
	Author = {Chen, Ricky T. Q. and Rubanova, Yulia and Bettencourt, Jesse and Duvenaud, David K},
	Booktitle = {Advances in Neural Information Processing Systems 31},
	Title = {Neural ordinary differential equations},
	Year = {2018}}

@inproceedings{lipman2022flow,
	Author = {Lipman, Yaron and Chen, Ricky T. Q. and Ben-Hamu, Heli and Nickel, Maximilian and Le, Matt},
	Booktitle = {International Conference on Learning Representations},
	Title = {Flow matching for generative modeling},
	Year = {2023}}

@inproceedings{irons2022triangular,
	Author = {Irons, Nicholas J and Scetbon, Meyer and Pal, Soumik and Harchaoui, Zaid},
	Booktitle = {International Conference on Artificial Intelligence and Statistics},
	Title = {Triangular flows for generative modeling: Statistical consistency, smoothness classes, and fast rates},
	Year = {2022}}

@article{van2011local,
	Author = {van der Vaart, Aad W. and Wellner, Jon A.},
	Journal = {Electronic Journal of Statistics},
	Pages = {192--203},
	Title = {A local maximal inequality under uniform entropy},
	Volume = {5},
	Year = {2011}}

@article{tran2026generative,
	Author = {Tran, Leon and Ye, Ting and Ding, Peng and Han, Fang},
	Journal = {arXiv preprint arXiv:2602.17052},
	Title = {Generative modeling for the bootstrap},
	Year = {2026}}

@article{efron1979bootstrap,
	Author = {Efron, B},
	Journal = {The Annals of Statistics},
	Number = {1},
	Pages = {1--26},
	Title = {Bootstrap Methods: Another Look at the Jackknife},
	Volume = {7},
	Year = {1979}}

@book{durrett2019probability,
	Author = {Durrett, Rick},
	Publisher = {Cambridge University Press},
	Title = {Probability: Theory and Examples (5th Edition)},
	Year = {2019}}

@book{kosorok2008introduction,
	Author = {Kosorok, Michael R},
	Publisher = {Springer},
	Title = {Introduction to Empirical Processes and Semiparametric Inference},
	Year = {2008}}

@article{mason1992rank,
	Author = {Mason, David M and Newton, Michael A},
	Journal = {The Annals of Statistics},
	Number = {3},
	Pages = {1611--1624},
	Publisher = {JSTOR},
	Title = {A rank statistics approach to the consistency of a general bootstrap},
	Volume = {20},
	Year = {1992}}

@article{praestgaard1993exchangeably,
	Author = {Praestgaard, Jens and Wellner, Jon A},
	Journal = {The Annals of Probability},
	Number = {4},
	Pages = {2053--2086},
	Title = {Exchangeably Weighted Bootstraps of the General Empirical Process},
	Volume = {21},
	Year = {1993}}

@article{lin2023failure,
	Author = {Lin, Zhexiao and Han, Fang},
	Journal = {Biometrika},
	Number = {3},
	Pages = {1063--1070},
	Title = {On the failure of the bootstrap for {C}hatterjee's rank correlation},
	Volume = {111},
	Year = {2024}}

@article{chernozhukov2018double,
	Author = {Chernozhukov, Victor and Chetverikov, Denis and Demirer, Mert and Duflo, Esther and Hansen, Christian and Newey, Whitney and Robins, James},
	Journal = {The Econometrics Journal},
	Number = {1},
	Pages = {C1--C68},
	Publisher = {Oxford University Press Oxford, UK},
	Title = {Double/debiased machine learning for treatment and structural parameters},
	Volume = {21},
	Year = {2018}}

@book{MR1385671,
	Author = {van der Vaart, Aad W. and Wellner, Jon A.},
	Publisher = {Springer},
	Title = {Weak Convergence and Empirical Processes: with Applications to Statistics},
	Year = {1996}}
}

\end{document}